\documentclass[10pt]{article}
\usepackage{amssymb}
\usepackage{amsmath}
\usepackage{amsfonts}
\usepackage{textcomp}
\usepackage{amsthm}
\usepackage{bookmark}
\usepackage{appendix}
\usepackage{dsfont}
\usepackage{changes}%\usepackage[final]{changes}  %\listofchanges
\usepackage{xcolor}
\usepackage{verbatim}
\usepackage{enumitem}
\usepackage{url}
\usepackage{hyperref}
\usepackage{cases}
\usepackage[T1]{fontenc}
\date{}

\newtheorem{Theorem}{Theorem}[section]
\newtheorem{Lemma}[Theorem]{Lemma}
\newtheorem{Remark}[Theorem]{Remark}

\newtheorem{Proposition}[Theorem]{Proposition}
\newtheorem*{Theorem*}{Theorem}

\numberwithin{equation}{section} \theoremstyle{plain}

\renewcommand{\theequation}{\arabic{section}.\arabic{equation}}

\allowdisplaybreaks[4]

\def\R{{\textbf{R}}}
\def\H{\mathcal{H}}

\def\p{\partial}

\def\f{\frac}
\def\R{\mathbb R}
\def\N{\mathbb N}

\def\i{\mathrm{i}}

\def\Om{\Omega}

\def\n{\nabla}

\def\E{\mathcal E}
\title{Ground States of a Transversely Confined Polaron: Uniqueness and an Explicit Limiting Profile}

\author{Shuang Wu\thanks{School of Mathematics and Statistics,  and Hubei Key Laboratory of Mathematical Sciences, Central China Normal University, Wuhan 430079, P. R. China. Email: swu@mails.ccnu.edu.cn.},
and  Hongfei Zhang\thanks{School of Mathematics and Statistics, and  Hubei Key Laboratory of
Mathematical Sciences, Central China Normal University, Wuhan 430079,
P. R. China. Email: hfzhang@mails.ccnu.edu.cn.}
}

\date{\today}

\begin{document}
\maketitle
\begin{abstract}
We study minimizers of the Pekar functional with a transverse harmonic potential, which models a polaron harmonically trapped in the transverse plane but free along the longitudinal direction. For any confinement strength $\Omega>0$, minimizers exist and, up to a translation along the $x_3$-direction,  are radially decreasing in $(x_1,x_2)$ and symmetric decreasing in $x_3$. For small $\Omega>0$ we prove uniqueness up to these symmetries. For sufficiently large $\Omega>0$ we derive a three-term asymptotic expansion of the ground state energy and show that, after a suitable rescaling, the minimizers converge in $H^1(\mathbb{R}^3)\cap L^\infty(\mathbb{R}^3)$ to the product of the normalized ground state of a two-dimensional linear harmonic oscillator and that of an effective one-dimensional nonlinear local problem, displaying asymptotic variable separation.
This gives, for the transversely
confined Pekar model, a rigorous ground-state counterpart of the
three-dimensional-to-one-dimensional dimension-reduction phenomenon
numerically observed for Coulomb-type Schr\"odinger equations with
anisotropic confining potentials in [W. Z. Bao, H. Y. Jian, N. J. Mauser
and Y. Zhang, SIAM J. Appl. Math., 2013]. Moreover,
uniqueness up to the same symmetries is also established for all sufficiently large $\Omega>0$.
\end{abstract}

\vskip 0.05truein



\vskip 0.05truein

\section{Introduction}
A polaron describes an electron interacting with the quantized optical modes of a polar crystal, and a  ``large'' polaron refers to the case where the spatial extension of this polaron is large compared with the spacing of the underlying lattice.  In the large-polaron regime, the Pekar model provides a nonlinear variational description of the ground state, in which the phonon
field is treated classically and the electron density interacts with itself through a Coulomb-type attractive potential. For general background on polaron models and related analytical properties, we refer to the surveys \cite{Dev,Dev07,DA,GL}.

For three-dimensional models, strong transverse confinement is a standard mechanism leading to
quasi-one-dimensional effective models. In the Gross--Pitaevskii and Schr\"odinger--Poisson literature, this regime is often called cigar-shaped confinement, where the transverse variables are strongly confined while the longitudinal direction remains effective; see  \cite{BAC,BJMZ} and references therein. In a broader context, related rigorous results on partially confined Schr\"odinger--Poisson systems and strongly anisotropic nonlinear Schr\"odinger equations can be found in \cite{BMP,BMSW}. Low-dimensional and quasi-one-dimensional polaron systems have also attracted considerable attention, e.g. \cite{CTD,GAD}. Such quasi-one-dimensional regimes can be realized experimentally by imposing strong transverse confinement. For example, Catani et al. \cite{CLNGIMKG} created arrays of one-dimensional tubes by means of two-dimensional optical lattices, with the transverse harmonic oscillator frequency large enough to ensure the one-dimensional regime.  In these settings, the observable physics is effectively one-dimensional, yet a rigorous derivation of this dimension reduction from the full 3D model remains challenging, especially for nonlocal interactions.

Motivated by this partial confinement mechanism, we study in this paper a transversely confined Pekar model. More precisely, the confinement is imposed only in the transverse variables $x_\perp=(x_1,x_2)\in\R^2$, while the model remains translation invariant in the longitudinal $x_3$-direction. The polaron energy in this model is given by the functional
\begin{equation}\label{functional}
\E_\Omega(u):=\int_{\R^3}|\n u|^2dx+\Omega\int_{\R^3}|x_\perp|^2|u|^2dx-\f12\iint_{\R^3\times\R^3}\f{|u(x)|^2|u(y)|^2}{|x-y|}dxdy,
\end{equation}
and the ground state energy is defined as
\begin{equation}\label{inf}
E_\Omega=\inf\{\E_\Omega(u): \|u\|_2:=\|u\|_{L^{2}(\R^3)}=1,\ u\in \H\},
\end{equation}
where $x=(x_\perp,x_3)\in\R^2\times \R$, $\Omega>0$ represents the strength of the two-dimensional transverse harmonic trap. The corresponding Hilbert space $\H$ is defined by
\begin{equation}
\H:=\Big\{u\in H^1(\R^3): \int_{\R^3}|x_\perp|^2|u|^2dx<\infty\Big\},
\end{equation}
together with the norm
$$\|u\|_{\H}=\Big\{\int_{\R^{3}}|\nabla u|^{2}dx+\int_{\R^3}|x_\perp|^2|u|^2dx+\int_{\R^3}|u|^2dx\Big\}^{\frac{1}{2}}.$$

Partial confinement has also been studied for nonlinear Schr\"odinger equations (NLS) with local nonlinearities, e.g.  \cite{ACD, BBJV, HJ}. In particular, Bellazzini et al. \cite{BBJV} proved the existence, orbital stability and qualitative properties of  standing waves for supercritical NLS with a harmonic confinement acting only in the  $x_\perp$-variables. More recently, Hong and Jin \cite{HJ} studied the strong-confinement limit for the cubic NLS with a strong two-dimensional
harmonic potential, proving dimension reduction to the one-dimensional ground state and uniqueness of the three-dimensional minimizer for sufficiently strong confinement.

Compared with these local NLS models, the transversely  confined Pekar functional contains a nonlocal attractive Coulomb self-interaction. Thus, the lack of compactness caused by the translation invariance in the unconfined direction and the dimension reduction mechanism have to be treated together with this nonlocal Coulomb interaction, which makes the effective one-dimensional limit and the uniqueness analysis more delicate. To the best of our knowledge, uniqueness and sharp strong-confinement asymptotics for ground states of a transversely confined Pekar functional seem
not to have been established before. In this paper, we derive a three-term expansion of the ground state energy and explicitly identify the limiting one-dimensional profile in the strong-confinement regime, and prove uniqueness of positive minimizers up to translations in the unconfined direction in both the weak- and strong-confinement regimes.

Whenever a minimizer exists, the variational principle implies that every minimizer $u_\Omega$ satisfies the Euler-Lagrange equation
\begin{equation}\label{ele}
\begin{split}
(-\Delta_\perp+\Omega|x_\perp|^2)u_\Omega(x)-\p_3^2u_\Omega-\int_{\R^3}
\f{|u_\Omega(y)|^2}{|x-y|}dy u_\Omega(x)=\mu_\Omega u_\Omega(x)\text{ \ in \ }\R^3,
\end{split}
\end{equation}
where $\Delta_\perp:=\p_1^2+\p_2^2$, and $\mu_\Omega\in\R$ is a suitable Lagrange multiplier.
One can note from \cite[Theorem 6.17]{Lieb1} that if $u_\Omega$ is a minimizer of $E_\Omega$, then $\E_\Omega(u_\Omega)=\E_\Omega(|u_\Omega|)$, which implies that $|u_\Omega|$ is also a minimizer of $E_\Omega$. We thus assume that the minimizer is nonnegative. By Harnack's inequality (cf. \cite[Corollary 5.3]{Trudinger}), we then deduce from (\ref{ele}) that $|u_\Omega|>0$ holds in $\R^3$. Moreover, we have $u\equiv e^{i\theta}|u|~(\theta\in \R)$ for every minimizer of $E_\Om$. Therefore, without loss of generality, in the present paper we focus on positive minimizers of $E_\Omega$.

We first establish the existence and qualitative properties of minimizers
for  (\ref{inf}).

\begin{Theorem}\label{existence}
For any fixed $\Omega>0$, $E_\Omega$ admits at least one minimizer.  Moreover, up to a translation along the $x_3$-direction, the positive minimizer is radially decreasing with respect to $(x_1,x_2)$ and $x_3$.
\end{Theorem}

The existence proof is based on the concentration--compactness
principle. The transverse confinement prevents loss of compactness in
the $x_\perp$-variables, so that the only possible drift of minimizing
sequences occurs along the unconfined $x_3$-direction. The vanishing
and dichotomy alternatives are excluded by a strict energy bound and
a suitable splitting estimate for the nonlocal Coulomb interaction.  It is worth mentioning that, unlike \cite{BBJV}, we make full use of the spectral properties of the operator $-\Delta+\Omega|x_\perp|^2$ in $L^2(\R^3)$ to exclude the dichotomy.
Consequently, every minimizing sequence is compact up to translations
in the $x_3$-direction, which yields the existence of a minimizer.

The qualitative properties of minimizers require genuinely nonlocal arguments. While the radial symmetry and monotonicity in the confined variables follow from partial Schwarz symmetrization combined with the Riesz rearrangement inequality for the Coulomb term, the symmetry in the unconfined direction cannot be obtained by the same reflection argument used in the local NLS case (eg. \cite{BBJV}). We instead introduce a reflection construction for the density, use an exact Fourier identity for the Coulomb energy difference deriving in \cite{LM08}, and establish a Neumann boundary condition for the potential at the symmetry plane. Combining this with a unique continuation principle for the coupled Schr?dinger¨CPoisson system, we prove that any positive minimizer is even in
$x_3$ after a suitable translation. The monotonicity with respect to
 $|x_3|$ then follows from the equality characterization of Steiner symmetrization.

The Euler--Lagrange equation \eqref{ele} associated with \eqref{functional} is a
particular attractive Hartree equation of the form
\begin{equation*}
-\Delta\phi+V(x)\phi+\bigl(v*|\phi|^2\bigr)\phi=\mu\phi,
\end{equation*}
with $V(x)=\Omega|x_\perp|^2$ and $v(x)=-{|x|}^{-1}$. As pointed out by R. Seiringer in \cite{Seiringer}, at least for attractive interactions, uniqueness of minimizers does not hold in
general; see, e.g., \cite{AFGST,GS}. In contrast, we prove that, for the partially confined Pekar functional considered here, the positive minimizer is unique up to translations in the $x_3$-direction in both the weak- and strong-confinement regimes. We first state the result for weak-confinement.

\begin{Theorem}
\label{small-uniqueness}
There exists a sufficiently small $\Omega_0>0$ such that, for every
$0<\Omega<\Omega_0$, the positive minimizer of \eqref{inf} is unique up to translations in the $x_3$-direction.
\end{Theorem}

%A key observation in the proof of Theorem 1.2 is that the minimizers of $E_\Omega$ can be chosen to be even, and such even functions are $L^2$-orthogonal to the kernel of the linearized operator of the limit problem $E_0$ at its minimizers .
 A natural tool for carrying out such a perturbation analysis is the
implicit function theorem; see, e.g., \cite{GZZ}.
However, a straightforward application of the classical implicit
function theorem is not available here, since the natural
variational space changes in the limit: for $\Omega>0$, the
functional $E_\Omega$ is naturally defined on the weighted space
$\mathcal H$, whereas the limiting functional $E_0$ is naturally
defined on $H^1(\R^3)$. As $\Omega\searrow0$, the confinement term disappears from the
limiting functional, so that the limiting problem no longer provides
coercive control of the weighted norm
$\big\||x_\perp|u\big\|_{L^2(\R^3)}$.  As a result, there is no straightforward fixed
Banach-space framework in which the full perturbed equation is
well defined and the limiting linearized operator is simultaneously
invertible. Instead, we prove the local uniqueness   by combining a
suitable coercivity estimate with a nonlinear remainder estimate. The details are
given in Appendix \ref{Appendix}.
%This enables us to apply the implicit function theorem directly on the space of even functions to obtain the local uniqueness as $\Omega\searrow0$. The details are given in Appendix \ref{Appendix}.}

We now turn to the opposite strong-confinement regime
$\Omega\to+\infty$, which is the main focus of this paper. In this regime,
both the dimension reduction to an effective one-dimensional problem and the
uniqueness of minimizers involve new difficulties arising from the nonlocal
Hartree interaction. We first state the dimension reduction result.

\begin{Theorem}\label{thm1.2}
Let $u_\Omega$ be a positive minimizer of $E_\Omega$ defined by
\eqref{inf}. Then there exist translations
$\{\tau_\Omega\}\subset\R$ such that the rescaled functions
\begin{equation*}
w_\Omega^\tau(x_\perp,x_3):=\Omega^{-\frac14}(\ln\Omega)^{-\frac12}u_\Omega\left(\Omega^{-\frac14}x_\perp,(\ln\Omega)^{-1}(x_3+\tau_\Omega)\right)
\end{equation*}
satisfy
\begin{equation*}
w_\Omega^\tau\to W_0\quad\text{strongly in }\quad H^1(\R^3)\cap L^\infty(\R^3)\quad\text{as }\quad\Omega\to+\infty,
\end{equation*}
where
\begin{equation*}
W_0(x_\perp,x_3)=\frac{1}{4\sqrt{\pi}}e^{-\frac{|x_\perp|^2}{2}}\left(\cosh\frac{x_3}{8}\right)^{-1}\quad \text{for}\quad x=(x_\perp, x_3)\in \R^3.
\end{equation*}
\end{Theorem}

Theorem \ref{thm1.2} shows that when the 2D confinement is sufficiently strong (i.e., in the large $\Omega$ limit), the ground state of the three-dimensional nonlinear nonlocal problem undergoes asymptotic variable separation. More precisely, after a suitable translation and rescaling, the minimizer $u_\Omega$ asymptotically factorizes as the product of a two-dimensional Gaussian $\frac{1}{\sqrt\pi}e^{-|x_\perp|^2/2}$--which is the ground state of a $2$D linear harmonic oscillator--and a one-dimensional soliton profile $(4\cosh(x_3/8))^{-1}$--which is the ground state of a $1$D nonlinear local problem. Thus, the full $3$D ground state is asymptotically given by the product of the ground states of a 2D linear problem and a $1$D nonlinear local problem.
This concentration mechanism should be distinguished from the standard dimension-reduction regime under strong anisotropic confinement. In such problems, the strongly confining part of the external potential is usually imposed only in some of the spatial variables,  while the remaining variables are retained as effective lower-dimensional variables; see, e.g., \cite{BAC, BJMZ,RY}. By contrast, in the present Pekar model, no spatial variable remains at the original scale: the transverse variables concentrate at the harmonic-oscillator scale $\Omega^{-1/4}$, whereas the longitudinal variable concentrates at the logarithmic scale $(\ln\Omega)^{-1}$.

The proof of Theorem \ref{thm1.2} proceeds via a sharp energy expansion combined with a transverse spectral decomposition.
In the
concentration analysis of attractive Gross--Pitaevskii minimizers,
precise energy estimates typically determine a single isotropic
blow-up scale, after which the rescaled minimizers converge to a
full-dimensional limiting profile; see, e.g., \cite{GS,GZZ}. Such a
one-scale scheme is not directly applicable here, since the transverse
and longitudinal variables concentrate at different scales and the
limiting profile emerges through asymptotic variable separation.
To
overcome this difficulty, we first establish the sharp energy asymptotics stated in Theorem
\ref{thm3.1} below.
More precisely, inspired by \cite{FG}, we prove that
\begin{equation*}
E_\Omega=2\sqrt{\Omega}-\f{1}{192}(\ln \Omega)^2+\f{1}{24}(\ln \Omega)(\ln \ln \Omega)+O(\ln \Omega)\text{ \ as \ }\Omega\to+\infty.
\end{equation*}
This expansion determines the correct anisotropic scales and provides the starting point for the profile analysis. The leading term $2\sqrt{\Omega}$, together with the spectral gap of the two-dimensional Hermite operator, is used to control the components away from the lowest transverse Hermite mode. The subsequent logarithmic correction then identifies the effective longitudinal scale $(\ln\Omega)^{-1}$.

Compared with the dimension-reduction analysis for NLS with local nonlinearities under strong partial confinement, such as \cite{BMSW,HJ}, the main difficulty here is that the nonlinear term is nonlocal.
For the local nonlinearities considered there,  projection onto
the lowest transverse mode preserves the local structure of the
nonlinearity and directly yields an effective lower-dimensional local
term.
%For a local cubic nonlinearity, after the projection onto the lowest transverse mode, the nonlinear energy directly reduces to a one-dimensional quartic term.
In the present Pekar model, however, the Hartree term remains nonlocal even after the transverse projection. Thus, one has to analyze the effective one-dimensional Coulomb kernel generated by the lowest transverse mode and prove that, after the normalization by $\ln\Omega$, it converges to a local interaction in the
sense of bilinear forms, see Lemma \ref{lem4.3} below. %Moreover, unlike the variational framework in \cite{HJ}, no additional transverse-energy constraint is imposed in our minimization problem, and the required control of the higher transverse modes has to be derived directly from the Pekar energy. A further distinction concerns the scaling of the interaction.
Moreover, in the weak-interaction limit studied in \cite{BAC,HJ},  the interaction coefficients tend to zero as the strength of the transverse confinement
increases.  In contrast, the Coulomb coefficient in \eqref{functional} is fixed. As a result, the reduced Hartree interaction is responsible for both the logarithmic longitudinal scale $(\ln\Omega)^{-1}$ and the limiting one-dimensional profile.

As an application of the above profile convergence, we finally prove uniqueness of positive minimizers for sufficiently strong-confinement.

\begin{Theorem}
\label{large-uniqueness}
There exists a sufficiently large $\Omega_1>0$ such that, for every $\Omega>\Omega_1$, the
positive minimizer of \eqref{inf} is unique up to
translations in the
$x_3$-direction.
\end{Theorem}

The proof requires a mechanism different from the existing
local-uniqueness arguments for concentrating Gross--Pitaevskii
minimizers. In the local setting, after a single blow-up rescaling, the
normalized difference of two minimizers is reduced to a linear
combination of the scaling and translation modes of the limiting
equation. The corresponding coefficients can then be eliminated by
local Pohozaev identities and the nondegeneracy of a finite-dimensional
location function; see, e.g., \cite{GLW}. This argument cannot be
directly adapted here. Indeed, the transverse and longitudinal
variables concentrate at the different scales $\Omega^{-1/4}$ and
$(\ln\Omega)^{-1}$, respectively, while the Hartree interaction is
nonlocal. Moreover, the exact translation invariance in the
$x_3$-direction makes the longitudinal location mode genuinely
degenerate, which explains why uniqueness can hold only modulo
longitudinal translations.

Our approach is instead based on a uniform coercivity estimate for the
linearized quadratic form on even perturbations satisfying the mass
orthogonality condition.
%Since the longitudinal translation mode is odd in $x_3$, it is automatically excluded from this subspace.
For an admissible perturbation $\phi$,
we use the transverse decomposition
$$  \phi=P_1\phi+Q_1\phi,$$
where $P_1\phi$ belongs to the lowest transverse Hermite mode and
$Q_1\phi$ contains all higher modes. The latter is controlled by the
spectral gap of the transverse harmonic oscillator, whereas the former
is governed in the limit by a one-dimensional linearized quadratic
form.

The main difficulty is that the nonlocal Hartree interaction couples
these two components. Lemma \ref{0lnc} provides estimates showing that the Hartree terms
involving the higher transverse modes, as well as the mixed terms, can
be controlled by the positive kinetic and transverse spectral-gap
terms, together with
an $L^2$-remainder involving the
higher-mode component $Q_1\phi$.  Combining these estimates
with the coercivity of the limiting one-dimensional quadratic form and
the transverse spectral gap, we obtain the uniform coercivity estimate stated
 in Proposition \ref{0prlo}. This coercivity estimate
then yields the uniqueness asserted in Theorem
\ref{large-uniqueness}.

This paper is organized as follows. In Section \ref{S2}, we prove Theorem \ref{existence}, which establishes the existence of minimizers for \eqref{inf} and some of their basic properties. In Section \ref{S3}, we derive the refined asymptotic expansion of $E_\Omega$ as $\Omega\to+\infty$. Based on this expansion, we prove Theorem \ref{thm1.2} in Section \ref{S4}. Section \ref{S5} is devoted to the proof of Theorem \ref{large-uniqueness}, which gives the uniqueness of positive minimizers for sufficiently large $\Omega>0$. Finally, Theorem \ref{small-uniqueness}, concerning the uniqueness of positive minimizers for sufficiently small $\Omega>0$, is proved in the Appendix \ref{Appendix}.

\section{Existence and Qualitative Properties of Minimizers}\label{S2}

In this section, we first prove the existence part of Theorem \ref{existence} by the concentration-compactness principle (cf. \cite{Lions}). Next, we establish the qualitative properties of  positive minimizers, namely the radial symmetry and monotonicity with respect to $|x_\perp|$, and the even symmetry and monotonicity with respect to $|x_3|$ after a suitable translation in the $x_3$-direction.

We will use the following Coulomb-type
Gagliardo--Nirenberg inequality:
\begin{equation}\label{GN}
\iint_{\R^3\times\R^3}\f{|u(x)|^2|u(y)|^2}{|x-y|}dxdy\leq S \Big(\int_{\R^3}|u|^2dx\Big)^{\f{3}{2}}\Big(\int_{\R^3}|\nabla u|^2dx\Big)^{\f{1}{2}}, \ \ u\in H^1(\R^3),
\end{equation}
for some constant $S>0$.
This inequality follows from the
Hardy-Littlewood-Sobolev inequality (cf. \cite[Theorem 4.3]{Lieb1}) and the interpolation inequality.
It then follows from (\ref{GN}) that for every $u\in\mathcal H$ with $\|u\|_2=1$,
\begin{equation}\label{2.1}
\begin{aligned}
\E_{\Omega}(u)
\geq&\int_{\R^3}|\n u|^2dx+\Omega\int_{\R^3}|x_\perp|^2|u|^2dx-\f{ S}{2} \Big(\int_{\R^3}|u|^2dx\Big)^{\f{3}{2}}\Big(\int_{\R^3}|\nabla u|^2dx\Big)^{\f{1}{2}}\\
\geq&\f34\int_{\R^3}|\n u|^2dx+\Omega\int_{\R^3}|x_\perp|^2|u|^2dx-\f{S^2}{4},
\end{aligned}
\end{equation}
which implies that $\E_{\Omega}(u)$ is bounded from below.
Thus, for any fixed $\Omega>0$, there exists a minimizing sequence $\{u_n\}\subset\H$ with $\|u_n\|_2=1$ satisfying $\lim_{n\to\infty}\E_{\Omega}(u_n)=E_{\Omega}$. Moreover, it follows from (\ref{2.1}) that $\{u_n\}$ is bounded uniformly in $\H$.

We begin with two preliminary lemmas. The first one gives a non-vanishing
property of the Coulomb interaction along any minimizing sequence, while the
second one excludes vanishing up to translations in the free $x_3$-direction.

\begin{Lemma}\label{lem2.1}
Let $\{u_n\}\subset\H$ be a minimizing sequence of (\ref{inf}), then we have
\begin{equation}\label{2.2}
\liminf_{n\to\infty}\iint_{\R^3\times\R^3}\f{|u_n(x)|^2|u_n(y)|^2}{|x-y|}dxdy>0.
\end{equation}
\end{Lemma}
\noindent{\bf Proof.}
For any fixed $\Omega>0$, define
\begin{equation}\label{2.3}
\Lambda_{0}:=\inf_{\substack{w\in\H\\\int_{\R^3}|w|^2dx=1}}\Big(\int_{\R^3}|\nabla w|^2dx
+\Omega\int_{\R^3}|x_\perp|^2|w|^2dx\Big),
\end{equation}
and
\begin{equation}\label{2.4}
\lambda_{0}:=\inf_{\substack{\int_{\R^2}|\psi|^2dx_{\perp}=1\\\psi\in\{\psi\in H^1(\R^2):\int_{\R^2}(|\nabla \psi|^2+|x_\perp|^2|\psi|^2)dx_{\perp}<\infty\}}}\Big(\int_{\R^2}|\nabla \psi|^2dx_{\perp}
+\Omega\int_{\R^2}|x_\perp|^2|\psi|^2dx_{\perp}\Big).
\end{equation}
Similar to \cite[Lemma 2.1]{BBJV}, one can check that $\Lambda_{0}=\lambda_{0}$. Moreover, it can be calculated that $\lambda_{0}=2\sqrt{\Omega}$, together with a corresponding unique positive minimizer
\begin{equation}\label{psi0}
\psi_0(x_{\perp})=\frac{\Omega^\f14}{\sqrt{\pi}}e^{-\frac{\sqrt{\Omega}|x_{\perp}|^2}{2}}.
\end{equation}

Assume by contradiction that (\ref{2.2}) is false. Then, there exists a subsequence $\{u_{n_k}\}$ of $\{u_n\}$ such that
\begin{equation*}
\liminf_{n\to\infty}\iint_{\R^3\times\R^3}\f{|u_n(x)|^2|u_n(y)|^2}{|x-y|}dxdy
=\lim_{n_k\to\infty}\iint_{\R^3\times\R^3}\f{|u_{n_k}(x)|^2|u_{n_k}(y)|^2}{|x-y|}dxdy
=0,
\end{equation*}
and
%\begin{equation}\label{2.5}
\begin{align}
E_{\Omega}=&\lim_{n_k\to\infty}\E_{\Omega}(u_{n_k})\label{2.5}\\
%=&\int_{\R^3}|\n u_{n_k}|^2dx+\Omega\int_{\R^3}|x_\perp|^2|u_{n_k}|^2dx-\f12\iint_{\R^3\times\R^3}\f{|u_{n_k}(x)|^2|u_{n_k}(y)|^2}{|x-y|}dxdy
%+o(1)\label{2.5}\\
=&\int_{\R^3}|\n u_{n_k}|^2dx+\Omega\int_{\R^3}|x_\perp|^2|u_{n_k}|^2dx+o(1)
\geq \Lambda_{0}+o(1)\ \ \text{as}\ \ n_k\to\infty,\nonumber
\end{align}
%\end{equation}
where $o(1)$ satisfies $\lim_{{n_k}\to\infty}o(1)=0$. We thus derive from (\ref{2.5}) that
\begin{equation}\label{2.9}
E_{\Omega}\geq \Lambda_{0}.
%\ \ \text{for any}\ \ r>0.
\end{equation}

On the other hand, we claim that
\begin{equation}\label{2.6}
E_{\Omega}< \Lambda_{0}.%\ \ \text{for any}\ \ r>0.
\end{equation}
In fact, set
$$w_{\mu}(x)=\psi_0(x_{\perp})\varphi_{\mu}(x_3)\in\H,\ \ 0<\mu<1,$$
where $\psi_0(x_{\perp})$ is given by \eqref{psi0}, and $\varphi_{\mu}(x_3)=\sqrt{\mu}\varphi(\mu x_3)$ with $\varphi(x_3)=\Big(\f2\pi\Big)^{\f14}e^{-x_3^2}$, so that $\int_\R|\varphi_{\mu}(x_3)|^2dx_3=1$. We then deduce from (\ref{functional}), (\ref{2.3}) and (\ref{2.4}) that
\begin{align}
\E_{\Omega}(w_{\mu})
=&\Lambda_0+\int_{\R}|\varphi_{\mu}'(x_3)|^2dx_3-\f12\iint_{\R^3\times\R^3}\f{|w_{\mu}(x)|^2
|w_{\mu}(y)|^2}{|x-y|}dxdy\label{2.7}\\
\leq&\Lambda_0+\mu^2\int_{\R}|\varphi'(x_3)|^2dx_3 -\f{\mu}{2}\iint_{\R^3\times\R^3}\f{|\psi_0(x_{\perp})|^2
|\varphi(x_3)|^2|\psi_0(y_{\perp})|^2
|\varphi(y_3)|^2}{|x-y|}dxdy \nonumber\\
<& \Lambda_0\quad \text{ for sufficiently small  }\quad \mu>0.\nonumber
\end{align}
We thus conclude from (\ref{2.7}) that the claim (\ref{2.6}) holds true. However, this is a contradiction with (\ref{2.9}). Thus, (\ref{2.2}) holds true, and the proof of Lemma \ref{lem2.1} is therefore complete.\qed

It follows from Lemma \ref{lem2.1} that there exist constants $\varepsilon_0>0$ and
$N_0>0$ such that
\begin{equation}\label{2.10}
\iint_{\R^3\times\R^3}\f{|u_n(x)|^2|u_n(y)|^2}{|x-y|}dxdy>\varepsilon_0\ \ \text{for all}\ \ n \geq N_0.
\end{equation}
%Following the concentration-compactness principle of  Lions (cf. \cite{Lions}), we first exclude the vanishing of the minimizing sequence.
We now use this lower bound to exclude the vanishing alternative in the
concentration-compactness principle of Lions \cite{Lions}.

\begin{Lemma}\label{lem2.2}
Let $\{u_n\}$ be a minimizing sequence of $E_{\Omega}$, then there exist a sequence $\{x_n=(0,0,z_n)\}\subset\R^3$ and a constant $R>0$ such that
\begin{equation}
\liminf_{n\to\infty}\int_{B_R(x_n)}|u_n|^2dx>0.
\end{equation}
\end{Lemma}

\noindent{\bf Proof.} We first claim that there exists a constant $r_0>0$ such that
\begin{equation}\label{NV}
\liminf_{n\to\infty}\sup_{y\in\R^3}\int_{B_{r_0}(y)}|u_n|^2dx>0.
\end{equation}
Otherwise, there exists a subsequence, still denoted by $\{u_n\}$, of $\{u_n\}$ such that
\begin{equation}\label{A2-2}
\lim_{n\to\infty}\sup_{y\in\R^3}\int_{B_r(y)}|u_n|^2dx=0\text{ \ for any fixed \ }r>0.
\end{equation}
We then  obtain from  \cite[Lemma 1.21]{Willem} that $u_n\to0$ strongly in $L^p(\R^3)$ as $n\to\infty$ for $2<p<6$. Therefore, using the Hardy-Littlewood-Sobolev inequality (cf. \cite[Theorem 4.3]{Lieb1}) we deduce that
\begin{equation*}
\iint_{\R^3\times\R^3}\f{|u_n(x)|^2|u_n(y)|^2}{|x-y|}dxdy\leq C\Big(\int_{\R^3}|u_{n}(x)|^{\f{12}{5}}dx\Big)^{\f53}\to0\text{ \ as \ }n\to\infty,
\end{equation*}
which contradicts with \eqref{2.10}, and hence \eqref{NV} holds true. It follows from \eqref{NV} that there exists a sequence $\{x'_n\}\subset\R^3$ such that
\begin{equation}\label{NV1}
\liminf_{n\to\infty}\int_{\R^2\times(z_n-r,z_n+r)}|u_n|^2dx\geq\liminf_{n\to\infty}\int_{B_{r_0}(x'_n)}|u_n|^2dx>0\text{ \ for every \ }r\geq r_0,
\end{equation}
where $z_n\in\R$ is the third component of $x'_n$.

Similar to \cite[Lemma 3.1]{Z}, one can check that the embedding
\begin{equation}\label{ceb}
\H_r:=\Big\{u\in H^1(T_r): \int_{T_r}|x_\perp|^2u^2dx<\infty\Big\}\hookrightarrow L^{2}(T_r)
\end{equation}
is compact for every fixed $r>0$, where $T_r=\R^2\times(-r,r)$.
We then obtain from \eqref{NV1} and \eqref{ceb} that there exist a subsequence, still denoted by $\{u_n\}$, of $\{u_n\}$ and a function $u\in \H$ such that
$$u_n(x+(0,0,z_n))\to u\not\equiv0\ \ \text{strongly in}\ \ L^2(T_r)\ \ \text{for every $r\geq r_0$ as}\ \ n\to\infty,$$
which further implies that Lemma \ref{lem2.2} holds true by taking sufficiently large $R>r_0$ such that $\int_{B_R(0)}|u|^2dx>0$.\qed

\vskip 0.05truein

We are now ready to prove the existence part of Theorem \ref{existence}.

\vskip 0.05truein

\noindent{\bf Proof of the existence part of Theorem \ref{existence}.}
It suffices to prove that the dichotomy cannot occur. To this end, we assume that the asymptotic mass
\begin{equation}
\lim_{R\to\infty}\liminf_{n\to\infty}\int_{B_R(x_n)}|u_n|^2dx=:\beta\in(0,1),
\end{equation}
where the sequence $\{x_n\}\subset \R^3$ is given by Lemma \ref{lem2.2}. Since the functional \eqref{functional} is invariant under the translation with respect to the third component, we assume without loss of generality that $x_n\equiv0$ for all $n\in\N^+$. There then exist (cf. \cite{Lions}) a subsequence, still denoted by $\{u_n\}$, of $\{u_n\}$ and a sequence $\{R_n\}\subset\R^+$ with  $\lim_{n\to\infty}R_n=\infty$  such that
\begin{equation}\label{Levi1}
\lim_{n\to\infty}\int_{B_{R_n}}|u_n|^2dx=\beta, \lim_{n\to\infty}\int_{B_{2R_n}^c}|u_n|^2dx=1-\beta,
\end{equation}
and
\begin{equation}\label{Levi2}
\lim_{n\to\infty}\int_{B_{2R_n}\setminus B_{R_n}}|u_n|^2dx=0,
\end{equation}
which implies from the interpolation inequality that
\begin{equation}\label{ann}
\lim_{n\to\infty}\int_{B_{2R_n}\setminus B_{R_n}}|u_n|^pdx=0\text{ \ for every \ } p\in[2,6).
\end{equation}
 Here and in what follows, we use $B_R$ to represent $B_R(0)$.

Choose a cut-off function $f\in C_c^\infty(\R^3,[0,1])$ such that $f(x)\equiv1$ for $|x|<1$ and $f(x)\equiv0$ for $|x|>2$. Define
\begin{equation}{\label{fg}}
f_n(x):=f\big(x/R_n\big)\ \ \text{and}\ \  g_n(x):=g\big(x/R_n\big),
\end{equation}
where $g(x):=\sqrt{1-f^2(x)}$ and $R_n$ is as in \eqref{Levi1}. By the IMS formula (cf. \cite[Theorem 3.2]{CFKS87}), we have
\begin{align}
\int_{\R^3}|\n u_n|^2dx
&=\int_{\R^3}|\n (f_nu_n)|^2dx+\int_{\R^3}|\n (g_nu_n)|^2dx-\int_{\R^3}(|\n f_n|^2+|\n g_n|^2)|u_n|^2dx\nonumber\\
&=\int_{\R^3}|\n (f_nu_n)|^2dx+\int_{\R^3}|\n (g_nu_n)|^2dx+o(1)\quad \text{ as  }\quad  n\to\infty.\label{trdn3}
\end{align}
Moreover, one can deduce  that
\begin{equation}\label{V3}
\int_{\R^3}|x_\perp|^2|u_n|^2dx=\int_{\R^3}|x_\perp|^2|f_nu_n|^2dx+\int_{\R^3}|x_\perp|^2|g_nu_n|^2dx,
\end{equation}
and
\begin{equation*}
\begin{split}
&\iint_{\R^3\times\R^3}\f{|u_n(x)|^2|u_n(y)|^2}{|x-y|}dxdy\\
=&\iint_{\R^3\times\R^3}\f{|f_n(x)u_n(x)|^2|f_n(y)u_n(y)|^2}{|x-y|}dxdy+\iint_{\R^3\times\R^3}\f{|g_n(x)u_n(x)|^2|g_n(y)u_n(y)|^2}{|x-y|}dxdy\\
&+2\iint_{\R^3\times\R^3}\f{|f_n(x)u_n(x)|^2|g_n(y)u_n(y)|^2}{|x-y|}dxdy.\\
\end{split}
\end{equation*}
Using Hardy-Littlewood-Sobolev inequality, it yields that
\begin{equation*}
\begin{aligned}
\iint_{\R^3\times\R^3}\f{|f_n(x)u_n(x)|^2|g_n(y)u_n(y)|^2}{|x-y|}dxdy
\leq& \iint_{\R^3\times\R^3}\f{|\mathds{1}_{B_{2R_n}}(x)u_n(x)|^2
|\mathds{1}_{B_{R_n}^c}(y)u_n(y)|^2}{|x-y|}dxdy\\
%=&\iint_{\R^3\times\R^3}\f{|\mathds{1}_{B_{R_n}}(x)u_n(x)|^2
%|\mathds{1}_{B_{2R_n}^c}(y)u_n(y)|^2}{|x-y|}dxdy\nonumber\\
%&+\iint_{\R^3\times\R^3}\f{|\mathds{1}_{B_{R_n}}(x)u_n(x)|^2|\mathds{1}_{B_{2R_n}\setminus B_{R_n}}(y)u_n(y)|^2}{|x-y|}dxdy\nonumber\\
%&+\iint_{\R^3\times\R^3}\f{|\mathds{1}_{B_{2R_n}\setminus B_{R_n}}(x)u_n(x)|^2|\mathds{1}_{B_{2R_n}^c}(y)u_n(y)|^2}{|x-y|}dxdy\nonumber\\
%&+\iint_{\R^3\times\R^3}\f{|\mathds{1}_{B_{2R_n}\setminus B_{R_n}}(x)u_n(x)|^2|\mathds{1}_{B_{2R_n}\setminus B_{R_n}}(y)u_n(y)|^2}{|x-y|}dxdy\nonumber\\
\leq& C\Big(\f{1}{R_n}+\big(\int_{\R^3}|\mathds{1}_{B_{2R_n}\setminus B_{R_n}}u_n|^{\f{12}{5}}dx\big)^{\f{5}{6}}\Big)
=o(1)\text{ \ as \ }n\to\infty,\nonumber
\end{aligned}
\end{equation*}
where we used \eqref{ann} in the last inequality. Therefore,
\begin{align}
\iint_{\R^3\times\R^3}\f{|u_n(x)|^2|u_n(y)|^2}{|x-y|}dxdy
=&\iint_{\R^3\times\R^3}\f{|f_n(x)u_n(x)|^2|f_n(y)u_n(y)|^2}{|x-y|}dxdy \label{nonl}\\
&+\iint_{\R^3\times\R^3}\f{|g_n(x)u_n(x)|^2|g_n(y)u_n(y)|^2}{|x-y|}dxdy+o(1)\nonumber
\end{align}
as $n\to\infty$.
It then follows from \eqref{functional} and \eqref{trdn3}--\eqref{nonl} that
\begin{equation}\label{gfg3}
\begin{split}
\E_\Om(u_n)&=\E_\Om(f_nu_n)+\E_\Om(g_nu_n)+o(1)\text{ \ as \ }n\to\infty.
\end{split}
\end{equation}

For every $\lambda>0$, we have
\begin{equation*}
\begin{split}
\E_\Om(\lambda u)&=\lambda^2\Big(\int_{\R^3}|\n u|^2dx+\Omega\int_{\R^3}|x_\perp|^2|u|^2dx\Big)-\f{\lambda^{4}}{2}\iint_{\R^3\times\R^3}\f{|u(x)|^2|u(y)|^2}{|x-y|}dxdy\\
&=\lambda^{4}\E_\Om(u)+\lambda^2(1-\lambda^{2})\Big(\int_{\R^3}|\n u|^2dx+\Omega\int_{\R^3}|x_\perp|^2|u|^2dx\Big).
\end{split}
\end{equation*}
One can then deduce that
\begin{equation}\label{eag2}
\E_\Om(u)=\lambda^{-4}\E_\Om(\lambda u)+(1-\lambda^{-2})\Big(\int_{\R^3}|\n u|^2dx+\Omega\int_{\R^3}|x_\perp|^2|u|^2dx\Big).
\end{equation}
Take
\begin{equation*}
\lambda_n=\f{1}{\|f_nu_n\|_2}\text{ \ and \ }\lambda_n'={\f{1}{\|g_nu_n\|_2}},
\end{equation*}
so that $\|\lambda_n u_n f_n\|_2=1$ and  $\|\lambda_n'u_ng_n\|_2=1$.
It then follows from \eqref{eag2} that
\begin{equation*}\label{d3}
\begin{split}
\E_\Om(f_nu_n)
= &(\lambda_n)^{-4 }\E_\Om(\lambda_nf_nu_n)+(1-\lambda_n ^{-2})\Big(\int_{\R^3}|\n (f_nu_n)|^2dx+\Omega\int_{\R^3}|x_\perp|^2|f_nu_n|^2dx\Big)\\
\geq& \big\|f_nu_n\big\|_2^{4}E_\Om+\Big(1-\|f_nu_n\big\|_2^{2}\Big)\Lambda_0\|f_nu_n\|_2^2,
\end{split}
\end{equation*}
where $\Lambda_0$ is defined in \eqref{2.3}. Noting from \eqref{Levi1} and \eqref{Levi2} that $\|f_nu_n\|_2^2\to\beta$ as $n\to\infty$, we then deduce that
\begin{equation}\label{2.29}
\E_\Om(f_nu_n)\geq \beta^2 E_\Om+(1-\beta)\beta\Lambda_0+o(1)\text{ \ as \ }n\to\infty.
\end{equation}
Similarly, since $\|g_nu_n\|_2^2\to1-\beta$ as $n\to\infty$, one can deduce that
\begin{equation}\label{2.30}
\E_\Om(g_nu_n)\geq (1-\beta)^2 E_\Om+(1-\beta)\beta\Lambda_0+o(1)\text{ \ as \ }n\to\infty.
\end{equation}
It then follows from \eqref{gfg3}, \eqref{2.29} and \eqref{2.30} that
\begin{equation}
\begin{split}
E_\Om= \lim_{n\to\infty}\E_\Om(u_n)
\geq (2\beta^2-2\beta+1)E_\Om+2\beta(1-\beta)\Lambda_0,
\end{split}
\end{equation}
which implies that $E_\Om\geq\Lambda_0$ in view of the fact $0<\beta<1$. This contradicts the strict inequality \eqref{2.6}.

Thus, we conclude that there exist a subsequence, still denoted by $\{u_n\}$, of $\{u_n\}$ and a function $u\in\H$ such that
\begin{equation*}
 u_n\to u\text{ \ strongly in \ }L^2(\R^3) \text{ \ as \ } n\to\infty,%\text{ \ for \ }2\leq p<6,
\end{equation*}
which implies that
\begin{equation*}\label{2.24}
\lim_{n\to\infty}\iint_{\R^3\times\R^3}\f{|u_n(x)|^2|u_n(y)|^2}{|x-y|}dxdy\\
=\iint_{\R^3\times\R^3}\f{|{u}(x)|^2|{u}(y)|^2}{|x-y|}dxdy
\end{equation*}
in view of \cite[Lemma 2.2]{BFV}.
%By the weak lower semicontinuity,
Applying Fatou's lemma, we further deduce that
\begin{align}
E_{\Omega}
=&\lim_{n\to\infty}\Big(\int_{\R^3}|\n u_{n}|^2dx+\Omega\int_{\R^3}|x_\perp|^2|u_{n}|^2dx
-\f12\iint_{\R^3\times\R^3}\f{|u_{n}(x)|^2|u_{n}(y)|^2}{|x-y|}dxdy\Big)\nonumber\\
\geq& \int_{\R^3}|\n{u}|^2dx+\Omega\int_{\R^3}|x_\perp|^2|{u}|^2dx
-\f12\iint_{\R^3\times\R^3}\f{|{u}(x)|^2|{u}(y)|^2}{|x-y|}dxdy
= \E_{\Omega}({u})\geq E_{\Omega},\nonumber
\end{align}
which implies that ${u}$ is a minimizer of $E_{\Omega}$. This proves the
existence part of Theorem \ref{existence}.\qed

We next establish the qualitative properties of positive minimizers, which
will complete the proof of Theorem \ref{existence}.

\begin{Proposition}\label{lem2.3}
For any fixed $\Omega>0$, let $u_\Omega$ be a positive minimizer of $E_\Omega$, then up to a translation along the $x_3$-direction, the positive minimizer is radially decreasing with respect to $(x_1,x_2)$ and $x_3$.
%$u_\Omega(x_\bot, x_3)$ is radially symmetric and decreasing with respect to $|x_\bot|$ for every $x_3\in\R$.  Moreover, up to a translation in the $x_3$-direction, $u_\Omega(x_\bot, x_3)$ is even in $x_3$ and decreasing with respect to $|x_3|$.
\end{Proposition}
\noindent{\bf Proof.} Suppose $u_\Omega$ is a positive minimizer of $E_\Omega$, define the transverse symmetrization of $u_\Omega$ with respect to $x_\bot$:
$$\tilde{u}(x_\bot,x_3):=u^*_{x_3}(x_\bot),$$
where $u_{x_3}(x_\bot)=u_\Omega(x_\bot,x_3)$ and $*$ denotes the Schwartz rearrangement with respect to $x_\bot$. It yields from \cite[Lemma 8.2]{BS} that
$$\iint_{\R^3\times\R^3}\frac{|\tilde{u}(x)|^2|\tilde{u}(y)|^2}{|x-y|}dxdy\geq
\iint_{\R^3\times\R^3}\frac{| u_\Omega(x) |^2| u_\Omega(y) |^2}{|x-y|}dxdy.$$
Moreover, using the similar argument of \cite[Theorem 2]{BBJV}, one can check that $u_\Omega(x_\bot,x_3)$ is radially symmetric and decreasing with respect to $|x_\bot|$ for every $x_3\in\R$.

We next consider the symmetry of $u_\Om$ with respect to $x_3\in\R$. Since the functional $\E_\Om$ is translation invariant in the $x_3$-direction, we may translate $u_\Om$ so that
\begin{equation}\label{trans12}
\int_{\{x_3<0\}} |u_\Om(x)|^2dx=\int_{\{x_3>0\}} |u_\Om(x)|^2dx=\frac12.
\end{equation}
Define the reflected functions
\begin{equation*}
u_1(x_\perp,x_3)=
\begin{cases}
u_\Om(x_\perp,x_3),& x_3<0,\\
u_\Om(x_\perp,-x_3),& x_3\geq 0,
\end{cases}
\text{ \ and \ }
u_2(x_\perp,x_3)=
\begin{cases}
u_\Om(x_\perp,-x_3),& x_3<0,\\
u_\Om(x_\perp,x_3),& x_3\geq 0,
\end{cases}
\end{equation*}
so that $u_1,u_2\in\mathcal H$ and $\|u_1\|_{2}=\|u_2\|_{2}=1$. Moreover,
\begin{equation*}
\int_{\R^3}\big(|\nabla u_1|^2+|\nabla u_2|^2\big)dx=2\int_{\R^3}|\nabla u_\Omega|^2dx,
\text{ \ and \ }
\int_{\R^3} |x_\perp|^2\big(|u_1|^2+|u_2|^2\big)dx=2\int_{\R^3} |x_\perp|^2|u_\Omega|^2dx.
\end{equation*}
For the nonlocal term, set $\rho:=|u_\Om|^2$, $\rho_1:=|u_1|^2$ and $\rho_2:=|u_2|^2$. Similar to \cite[(2.38)]{LM08}, one can check that
\begin{equation*}
\begin{aligned}
&\iint_{\R^3\times \R^3} \frac{\rho_1(x)\rho_1(y)}{|x-y|} dxdy
+\iint_{\R^3\times \R^3} \frac{\rho_2(x)\rho_2(y)}{|x-y|} dxdy
-2\iint_{\R^3\times \R^3} \frac{\rho(x)\rho(y)}{|x-y|} dxdy\\
=&32\int_{\mathbb R^2}\frac{1}{|\xi'|}
\left|\int_0^\infty
\widehat{A\rho}(\xi',\xi_3)\frac{\xi_3}{|\xi'|^2+\xi_3^2} d\xi_3
\right|^2 d\xi'
\geq 0,
\end{aligned}
\end{equation*}
where $A\rho(x_\perp,x_3)=\frac12\big(\rho(x_\perp,x_3)-\rho(x_\perp,-x_3)\big)$ is the odd part of $\rho$ with respect to $x_3$, and $\xi'=(\xi_1,\xi_2)$. Consequently,
\begin{equation*}
\mathcal E_\Omega(u_1)+\mathcal E_\Omega(u_2)-2\mathcal E_\Omega(u_\Om)\leq 0.
\end{equation*}
Since $u_\Omega$ is a minimizer and $u_1,u_2$ are admissible, we have $\mathcal E_\Omega(u_1)\geq\mathcal E_\Omega(u_\Om)$ and $\mathcal E_\Omega(u_2)\geq\mathcal E_\Omega(u_\Om)$. Hence equality must hold everywhere. Therefore,  we have for $a.e.$  $\xi'\in\mathbb R^2$,
\begin{equation}\label{symmx3}
\int_0^\infty
\widehat{A\rho}(\xi',\xi_3)\frac{\xi_3}{|\xi'|^2+\xi_3^2} d\xi_3=0.
\end{equation}

Set $U:=I(\rho)=\int_{\R^3}\frac{\rho(y)}{|x-y|} dy$. Using \eqref{symmx3} and
$\widehat U(\xi)=\frac{4\pi}{|\xi|^2}\widehat\rho(\xi)$, we obtain that $\int_{\R}i\xi_3\widehat U(\xi',\xi_3) d\xi_3=0$ for almost every \(\xi'\in\R^2\), where the even part of $\rho$ gives no contribution by oddness in
$\xi_3$. Arguing as in the proof of
\cite[Theorem~4.6]{LM08}, we conclude that
$$
\int_{\R^2}\partial_{x_3}U(x_\perp,0)\psi(x_\perp) dx_\perp=0\quad\text{for every } \quad\psi\in C_c^\infty(\R^2).
$$
Since $\partial_{x_3}U$ is continuous, it follows that
\begin{equation}\label{Neumann}
\partial_{x_3}U(x_\perp,0)=0
\qquad\text{for all }\quad x_\perp\in\R^2.
\end{equation}

Set $ U_1:=I(|u_1|^2)$. On the half-space $\{x\in\R^3:~x_3<0\}$, we have $u_\Omega=u_1$, so both $U$ and $U_1$ satisfy the same Poisson equation
\begin{equation*}
-\Delta U=4\pi |u_\Omega|^2\text{ \ and \ } -\Delta U_1=4\pi |u_\Omega|^2\text{ \ in \ }\{x\in\R^3:~x_3<0\}.
\end{equation*}
Moreover, $U_1$ is symmetric with respect to $x_3$, hence $\partial_{x_3}U_1(x_\perp, 0)=0$. Together with \eqref{Neumann}, this gives the same Neumann condition on the boundary.
Note that $U\to 0$ and $U_1\to 0$ as $|x|\to +\infty$.
Hence the solution of this Neumann problem is unique, we have $U=U_1$ in $\{x\in\R^3:~x_3<0\}$.
Since  $u_\Omega$ and $u_1$ satisfy
\begin{equation*}
-\Delta u_\Om+\Omega|x_\perp|^2u_\Om-Uu_\Om=\mu_\Om u_\Om\text{ \ in \ }\R^3,
\end{equation*}
and
\begin{equation*}
-\Delta u_1+\Omega|x_\perp|^2u_1-U_1u_1=\mu_1 u_1\text{ \ in \ }\R^3,
\end{equation*}
for some Lagrange multipliers $\mu_\Om,\mu_1\in\mathbb R$. On the open set $\{x\in\R^3:x_3<0\}$, we have $u_\Om=u_1$ and $U=U_1$. Since it follows from \eqref{trans12} $u_\Om$ cannot vanish identically on $\{x\in\R^3:x_3<0\}$, we obtain that $\mu_\Om=\mu_1$. Consequently, the difference $u_\Om-u_1$ and $U-U_1$ satisfy the linear system
\begin{equation*}
\begin{cases}
-\Delta(u_\Omega-u_1)+\bigl(\Omega|x_\perp|^2-\mu_1-U\bigr)(u_\Omega-u_1)-u_1(U-U_1)=0\text{ \ in \ }\R^3,\\
-\Delta(U-U_1)-4\pi(u_\Omega+u_1)(u_\Omega-u_1)=0\text{ \ in \ }\R^3,
\end{cases}
\end{equation*}
with $u_{\Omega}-u_1\equiv U-U_1\equiv0$ on the half-space $\{x\in\R^3:x_3<0\}$. By the unique continuation principle (see \cite{LM08} and the references therein), we conclude that $u_\Om-u_1\equiv0$ everywhere. Hence $u_{\Omega}=u_1$, i.e.,
\begin{equation*}
u_{\Omega}(x_\perp,x_3)=u_{\Omega}(x_\perp,-x_3)\text{\ \  in \ }\mathbb R^3.
\end{equation*}
This proves the desired symmetry of $u_\Omega(x_\bot, x_3)$ in the third direction.
Furthermore, similar to \cite[Theorem 2]{BBJV}, one can check that $u_\Omega$ is decreasing with respect to $|x_3|$. The proof of Proposition \ref{lem2.3} is therefore complete. \qed
\vskip 0.05truein
\noindent{\bf Proof of Theorem \ref{existence}.}
Combining the existence result above with Proposition \ref{lem2.3}, we complete
the proof of Theorem \ref{existence}. \qed

\section{Energy Estimates of $E_\Omega$ as $\Omega\to+\infty$}\label{S3}
This section is devoted to establishing a sharp energy estimate for
$E_\Omega$ defined in \eqref{inf}. This estimate is a crucial ingredient in the
proof of Theorem \ref{thm1.2}, since it determines the correct anisotropic
scales of the minimizers in the large-$\Omega$ limit. More precisely, we shall prove the following Theorem in this section.

\begin{Theorem}\label{thm3.1}
The ground state energy $E_{\Omega}$ satisfies
\begin{equation*}
E_{\Omega}=2\sqrt{\Omega}-\f{1}{192}(\ln \Omega)^2+\f{1}{24}(\ln \Omega)(\ln \ln \Omega)+O(\ln \Omega)\text{ \ as \ }\Omega\to+\infty.
\end{equation*}
\end{Theorem}

We first consider the lower bound estimate of $E_{{\Omega}}$ for some ${\Omega}>1$.
Motivated by \cite{FG}, define the projection operator $P_{{\Omega}}$ as follows
\begin{equation}\label{P}
(P_{{\Omega}}u)(x_\perp, x_3):=\psi_{{\Omega}}(x_\perp)\int_{\R^2}\psi_{{\Omega}}(y_\perp)u(y_\perp, x_3)dy_\perp \ \ \text{for any}\ \ u\in\H,
\end{equation}
where
\begin{equation}\label{ot}
\psi_{{\Omega}}(x_{\perp})=\frac{{\Omega}^\f14}{\sqrt{\pi}}
e^{-\frac{\sqrt{{\Omega}}|x_{\perp}|^2}{2}}.
\end{equation}
We also introduce the operator
\begin{equation}\label{A}
A_{{\Omega}}:=-\Delta_{\perp}+{\Omega}|x_{\perp}|^2\ \ \text{for}\ \ x_{\perp}\in\R^2.
\end{equation}
Moreover, we use the notation
$$\langle u, A_{{\Omega}} v\rangle:=\int_{\R^3}u A_{{\Omega}} v dx
=\int_{\R^3} \n_\perp u \cdot \n_\perp v dx+{\Omega}\int_{\R^3}|x_\perp|^2 u v dx,~u, v \in \H,$$
and
$\|\sqrt{A_{{\Omega}}} u\|^2_2:=\langle u, A_{{\Omega}} u\rangle$.
Note that $P_{{\Omega}}$ is an orthogonal projection,  and $P_{{\Omega}}$ commutes with $A_{{\Omega}}$ and $\partial_{x_3}$. Thus, one can check that for any $\phi\in \H$,
\begin{equation}\label{3.1}
\|\sqrt{A_{{\Omega}}} \phi\|^2_2=\|\sqrt{A_{{\Omega}}} P_{{\Omega}}\phi\|^2_2+\|\sqrt{A_{{\Omega}}} Q_{{\Omega}}\phi\|^2_2,
\end{equation}
and
\begin{equation}\label{3.2}
\|\partial_{x_3} \phi\|^2_2=\|\partial_{x_3} P_{{\Omega}}\phi\|^2_2+\|\partial_{x_3} Q_{{\Omega}}\phi\|^2_2,
\end{equation}
where
\begin{equation}\label{QO}
Q_{{\Omega}}=I-P_{{\Omega}}.
\end{equation}
Moreover, since $A_{{\Omega}} P_{{\Omega}}=2\sqrt{{\Omega}} P_{{\Omega}}$, we have
\begin{equation}\label{3.3}
\|\sqrt{A_{{\Omega}}} P_{{\Omega}}\phi\|_2^2=2\sqrt{{\Omega}}\|P_{{\Omega}}\phi\|_2^2.
\end{equation}
Additionally, the structure of the spectrum of $A_{{\Omega}}$ (cf. \cite[Section 8.3]{Teschl}) implies that
\begin{equation}\label{3.12}
\|\sqrt{A_{{\Omega}}} Q_{{\Omega}}\phi\|_2^2\geq4\sqrt{{\Omega}}\|Q_{{\Omega}}\phi\|_2^2.
\end{equation}

For any $\phi\in\H$ with $\|\phi\|_{2}=1$, it follows that
$\|P_{{\Omega}} \phi\|^2_2=1-\|Q_{{\Omega}}\phi\|^2_2$.
Using \cite[Lemma 4.3]{FG},
we then deduce from (\ref{functional}), (\ref{GN}), (\ref{3.1}), (\ref{3.2}) and  (\ref{3.3}) that for $0<\kappa\leq1$,
%\begin{equation}
\begin{align}
\E_{{\Omega}}(\phi)
=&\big\|\sqrt{A_{{\Omega}}} \phi\big\|^2_2+\|\partial_{x_3} \phi\|^2_2-\f12\iint_{\R^3\times\R^3}\f{|\phi(x)|^2|\phi(y)|^2}{|x-y|}dxdy\nonumber\\
\geq&\big\|\sqrt{A_{{\Omega}}} P_{{\Omega}}\phi\big\|^2_2+\big\|\sqrt{A_{{\Omega}}} Q_{{\Omega}}\phi\big\|^2_2
-\frac{1+\kappa}{2}\iint_{\R^3\times\R^3}\frac{|P_{{\Omega}} \phi(x)|^2|P_{{\Omega}} \phi(y)|^2}{|x-y|}dxdy\nonumber\\
&+\|\partial_{x_3} P_{{\Omega}}\phi\|^2_2+\|\partial_{x_3} Q_{{\Omega}}\phi\|^2_2
-\frac{C}{2}\kappa^{-3}\iint_{\R^3\times\R^3}\frac{|Q_{{\Omega}}\phi(x)|^2|Q_{{\Omega}}\phi(y)|^2}{|x-y|}dxdy\label{3.11}\\
\geq&2\sqrt{{\Omega}}\big(1-\|Q_{{\Omega}}\phi(x)\|^2_2\big)+\|\partial_{x_3} P_{{\Omega}}\phi\|^2_2-\frac{1+\kappa}{2}\iint_{\R^3\times\R^3}\frac{|P_{{\Omega}} \phi(x)|^2|P_{{\Omega}} \phi(y)|^2}{|x-y|}dxdy\nonumber\\
&+\big\|\sqrt{A_{{\Omega}}} Q_{{\Omega}}\phi\big\|^2_2+\|\partial_{x_3} Q_{{\Omega}}\phi\|^2_2
-\frac{C}{2}\kappa^{-3}\|Q_{{\Omega}}\phi\|_2^3\|\nabla Q_{{\Omega}}\phi\|_2,\nonumber
\end{align}
%\end{equation}
where $A_{{\Omega}}$ is as in (\ref{A}).
Let us first estimate the terms in (\ref{3.11}) that involve $P_{{\Omega}} \phi$ in the following lemma.

\begin{Lemma}\label{lemP}
For any $\phi\in\H$ with $\|\phi\|_{2}=1$,
 we have
\begin{align}
&\|\partial_{x_3} P_{{\Omega}} \phi\|_{2}^2- \f{1+\kappa_\Omega}{2 } \iint_{\R^3\times\R^3}\frac{|P_{{\Omega}} \phi(x)|^2|P_{{\Omega}} \phi(y)|^2}{|x-y|}dxdy \nonumber\\
\geq& -\f{1}{48}\big(\ln B_{{\Omega}}\big)^2+\f{1}{12}(\ln B_{{\Omega}})(\ln \ln B_{{\Omega}})-C\ln B_{{\Omega}}\ \ \text{as}\ \ {\Omega}\to+\infty,\nonumber
\end{align}
where $P_\Omega$ is as in \eqref{P}, and $\kappa_\Omega=\frac{1}{\ln B_{{\Omega}}}$ with   $B_{{\Omega}}:=2\sqrt{{\Omega}}$.
\end{Lemma}
\noindent{\bf Proof.}
By the similar argument of \cite[Lemma 4.4]{FG}, one can check that there is a constant $C>0$, independent of ${\Omega}$, such that for any $\phi\in\H$,
\begin{equation}\label{v}
R_{B_{{\Omega}}}^{(2)}(P_{{\Omega}} \phi)\leq C \|P_{{\Omega}} \phi\|_2^3\|\partial_{x_3} P_{{\Omega}} \phi\|_2,
\end{equation}
where \begin{equation}\label{3.10}
R_B^{(2)}(\phi):=\int_{\R}\iint_{\R^2\times\R^2}|\phi(x_{\perp}, x_3)|^2|\phi(y_{\perp}, x_3)|^2
\bar{K}_B(x_{\perp}-y_{\perp})dx_{\perp}dy_{\perp}dx_3,
\end{equation}
and $\bar{K}_B$ is given by
$$\bar{K}_B(x_{\perp})=\ln \big(1+\sqrt{1+(\ln B)^2|x_{\perp}|^2}\big)-\ln \big(\sqrt{B}|x_{\perp}|\big).$$
Note that $\|P_{{\Omega}} \phi\|_{2}\leq\|\phi\|_{2}=1$. Using \cite[Proposition 2.1]{FG} and \cite[Corollary 3.2]{FG}, one can calculate from \eqref{v} that for large ${\Omega}>0$,
\begin{align}
&\|\partial_{x_3} P_{{\Omega}} \phi\|_{2}^2- \f{1+\kappa_\Omega}{2 } \iint_{\R^3\times\R^3}\frac{|P_{{\Omega}} \phi(x)|^2|P_{{\Omega}} \phi(y)|^2}{|x-y|}dxdy\nonumber\\
\geq&(1-\kappa_\Omega)\|\partial_{x_3} P_{{\Omega}} \phi\|_{2}^2
-(1+\kappa_\Omega)C_{B_{{\Omega}}}\int_{\R}\Big(\int_{\R^2}|P_{{\Omega}} \phi(x_{\perp}, x_3)|^2dx_{\perp}\Big)^2dx_3\nonumber\\
&+\kappa_\Omega\|\partial_{x_3} P_{{\Omega}} \phi\|_{2}^2-C(1+\kappa_\Omega)\Big(\ln B_{{\Omega}}+\frac{1}{\sqrt{\ln B_{{\Omega}}}}\|\partial_{x_3} P_{{\Omega}} \phi\|_{2}^{\f32}+\|\partial_{x_3} P_{{\Omega}} \phi\|_2\Big)\label{313}\\
\geq&-\frac{(1+\kappa_\Omega)^2C_{B_{{\Omega}}}^2}{12(1-\kappa_\Omega)}\|P_{{\Omega}} \phi\|_2^6+\Big[\f{\kappa_\Omega}{2} \|\partial_{x_3} P_{{\Omega}} \phi\|_{2}^2-C\frac{1+\kappa_\Omega}{\sqrt{\ln B_{{\Omega}}}}\|\partial_{x_3} P_{{\Omega}} \phi\|_{2}^{\f32}\Big]\nonumber\\
&+\Big[\f{\kappa_\Omega}{2} \|\partial_{x_3} P_{{\Omega}} \phi\|_{2}^2-C(1+\kappa_\Omega) \|\partial_{x_3} P_{{\Omega}} \phi\|_{2} \Big]-C (1+\kappa_\Omega) \ln B_{{\Omega}},\nonumber
\end{align}
where $C_{B_{{\Omega}}}:=\frac{\ln B_{{\Omega}}}{2}-\ln\ln B_{{\Omega}}$.
In the last inequality of (\ref{313}), we  minimize the contents of the two brackets with respect to $\partial_{x_3} P_{{\Omega}} \phi$, respectively. Then we obtain that
\begin{align}
&\|\partial_{x_3} P_{{\Omega}} \phi\|_{2}^2- \f{1+\kappa_\Omega}{2 } \iint_{\R^3\times\R^3}\frac{|P_{{\Omega}} \phi(x)|^2|P_{{\Omega}} \phi(y)|^2}{|x-y|}dxdy\label{3.13}\\
\geq&\Big(-\frac{1}{12}-\frac{C}{\ln B_{{\Omega}}}\Big)C_{B_{{\Omega}}}^2 \|P_{{\Omega}} \phi\|_2^6-C\Big(1+\frac{1}{\ln B_{{\Omega}}}\Big)^4\ln B_{{\Omega}}-C\Big(1+\frac{1}{\ln B_{{\Omega}}}\Big)^2\ln B_{{\Omega}}-C\ln B_{{\Omega}}\nonumber\\
\geq&-\f{1}{48}\big(\ln B_{{\Omega}}\big)^2+\f{1}{12}(\ln B_{{\Omega}})(\ln \ln B_{{\Omega}})-C\ln B_{{\Omega}}\ \ \text{as}\ \ {\Omega}\to+\infty.\nonumber
\end{align}
The proof of Lemma \ref{lemP} is therefore complete.\qed

\vskip 0.05truein
We then  establish the following lower bound estimate of $E_{\Omega}$ as $\Omega\to +\infty$.

\begin{Proposition}\label{pro3.3}
The ground state energy $E_{\Omega}$ satisfies
\begin{equation*}
E_{\Omega}\geq2\sqrt{\Omega}-\f{1}{192}(\ln \Omega)^2+\f{1}{24}(\ln \Omega)(\ln \ln \Omega)-C\ln \Omega\ \  \text{as}\  \  \Omega\to+\infty,
\end{equation*}
where the constant $C>0$ is independent of $\Omega$.
\end{Proposition}
\noindent{\bf Proof.}
In view of \eqref{3.11} and Lemma \ref{lemP},
 we next estimate the terms that involve $Q_{{\Omega}}\phi$ in (\ref{3.11}). Note that
\begin{equation}\label{3.215}
 \|\nabla Q_{{\Omega}}\phi\|_2\leq \|\sqrt{A_{{\Omega}}} Q_{{\Omega}}\phi\|_2+\|\partial_{x_3} Q_{{\Omega}}\phi\|_2.
%\Big(\int_{\R^3}|\sqrt{A_{{\Omega}}} Q_{{\Omega}}\phi|^2dx\Big)^{\f12}+\Big(\int_{\R^3}|\partial_{x_3} Q_{{\Omega}}\phi|^2dx\Big).
\end{equation}
We then introduce another small parameter $0<\rho\leq 1$, and deduce from (\ref{3.12}) and (\ref{3.215}) that for sufficiently large ${\Omega}>0$,
\begin{align}
&\big\|\sqrt{A_{{\Omega}}} Q_{{\Omega}}\phi\big\|^2_2-2\sqrt{{\Omega}}\|Q_{{\Omega}}\phi\|^2_2
-\frac{C}{2}\kappa_\Omega^{-3}\|Q_{{\Omega}}\phi\|_2^3\|\nabla Q_{{\Omega}}\phi\|_2 \nonumber\\
\geq&\big\|\sqrt{A_{{\Omega}}} Q_{{\Omega}}\phi\big\|^2_2-2\sqrt{{\Omega}}\|Q_{{\Omega}}\phi\|^2_2-C \kappa_\Omega^{-3}\|Q_{{\Omega}} \phi\|_2^3\big\|\sqrt{A_{{\Omega}}} Q_{{\Omega}}\phi\big\|_2-C \kappa_\Omega^{-3}\|Q_{{\Omega}}\phi\|_2^3\|\partial_{x_3} Q_{{\Omega}} \phi\|_2\nonumber\\
=&(1-\rho)\big\|\sqrt{A_{{\Omega}}} Q_{{\Omega}}\phi\big\|^2_2-2\sqrt{{\Omega}}\|Q_{{\Omega}}\phi\|^2_2-C \kappa_\Omega^{-3}\|Q_{{\Omega}}\phi\|_2^3\|\partial_{x_3} Q_{{\Omega}}\phi\|_2\nonumber\\
&+\Big(\rho\big\|\sqrt{A_{{\Omega}}} Q_{{\Omega}}\phi\big\|^2_2-C \kappa_\Omega^{-3}\|Q_{{\Omega}} \phi\|_2^3\big\|\sqrt{A_{{\Omega}}} Q_{{\Omega}}\phi\big\|_2\Big)\nonumber\\
\geq&2\sqrt{{\Omega}}(1-2\rho)\|Q_{{\Omega}}\phi\|^2_2-C \kappa_\Omega^{-3}\|Q_{{\Omega}}\phi\|_2^3\|\partial_{x_3} Q_{{\Omega}}\phi\|_2\\
&+\Big(\rho\big\|\sqrt{A_{{\Omega}}} Q_{{\Omega}}\phi\big\|^2_2-C \kappa_\Omega^{-3}\|Q_{{\Omega}} \phi\|_2^3\big\|\sqrt{A_{{\Omega}}} Q_{{\Omega}}\phi\big\|_2\Big)\nonumber\\
\geq&2\sqrt{{\Omega}}(1-2\rho)\|Q_{{\Omega}}\phi\|^2_2-C \kappa_\Omega^{-6}\|Q_{{\Omega}}\phi\|^6_2\rho^{-1}-C \kappa_\Omega^{-3}\|Q_{{\Omega}} \phi\|_2^3\|\partial_{x_3} Q_{{\Omega}}\phi\|_2\nonumber\\
\geq&2\sqrt{{\Omega}}\|Q_{{\Omega}}\phi\|^2_2-C {\Omega}^{\f14}\kappa_\Omega^{-3}\|Q_{{\Omega}}\phi\|^4_2-C \kappa_\Omega^{-3}\|Q_{{\Omega}} \phi\|_2^3\|\partial_{x_3} Q_{{\Omega}}\phi\|_2.\nonumber
\end{align}
Here the last two inequalities follow from minimizing with respect to $\|\sqrt{A_{{\Omega}}} Q_{{\Omega}} \phi\|_2$ and choosing $\rho=\frac{\|Q_{{\Omega}}\phi\|^2_2}{\kappa_\Omega^{3}{\Omega}^{\f14}}$, respectively.
To estimate the remaining terms of (\ref{3.11}), we note that
$$\|\partial_{x_3} Q_{{\Omega}}\phi\|_2^2-C \kappa_\Omega^{-3}\|Q_{{\Omega}}\phi\|_2^3\|\partial_{x_3} Q_{{\Omega}}\phi\|_2\geq -C \kappa_\Omega^{-6} \|Q_{{\Omega}}\phi\|_2^6$$
by minimizing with respect to $\|\partial_{x_3} Q_{{\Omega}}\phi\|_2$.
Thus, for sufficiently large $\Omega>0$, all terms in (\ref{3.11}) that involve $Q_{{\Omega}}\phi$ satisfy
\begin{align}
&\big\|\sqrt{A_{{\Omega}}} Q_{{\Omega}}\phi\big\|^2_2-2\sqrt{{\Omega}}\|Q_{{\Omega}}\phi\|^2_2+\|\partial_{x_3} Q_{{\Omega}}\phi\|_2^2
-\frac{C}{2}\kappa_\Omega^{-3}\|Q_{{\Omega}}\phi\|_2^3\|\nabla Q_{{\Omega}}\phi\|_2\nonumber\\
\geq&2\sqrt{{\Omega}}\|Q_{{\Omega}}\phi\|^2_2-C {\Omega}^{\f14}\kappa_\Omega^{-3}\|Q_{{\Omega}}\phi\|^4_2-C \kappa_\Omega^{-6} \|Q_{{\Omega}}\phi\|_2^6\label{3.14}\\
\geq&\|Q_{{\Omega}}\phi\|_2^2\Big[2\sqrt{{\Omega}}-C {\Omega}^{\f14} \Big(\ln \big(2\sqrt{{\Omega}}\big)\Big)^3-C\Big(\ln \big(2\sqrt{{\Omega}}\big)\Big)^6\Big]
\geq 0,\nonumber
\end{align}
where $\|Q_{{\Omega}}\phi\|_2\leq1$ and $\kappa_\Omega=\frac{1}{\ln \big(2\sqrt{{\Omega}}\big)}$ are used.

Thus, we conclude from (\ref{3.11}), (\ref{3.13}) and (\ref{3.14}) that
 %for sufficiently large ${\Omega}$,
\begin{align}
\E_{{\Omega}}(\phi)\geq&2\sqrt{{\Omega}}-\f{1}{48}\big(\ln B_{{\Omega}}\big)^2+\f{1}{12}(\ln B_{{\Omega}})(\ln \ln B_{{\Omega}})-C\ln B_{{\Omega}}\label{3.19}\\
\geq&2\sqrt{{\Omega}}-\f{1}{192}(\ln {\Omega})^2+\f{1}{24}(\ln {\Omega})(\ln \ln {\Omega})-C\ln {\Omega} \ \ \text{  as  }\  \ {\Omega}\to+\infty,\nonumber
\end{align}
where $B_{{\Omega}}=2\sqrt{{\Omega}}$ is used.
The proof of Proposition \ref{pro3.3} is therefore complete.\qed

\vskip 0.05truein

We are now ready to prove Theorem \ref{thm3.1}.
\vskip 0.05truein

\noindent{\bf Proof of Theorem \ref{thm3.1}.} In view of Proposition \ref{pro3.3}, it suffices to prove that
\begin{equation}\label{3.195}
E_{\Omega}\leq2\sqrt{\Omega}-\f{1}{192}(\ln \Omega)^2+\f{1}{24}(\ln \Omega)(\ln \ln \Omega)+C\ln \Omega\ \  \text{as}\  \  \Omega\to+\infty,
\end{equation}
where the constant $C>0$ is independent of $\Omega$. Choose a test function
\begin{equation*}
\varphi_{{\Omega}}(x)=\psi_{{\Omega}}(x_{\perp})\frac{\sqrt{|\ln B_{{\Omega}}|}}{2\sqrt{2}}\Big(\cosh \frac{|\ln B_{{\Omega}}|x_3}{4}\Big)^{-1}
:=\psi_{{\Omega}}(x_{\perp}) \tilde{g}_{{\Omega}}(x_3), \ \ x=(x_{\perp}, x_3) \in \R^3,
\end{equation*}
where $B_{{\Omega}}=2\sqrt{{\Omega}}$, and $\psi_{{\Omega}}(x_{\perp})$ is as in (\ref{ot}). One can check that for any ${\Omega}>1$,
$$\int_{\R^3}\varphi_{{\Omega}}^2dx=1\ \ \text{and}\ \ \langle \varphi_{{\Omega}},~A_{{\Omega}} \varphi_{{\Omega}}\rangle=2\sqrt{{\Omega}}.$$
Note that
$
\tilde{g}_{{\Omega}}(x_3)=\Big(\int_{\R^2}\varphi_{{\Omega}}^2(x_{\perp}, x_3)dx_{\perp}\Big)^{\f12}$ for $x=(x_{\perp}, x_3)\in\R^3$.
One can calculate that
\begin{equation*}
\begin{aligned}
&\|\partial_{x_3}\varphi_{{\Omega}}\|_2^2-\frac{\ln B_{{\Omega}}}{2}\int_{\R}\Big(\int_{\R^2}\varphi_{{\Omega}}^2(x_{\perp}, x_3)dx_{\perp}\Big)^{2}dx_3\\
=&\int_{\R}|\tilde{g}'_{{\Omega}}(x_3)|^2dx_3-\frac{\ln B_{{\Omega}}}{2}\int_{\R}|\tilde{g}_{{\Omega}}(x_3)|^4dx_3=-\frac{(\ln B_{{\Omega}})^2}{48}.
\end{aligned}
\end{equation*}

Thus, we deduce from \cite[Proposition 2.1]{FG} that for sufficiently large ${\Omega}>1$,
\begin{align}
E_{{\Omega}}\leq&\E_{{\Omega}}(\varphi_{{\Omega}})\label{3.15}\\
\leq&2\sqrt{{\Omega}}-\frac{(\ln B_{{\Omega}})^2}{48}+(\ln \ln B_{{\Omega}})\int_{\R}\Big(\int_{\R^2}\varphi_{{\Omega}}^2(x_{\perp}, x_3)dx_{\perp}\Big)^2dx_3 +|R_{B_{{\Omega}}}^{(1)}(\varphi_{{\Omega}})|
-R_{B_{{\Omega}}}^{(2)}(\varphi_{{\Omega}}),\nonumber
\end{align}
where
\begin{equation*}
|R_{B_{\Omega}}^{(1)}(\varphi_{{\Omega}})|\leq\frac{\ln B_{\Omega}}{2}\|\varphi_{{\Omega}}\|_2^4+\frac{4}{\sqrt{\ln B_{\Omega}}}\|\varphi_{{\Omega}}\|_2^{\frac{5}{2}}\|\partial_{x_3} \varphi_{{\Omega}}\|_2^{\frac{3}{2}},
\end{equation*}
and $R_{B_{\Omega}}^{(2)}(\varphi_{{\Omega}})$ is defined as in \eqref{3.10}.
Note that
$\|\partial_{x_3} \varphi_{{\Omega}}\|_2^2=\int_{\R}|\tilde{g}'_{{\Omega}}(x_3)|^2dx_3= \frac{(\ln B_{{\Omega}})^2}{48}. $
It then follows that
%\begin{equation}\label{3.17}
$|R_{B_{{\Omega}}}^{(1)}(\varphi_{{\Omega}})|\leq C \ln B_{{\Omega}}$.
%\end{equation}
One can also calculate that
\begin{align}
R_{B_{\Omega}}^{(2)}(\varphi_{\Omega})\geq&-\int_{\R}\iint_{\R^2\times\R^2}
|\varphi_{\Omega}(x_{\perp}, x_3)|^2|\varphi_{\Omega}(y_{\perp}, x_3)|^2\ln_+ \big(\sqrt{B_{{\Omega}}}|x_{\perp}-y_{\perp}|\big)dx_{\perp}dy_{\perp}dx_3\nonumber\\
=&-\int_{\R} \tilde{g}^4_{{\Omega}}(x_3)dx_3\iint_{\R^2\times\R^2}|\psi_{{\Omega}}(x_{\perp})|^2
|\psi_{\Omega}(y_{\perp})|^2\ln_+ \big(\sqrt{B_{\Omega}}|x_{\perp}-y_{\perp}|\big)dx_{\perp}dy_{\perp} \nonumber\\
\geq&-\frac{\ln B_{\Omega}}{12 \pi^{2}}\iint_{\R^2\times\R^2}e^{-(|x_{\perp}|^2+|y_{\perp}|^2)}
\ln\big[\sqrt{2}(|x_{\perp}|+|y_{\perp}|)+1\big]dx_{\perp}dy_{\perp}\label{3.18}\\
\geq& -C\ln B_{\Omega}\ \  \text{as}\  \  \Omega\to+\infty.\nonumber
\end{align}
We then derive from (\ref{3.15}) and (\ref{3.18}) that for sufficiently large ${\Omega}$,
\begin{equation}\label{3.20}
E_{{\Omega}}\leq 2\sqrt{{\Omega}}-\frac{1}{48}(\ln B_{{\Omega}})^2+\f{1}{12}(\ln B_{{\Omega}})(\ln \ln B_{{\Omega}})+C \ln B_{{\Omega}}.
\end{equation}
Together with Proposition \ref{pro3.3}, we finally conclude that
$$E_{\Omega}=2\sqrt{\Omega}-\f{1}{192}(\ln \Omega)^2+\f{1}{24}(\ln \Omega)(\ln \ln \Omega)+O(\ln \Omega)\ \  \text{as}\  \  \Omega\to+\infty,$$
and the proof of Theorem \ref{thm3.1} is therefore complete.\qed

\section{Limiting Behavior of $u_\Omega$ as $\Omega\to+\infty$}\label{S4}
We now turn from the energy asymptotics to the limiting profile of
minimizers.
Motivated by Theorem \ref{thm3.1}, define
\begin{equation}\label{scaling}
w_\Omega(x):=\Omega^{-\f14}(\ln \Omega)^{-\f12} u_{\Omega}\big( \Omega^{-\f14}x_\perp,
( \ln \Omega)^{-1}x_3\big),
\end{equation}
where $x=(x_\perp,x_3)\in\R^2\times\R$, and $u_\Omega$ is a minimizer of $E_\Omega$ defined by \eqref{inf}.
To derive the convergence, we first project in the first two coordinates onto the ground state eigenspace of $A_1$, where $A_1$ is as in \eqref{A}.
The following proposition indicates that the first two coordinates of $w_\Omega$ are nearly lying in the  ground state eigenspace  of $A_1$.
\begin{Proposition}
\label{lem4.1}
For every $\Omega>1$, let $u_\Omega\in \H$ be the corresponding minimizer of \eqref{inf}. Then the function $w_\Omega$ defined by \eqref{scaling} satisfies
\begin{equation}\label{w}
\|Q_1w_\Omega\|_{\H}=\|w_\Omega-P_1w_\Omega\|_{\H}\to0\ \ \text{as}\ \ \Omega\to+\infty,
%w_\Omega=P_1w_\Omega+\varepsilon_\Omega,
\end{equation}
where $P_1$ and $Q_1$ are defined by \eqref{P} and \eqref{QO}, respectively.
%, and the function $\varepsilon_\Omega$ satisfies $\|\varepsilon_\Omega\|_{\H}\to0$ as $$.
\end{Proposition}
\noindent{\bf Proof.} Let $u_\Omega\in \H$ be a minimizer of $E_\Omega$ defined by \eqref{inf}.
%It follows from Lemma \ref{lem3.2} that $u_\Omega:=\alpha^{-\f32}u_\Omega(\alpha^{-1}x)$ is a minimizer of $E_{\Omega\alpha^{-4}, 1}$.
By Lemma \ref{lemP}, we obtain from  \eqref{3.11}, \eqref{3.14} and \eqref{3.20} that there exists a constant $C>0$, independent of $\Omega$, such that for sufficiently large $\Omega>0$,
\begin{align}
 0\leq&\big\|\sqrt{A_{\Omega}} Q_{\Omega}u_\Omega\big\|_2^2-2\sqrt{\Omega}\|Q_{\Omega}u_\Omega\|_2^2+\|\partial_{x_3}Q_{\Omega} u_\Omega\|_2^2-\frac{C}{2}\Big(\ln \big(2\sqrt{\Omega}\big)\Big)^3 \|Q_{\Omega}u_\Omega\|_2^3 \|\nabla Q_{\Omega}u_\Omega\|_2 \nonumber\\
\leq & C \ln \Omega,\label{4.1}
\end{align}
where $Q_{\Omega}=I-P_{\Omega}$.
%Note that
%$$\frac{C}{2}\Big(\ln \big(2\sqrt{\Omega}\big)\Big)^3 \|Q_{\Omega}u_\Omega\|_2^3 \|\nabla Q_{\Omega}u_\Omega\|_2\leq\f18 \|\nabla Q_{\Omega}u_\Omega\|_2^2+\frac{C^2}{2}\Big(\ln \big(2\sqrt{\Omega}\big)\Big)^6 \|Q_{\Omega}u_\Omega\|_2^6.$$
By Young's inequality, it then follows from \eqref{3.12}  and \eqref{4.1} that
%\begin{align}
%&\big\|\sqrt{A_{\Omega}} Q_{\Omega}u_\Omega\big\|_2^2-2\sqrt{\Omega}\|Q_{\Omega}u_\Omega\|_2^2+\|\partial_{x_3}Q_{\Omega} u_\Omega\|_2^2-\frac{C}{2}\Big(\ln \big(2\sqrt{\Omega}\big)\Big)^3 \|Q_{\Omega}u_\Omega\|_2^3 \|\nabla Q_{\Omega} u_\Omega\|_2\nonumber\\
%\geq&\f78\big\|\sqrt{A_{\Omega}} Q_{\Omega}u_\Omega\big\|_2^2-2\sqrt{\Omega}\|Q_{\Omega} u_\Omega\|_2^2+\f78\|\partial_{x_3}Q_{\Omega}u_\Omega\|_2^2-C\Big(\ln \big(2\sqrt{\Omega}\big)\Big)^6 \|Q_{\Omega} u_\Omega\|_2^6\nonumber\\
%\geq&\f18\big\|\sqrt{A_{\Omega}} Q_{\Omega}u_\Omega\big\|_2^2+\sqrt{\Omega}\|Q_{\Omega} u_\Omega\|_2^2+\f78\|\partial_{x_3}Q_{\Omega}u_\Omega\|_2^2-C\Big(\ln \big(2\sqrt{\Omega}\big)\Big)^6 \|Q_{\Omega} u_\Omega\|_2^6\nonumber\\
%\geq&\f18\big\|\sqrt{A_{\Omega}} Q_{\Omega}u_\Omega\big\|_2^2+\Big[\sqrt{\Omega}-C\Big(\ln \big(2\sqrt{\Omega}\big)\Big)^6\Big]\|Q_{\Omega}u_\Omega\|_2^2+\f78\|\partial_{x_3}Q_{\Omega}u_\Omega\|_2^2
%\geq 0\ \ \text{as}\ \ \Omega\to+\infty,\nonumber
%\end{align}
%where $\|Q_{\Omega}u_\Omega\|_2\leq\|u_\Omega\|_2=1$ is used. This implies that
\begin{equation}\label{4.2}
\big\|\sqrt{A_{\Omega}} Q_{\Omega}u_\Omega\big\|_2^2+\sqrt{\Omega}\|Q_{\Omega}u_\Omega\|_2^2+\|\partial_{x_3}Q_{\Omega} u_\Omega\|_2^2\leq C\ln \Omega \ \ \text{as}\ \ \Omega\to+\infty.
\end{equation}
One can further calculate from (\ref{scaling}) and \eqref{4.2} that
\begin{align}
%&\int_{\R^3}\big|\nabla \big(u_\Omega-P_{\Omega} %u_\Omega\big)\big|^2dx+\Omega\int_{\R^3}|x_\perp|^2\big|u_\Omega-P_{\Omega} u_\Omega\big|^2dx+\sqrt{\Omega}\int_{\R^3}\big|u_\Omega-P_{\Omega} u_\Omega\big|^2dx\nonumber\\
%=&\sqrt{\Omega}\int_{\R^3}\Big|\nabla_{\perp}\Big(w_\Omega(z_{\perp}, z_3)-\psi_{\Omega}\big(\Omega^{-\f14}z_{\perp}\big)\int_{\R^2}\psi_{\Omega}(y_\perp)
%w_\Omega\big(\Omega^{\f14}y_{\perp}, z_3\big)dy_\perp\Big)\Big|^2dz\label{4.4}\\
%&+\big(\ln \Omega\big)^2\int_{\R^3}\Big|\partial_{x_3}\Big(w_\Omega(z_{\perp}, z_3)-\psi_{\Omega}\big(\Omega^{-\f14}z_{\perp}\big)\int_{\R^2}\psi_{\Omega}(y_\perp)
%w_\Omega\big(\Omega^{\f14}y_{\perp}, z_3\big)dy_\perp\Big)\Big|^2dz\nonumber\\
%&+\sqrt{\Omega}\int_{\R^3}\big(|z_{\perp}|^2+1\big)\Big|w_\Omega(z_{\perp}, z_3)-\psi_{\Omega}\big(\Omega^{-\f14}z_{\perp}\big)\int_{\R^2}\psi_{\Omega}(y_\perp)w_\Omega\big(\Omega^{\f14}y_{\perp}, z_3\big)dy_\perp\Big|^2dz\nonumber\\
&\sqrt{\Omega}\int_{\R^3}|\nabla_{\perp} Q_1w_\Omega|^2dz+\big(\ln \Omega\big)^2\int_{\R^3}|\partial_{z_3} Q_1w_\Omega|^2dz
+\sqrt{\Omega}\int_{\R^3}\big(|z_{\perp}|^2+1\big)|Q_1w_\Omega|^2dz\nonumber\\
\leq& C\ln \Omega \quad \text{as} \quad \Omega\to+\infty,\label{4.4}
\end{align}
 where $Q_1=I-P_{1}$.
%, and the fact
%$$\psi_{\Omega}(\Omega^{-\f14}z_{\perp})\int_{\R^2}\psi_{\Omega}(y_\perp)w_\Omega\big(\Omega^{\f14}y_{\perp}, z_3\big)dy_\perp
%=\psi_{1}(z_{\perp})\int_{\R^2}\psi_{1}(y_{\perp})w_\Omega(y_{\perp}, z_3)dy_\perp:=P_{1}w_\Omega$$
%is used in the last identity.
We thus conclude from (\ref{4.4}) that
\begin{equation*}
\int_{\R^3}\big|\nabla Q_1w_\Omega\big|^2dz+\int_{\R^3}(1+|z_{\perp}|^2)\big|Q_1w_\Omega\big|^2dz\to0 \ \ ~\text{as}\ \ \Omega\to+\infty.
\end{equation*}
The proof of Proposition \ref{lem4.1} is therefore complete.
\qed

\vskip 0.05truein
We next prove a form estimate for the Coulomb kernel. It allows us to control
the singular interaction by the positive operator
$A_\Omega-2\sqrt{\Omega}-\partial_{x_3}^2$, up to a logarithmic error.

\begin{Lemma}\label{lemB}
For any $\lambda>0$ and $0<\eta<1$, there exists a constant $C_{\lambda, \eta}>0$, independent of sufficiently large $\Omega>0$, such that
$$\frac{\lambda}{|x|}\leq \eta\big( A_\Omega-2\sqrt{\Omega}-\partial_{x_3}^2 \big) + C_{\lambda,\eta} (\ln \Omega)^2$$
in the sense of quadratic forms on $L^2(\R^3)$, where $A_\Omega$ is as in \eqref{A}.
\end{Lemma}
\noindent{\bf Proof.} For any $\phi\in\H$, we have $\phi=P_\Omega \phi+Q_\Omega \phi$, where $P_\Omega$ and $Q_\Omega$ are as in \eqref{P} and \eqref{QO}, respectively. Denote $P_\Omega\phi(x_{\perp}, x_{3}):=\psi_{\Omega}(x_{\perp})h_\Omega(x_3)$ and $\tilde{v}_{\Omega}(x_3):=\int_{\R^2}\frac{|\psi_{\Omega}(x_{\perp})|^2}{\sqrt{|x_{\perp}|^2+x_3^2}}dx_{\perp}$, where $\psi_{\Omega}$ is as in \eqref{ot}. Since $\psi_{\Omega}(x_{\perp})=\Omega^{\frac{1}{4}}\psi_{1}(\Omega^{\frac{1}{4}}x_{\perp})$, we have
\begin{equation}\label{B2}
\tilde{v}_{\Omega}(x_3)=\Omega^{\frac{1}{4}}\int_{\R^2}\frac{|\psi_{1}(x_{\perp})|^2}{\sqrt{|x_{\perp}|^2
+\Omega^{\frac{1}{2}}x_3^2}}dx_{\perp}\leq \frac{C}{\sqrt{x_3^2+\Omega^{-\frac{1}{2}}}},
\end{equation}
where $\int_{\R^2}\frac{|\psi_{1}(x_{\perp})|^2}{\sqrt{|x_{\perp}|^2+t^2}}dx_{\perp}\leq \frac{C}{\sqrt{1+t^2}}$ is used. Moreover, one can check that
\begin{equation}\label{B3}
\int_{|x_3|\leq1}\frac{dx_3}{\sqrt{x_3^{2}+\Omega^{-\f12}}}\leq C \ln \Omega,
\end{equation}
and
\begin{equation}\label{B4}
\|h_\Omega(x_3)\|_{L^\infty(\R^1)}^2\leq 2\|h_\Omega(x_3)\|_{L^2(\R^1)}\|h'_\Omega(x_3)\|_{L^2(\R^1)}.
\end{equation}
It then follows from \eqref{B2}--\eqref{B4} that for any $0<s<1$,
\begin{align}
\int_{\R^3}\frac{|P_\Omega \phi(x)|^2}{|x|}dx=&\int_{\R}\tilde{v}_\Omega(x_3)|h_\Omega(x_3)|^2dx_3
\leq C\int_{\R}\frac{|h_\Omega(x_3)|^2}{\sqrt{x_3^2+\Omega^{-\f12}}}dx_3\nonumber\\
\leq&C\Big(\int_{|x_3|\leq 1}\frac{|h_\Omega(x_3)|^2}{\sqrt{x_3^2+\Omega^{-\f12}}}dx_3+\|h_\Omega\|_{L^2(\R^1)}^2\Big)\nonumber\\
\leq&C (\ln \Omega) \|h_\Omega\|_{L^2(\R^1)}\|h'_\Omega\|_{L^2(\R^1)}+C \|h_\Omega\|_{L^2(\R^1)}^2\nonumber\\
\leq& s\|h'_\Omega\|_{L^2(\R^1)}^2+C\Big(\f1s, 1\Big)(\ln \Omega)^2 \|h_\Omega\|_{L^2(\R^1)}^2,\nonumber
\end{align}
which implies that
\begin{equation}\label{B5}
\Big\langle P_{\Omega}\phi, \f{1}{|x|}P_{\Omega}\phi\Big\rangle\leq s\|\partial_{x_3}P_{\Omega}\phi\|_{2}^2+C\Big(\f1s, 1\Big)(\ln \Omega)^2 \|P_{\Omega}\phi\|_{2}^2.
\end{equation}

We next estimate $Q_{\Omega}\phi$. Note that for every $0<\varepsilon<1$,
$$\int_{\R^3}\frac{|u(x)|^2}{|x|}dx\leq \varepsilon\|\nabla u\|_2^2+\f1\varepsilon\|u\|_2^2,\ \ u\in H^1(\R^3).$$
Applying this estimate to $u=Q_\Omega \phi$, we have
\begin{equation}\label{B6}
\int_{\R^3}\frac{|Q_\Omega \phi|^2}{|x|}dx\leq \varepsilon\|\nabla Q_\Omega \phi\|_2^2+\f1\varepsilon\|Q_\Omega \phi\|_2^2.
\end{equation}
Since $A_\Omega-2\sqrt{\Omega}\geq2\sqrt{\Omega}$ on $\operatorname{Ran}Q_\Omega$, we derive that
\begin{equation}\label{B7}
\|\nabla_{x_{\perp}} Q_\Omega \phi\|_2^2\leq\langle Q_\Omega \phi, A_\Omega Q_\Omega \phi\rangle
\leq2\langle Q_\Omega \phi, (A_\Omega-2\sqrt{\Omega}) Q_\Omega \phi\rangle.
\end{equation}
It follows from \eqref{B6} and \eqref{B7} that
\begin{equation}\label{B8}
\int_{\R^3}\frac{|Q_\Omega \phi|^2}{|x|}dx\leq\varepsilon\Big[2\big\langle Q_\Omega \phi, (A_\Omega-2\sqrt{\Omega}) Q_\Omega \phi\big\rangle+\|\partial_{x_3} Q_\Omega \phi\|_2^2\Big]+\f1\varepsilon\|Q_\Omega \phi\|_2^2.
\end{equation}
For the mixed term, it can be calculated that for any $0<\delta<1$,
\begin{align}
\Big|\int_{\R^3}\frac{P_\Omega \phi(x) Q_\Omega \phi(x)}{|x|}dx\Big|\leq&\Big(\int_{\R^3}\frac{|P_\Omega \phi(x)|^2}{|x|}dx\Big)^{\f12}\Big(\int_{\R^3}\frac{|Q_\Omega \phi(x)|^2}{|x|}dx\Big)^{\f12}\nonumber\\
\leq&\delta \int_{\R^3}\frac{|P_\Omega \phi(x)|^2}{|x|}dx+\f{1}{4\delta}\int_{\R^3}\frac{|Q_\Omega \phi(x)|^2}{|x|}dx.\label{B9}
\end{align}
Combining (\ref{B5}), (\ref{B8}) and (\ref{B9}), we conclude that for any $\lambda>0$ and $0<\eta<1$,
$$\lambda\int_{\R^3}\frac{|\phi(x)|^2}{|x|}dx\leq\eta\Big[\Big\langle\phi, \big(A_\Omega-2\sqrt{\Omega}\big)\phi\Big\rangle+\|\partial_{x_3}\phi\|_2^2\Big]+C_{\lambda, \eta}(\ln \Omega)^2 \|\phi\|_2^2,$$
and the proof of Lemma \ref{lemB} is therefore complete.
\qed

\vskip 0.05truein
With the aid of Lemma \ref{lemB}, we then study the lowest-mode component
$P_1w_\Omega$.
The following lemma
shows that its associated one-dimensional profile is bounded uniformly  in
 $H^1(\R)$ as $\Omega\to+\infty$.

\begin{Lemma}\label{lem4.2}
Suppose $P_1$ and $w_\Omega$ are defined by \eqref{P} and \eqref{scaling}, respectively. Denote
\begin{equation}\label{c1}
P_1w_\Omega(x_{\perp}, x_3):=\psi_1 (x_{\perp})v_\Omega(x_3),
\end{equation}
where
\begin{equation}\label{4.15}
v_\Omega(x_3):=\int_{\R^2}\psi_1 (y_{\perp})w_\Omega(y_{\perp}, x_3)dy_{\perp}.
\end{equation}
Then $\{v_\Omega\}$ is bounded uniformly in $H^1(\R)$ as $\Omega\to+\infty$.
\end{Lemma}
\noindent{\bf Proof.} Suppose $u_\Omega$ is a minimizer of (\ref{inf}). It can be calculated that
$$P_1w_\Omega(x_{\perp}, x_3)=\Omega^{-\f14}(\ln \Omega)^{-\f12}P_\Omega u_\Omega\big(\Omega^{-\f14}x_{\perp}, (\ln \Omega)^{-1}x_3\big) \ \ \text{for}\ \ x=(x_{\perp}, x_3)\in\R^3.$$
Denoting $P_\Omega u_\Omega(x_{\perp}, x_3):=\psi_\Omega (x_{\perp})f_\Omega(x_3)$, we then have $v_\Omega(x_3)=(\ln \Omega)^{-\f12}f_{\Omega}\big((\ln \Omega)^{-1}x_3\big)$.
Since
$$\|v_{\Omega}\|_{L^2(\R)}=\|f_{\Omega}\|_{L^2(\R)}=\|P_{\Omega}u_{\Omega}\|_{2} \leq \|u_{\Omega}\|_{2}=1,$$
and
$$\|v'_{\Omega}\|_{L^2(\R)}^2=(\ln \Omega)^{-2}\|f'_{\Omega}\|_{L^2(\R)}^2
=(\ln \Omega)^{-2}\|\partial_{x_3}(P_{\Omega}u_{\Omega})\|_{2}^2 \leq (\ln \Omega)^{-2}\|\partial_{x_3} u_{\Omega}\|_{2}^2,$$
it suffices to prove that there exists a constant $C>0$, independent of $\Omega$, such that
\begin{equation}\label{4.5}
\|\partial_{x_3} u_{\Omega}\|_{2}^2\leq C (\ln \Omega)^{2}\quad \text{as} \quad \Omega\to+\infty.
\end{equation}
For any $u\in H^1(\R^3)$, denote
\begin{equation}\label{4.9}
B_\Omega(u)=\big\langle u, \big (A_\Omega-2\sqrt{\Omega}\big)u \big\rangle\geq0\ \ \text{and}\ \ G(u)=\|\partial_{x_3} u\|_{2}^2,
\end{equation}
where $A_\Omega$ is as in \eqref{A}. It then follows from \eqref{inf} that
\begin{equation}\label{4.12}
E_\Omega-2\sqrt{\Omega}=B_\Omega(u_\Omega)+G(u_\Omega)-\f12 D\big(|u_\Omega|^2, |u_\Omega|^2\big),
\end{equation}
where $D\big(|u_\Omega|^2, |u_\Omega|^2\big)=\iint_{\R^3\times\R^3}\f{|u_\Omega(x)|^2|u_\Omega(y)|^2}{|x-y|}dxdy$.

Next, we prove that for any $u\in H^1(\R^3)$ with $\|u\|_2=1$ and $0<\varepsilon<1$,
\begin{equation}\label{4.10}
\f12 D\big(|u|^2, |u|^2\big)\leq\varepsilon B_\Omega(u)+\varepsilon G(u)+C_\varepsilon(\ln \Omega)^{2}.
\end{equation}
Let $x=(x_\perp,x_3)\in\mathbb{R}^3,~y=(y_\perp,y_3)\in\mathbb{R}^3$, and introduce the coordinates
\begin{equation}\label{4.6}
R=\frac{x+y}{\sqrt{2}}, \qquad  r=\frac{x-y}{\sqrt{2}},
\end{equation}
where $R=(R_\perp,R_3)\in\mathbb{R}^3$ and $r=(r_\perp,r_3)\in\mathbb{R}^3$.
Then
\begin{align}
&\big(-\Delta_{x_\perp}+\Omega|x_\perp|^2-2\sqrt{\Omega}-\partial_{x_3}^2\big)
+\big(-\Delta_{y_\perp}+\Omega|y_\perp|^2-2\sqrt{\Omega}-\partial_{y_3}^2\big)\nonumber\\
=&\big(-\Delta_{R_\perp}+\Omega|R_\perp|^2-2\sqrt{\Omega}-\partial_{R_3}^2\big)
+\big(-\Delta_{r_\perp}+\Omega|r_\perp|^2-2\sqrt{\Omega}-\partial_{r_3}^2\big)\label{4.7}\\
\geq&-\Delta_{r_\perp}+\Omega|r_\perp|^2-2\sqrt{\Omega}-\partial_{r_3}^2. \nonumber
\end{align}
Moreover, it yields from Lemma \ref{lemB} that for any $\lambda>0$ and $0<\eta<1$,
\begin{equation}\label{r}
\frac{\lambda}{|r|}\leq \eta\big(-\Delta_{r_\perp}+\Omega|r_\perp|^2-2\sqrt{\Omega}-\partial_{r_3}^2 \big) + C_{\lambda,\eta} (\ln \Omega)^2
\end{equation}
in the sense of quadratic forms on $L^2(\R^3)$.
Choosing $\lambda=\frac{1}{2\sqrt{2}}$, we further obtain from (\ref{4.6})--\eqref{r} that
\begin{align}
&\frac{1}{2|x-y|} \label{4.8}\\
\leq& \eta \Big[\big(-\Delta_{x_\perp}+\Omega|x_\perp|^2-2\sqrt{\Omega}-\partial_{x_3}^2\big)
+\big(-\Delta_{y_\perp}+\Omega|y_\perp|^2-2\sqrt{\Omega}-\partial_{y_3}^2\big)\Big]+C_{\eta} (\ln \Omega)^2. \nonumber
\end{align}
It follows from (\ref{4.9}) and \eqref{4.8} that
\begin{align}
\f12 D\big(|u|^2, |u|^2\big)=&\big\langle u(x)u(y), \frac{1}{2|x-y|}u(x)u(y)\big\rangle\nonumber\\
\leq&\Big\langle u(x)u(y), \eta\Big[\big(-\Delta_{x_\perp}+\Omega|x_\perp|^2-2\sqrt{\Omega}-\partial_{x_3}^2\big)
\label{4.11}\\
&+\big(-\Delta_{y_\perp}+\Omega|y_\perp|^2-2\sqrt{\Omega}-\partial_{y_3}^2\big)\Big]u(x)u(y)\Big\rangle+C_{\eta} (\ln \Omega)^2\nonumber\\
=&2 \eta \big[B_{\Omega}(u)+G(u)\big]+C_{\eta} (\ln \Omega)^2.\nonumber
\end{align}
Finally, \eqref{4.10} is proved by taking $\eta=\f\varepsilon2$ in \eqref{4.11}.

Combining \eqref{4.12} and \eqref{4.10}, we have
\begin{equation}\label{4.13}
E_\Omega-2\sqrt{\Omega}\geq(1-\varepsilon)\big[B_{\Omega}(u_{\Omega})+G(u_{\Omega})\big]-C_{\varepsilon} (\ln \Omega)^2.
\end{equation}
On the other hand,
%for any fixed $v(x_3)\in H^1(\R)$ with $\|v\|_{L^2(\R)}=1$, take test function
%$$\bar{u}_{\Omega}(x_\perp, x_3)=\psi_{\Omega}(x_\perp)\sqrt{\ln \Omega}~v((\ln \Omega) x_3)\ \ \text{for}\ \
%x=(x_\perp, x_3)\in \R^3,$$
%where $\psi_{\Omega}(x_\perp)$ is as in \eqref{ot}. It can be calculated that
%$$\E_{\Omega}(\bar{u}_{\Omega})=2\sqrt{\Omega}+(\ln \Omega)^2\int_{\R}|v'(x_3)|^2dx_3-\f12 D(|\bar{u}_{\Omega}|^2, |\bar{u}_{\Omega}|^2),$$
it yields from Theorem \ref{thm3.1} that
\begin{equation}\label{4.14}
E_\Omega-2\sqrt{\Omega}\leq 0 \quad  \text{as} \quad \Omega\to+\infty.
\end{equation}
Together (\ref{4.13}) and \eqref{4.14}, we have
$$(1-\varepsilon)\big[B_{\Omega}(u_{\Omega})+G(u_{\Omega})\big]\leq  C_\varepsilon  (\ln \Omega)^2 \quad  \text{as} \quad \Omega\to+\infty.$$
We thus conclude that
$$G(u_{\Omega})=\|\partial_{x_3} u_{\Omega}\|_{2}^2\leq C (\ln \Omega)^2 \quad  \text{as} \quad \Omega\to+\infty,$$
which implies that \eqref{4.5} holds true. The proof of Lemma \ref{lem4.2} is therefore complete.
\qed

\vskip 0.05truein

Inspired by \cite{BSY, LSY}, we next analyze a basic reduction estimate for the Hartree interaction generated
by the lowest transverse mode.  After projecting
onto the lowest transverse mode, the transverse
variables can be integrated out, and the nonlocal interaction is reduced to an
effective one-dimensional bilinear form. For
$F,G\in L^1(\R)\cap H^1(\R)$, define
\begin{equation*}
 \mathcal K_\Omega(F,G):= \frac1{\ln\Omega} \iint_{\R\times\R}F(x_3)G(y_3)K_\Omega(x_3-y_3) dx_3dy_3,
\end{equation*}
where
\begin{equation}\label{4.16}
K_\Omega(s):=
\iint_{\R^2\times\R^2}
\frac{|\psi_1(x_{\perp})|^2|\psi_1(y_{\perp})|^2}
{\sqrt{s^2+(\ln \Omega)^2\Omega^{-\frac12}|x_{\perp}-y_{\perp}|^2}}
 dx_{\perp}dy_{\perp}
\end{equation}
with
$\psi_{1}(x_{\perp})=\frac{1}{\sqrt{\pi}}
e^{-\frac{|x_{\perp}|^2}{2}}$.
The following lemma gives the required one-dimensional reduction estimate.

\begin{Lemma}%[Bilinear reduction of the Hartree kernel]
\label{lem4.3}
Let $F,G\in L^1(\R)\cap H^1(\R)$, then
\begin{equation*}
\mathcal K_\Omega(F,G) =\frac12\int_\R F(x_3)G(x_3) dx_3 + o(1) \quad \text{as}\quad  \Omega\to+\infty.
\end{equation*}
More precisely, for sufficiently large $\Omega>0$,
%there exists a constant $C_{LH}:=C\big(\|F_\Omega\|_{L^1(\R)}\|G\|_{L^1(\R)},  \|F_\Omega\|_{L^2(\R)}\|G\|_{H^1(\R)}\big)>0$, independent of $\Omega$, such that
%\begin{equation*}       \|F_\Omega\|_{L^1(\R)}+\|G\|_{L^1(\R)} +\|F_\Omega\|_{H^1(\R)}+\|G\|_{H^1(\R)}
        %\leq M,
%\end{equation*}
\begin{equation*}
\left|\mathcal K_\Omega(F,G) -\frac12\int_\R F G dx_3\right|\leq o(1)\big(\|F\|_{L^1(\R)}\|G\|_{L^1(\R)}+  \|F\|_{L^2(\R)}\|G\|_{H^1(\R)}\big).
 %\quad \text{as}\quad  \Omega\to+\infty.
\end{equation*}
\end{Lemma}

\noindent{\bf Proof.}
Write $s=x_3-y_3$, then
\begin{equation*}
\mathcal K_\Omega(F,G) =\frac1{\ln\Omega}\int_\R K_\Omega(s) \int_\R F(x_3)G(x_3-s) dx_3ds .
\end{equation*}
For sufficiently large $\Om>0$, we split the $s$-integral into
\begin{equation*}
|s|\le\frac1{\ln\Omega},\qquad \frac1{\ln\Omega}<|s|\leq 1 \qquad\text{and}\qquad1<|s|.
\end{equation*}
For the near part, we have
\begin{equation*}
\begin{aligned}
&\left| \frac1{\ln\Omega}\int_{|s|\leq\frac1{\ln\Omega}} K_\Omega(s)\int_\R F(x_3)\bigl[G(x_3-s)-G(x_3)\bigr]  dx_3ds\right|\\
\leq&\frac1{\ln\Omega}\int_{|s|\leq\frac1{\ln\Omega}} K_\Omega(s) \|F\|_{L^2(\R)} \|G(\cdot-s)-G\|_{L^2(\R)} ds \\
\leq& \frac{C}{\ln\Omega} \|F\|_{L^2(\R)}\|G\|_{H^1(\R)}
\int_{|s|\leq\frac1{\ln\Omega}} |s|K_\Omega(s) ds .
\end{aligned}
\end{equation*}
Since
$\int_{|s|\leq\frac{1}{\ln \Omega}}\frac{|s|}{\sqrt{s^2+(\ln \Omega)^2\Omega^{-\f12}|x_{\perp}-y_{\perp}|^2}}ds\leq\frac{2}{\ln \Omega}$,
we have
\begin{align}
&\int_{|s|\leq\f{1}{\ln \Omega}}|s|K_\Omega(s)ds\label{4.19}\\
=&\iint_{\R^2\times\R^2} |\psi_1(x_{\perp})|^2|\psi_1(y_{\perp})|^2 \Big(\int_{|s|\leq\f{1}{\ln \Omega}}\frac{|s|}{\sqrt{s^2+(\ln \Omega)^2\Omega^{-\f12}|x_{\perp}-y_{\perp}|^2}}ds\Big)dx_{\perp}dy_{\perp}\leq\frac{2}{\ln \Omega}.\nonumber
\end{align}
It then follows from the above facts that
\begin{equation*}
\begin{aligned}
&\frac1{\ln\Omega}\int_{|s|\leq\frac1{\ln\Omega}}  K_\Omega(s) \int_\R F(x_3)G(x_3-s) dx_3ds\\
=& \frac{M_\Omega}{\ln\Omega} \int_\R F G dx_3 + o(1)\|F\|_{L^2(\R)}\|G\|_{H^1(\R)}
\ \ ~\text{as}\ \ \Omega\to+\infty,
\end{aligned}
\end{equation*}
where $M_\Omega:=\int_{|s|\leq\frac1{\ln\Omega}}K_\Omega(s) ds$.

We then estimate $M_\Omega$. One can check that for any fixed $x_{\perp}\neq y_{\perp}\in \R^2$,
\begin{equation}\label{4.20}
\int_{|s|\leq\frac{1}{\ln \Omega}}\frac{1}{\sqrt{s^2+(\ln \Omega)^2\Omega^{-\f12}|x_{\perp}-y_{\perp}|^2}}ds=\f12 \ln \Omega-4\ln \ln\Omega+O(1+\ln |x_{\perp}-y_{\perp}|)
\end{equation}
as $\Omega\to+\infty$. It then follows from \eqref{4.20} that
$$M_\Omega
%=\int_{|s|\leq\frac{1}{\ln \Omega}}\Big(\iint_{\R^2\times\R^2}\f{|\psi_1(x_{\perp})|^2|\psi_1(y_{\perp})|^2}{\sqrt{s^2+(\ln \Omega)^2\Omega^{-\f12}|x_{\perp}-y_{\perp}|^2}}dx_{\perp}dy_{\perp}\Big)ds
=\f12 \ln \Omega-4\ln \ln\Omega+O(1) \quad \text{as}\quad \Omega\to+\infty.$$
Hence the near part satisfies that for sufficiently large $\Omega>0$,
\begin{equation*}
\frac1{\ln\Omega}\int_{|s|\leq\frac1{\ln\Omega}}K_\Omega(s)\int_\R F(x_3)G(x_3-s)dx_3ds= \frac12\int_\R F G dx_3+o(1)\|F\|_{L^2(\R)}\|G\|_{H^1(\R)}.
%\quad \text{as} \quad \Omega\to+\infty.
\end{equation*}

Next, we estimate the middle part
$ \frac1{\ln\Omega}<|s|\leq1$.
Using $K_\Omega(s)\leq |s|^{-1}$ for $s\neq0$, we obtain that
\begin{equation*}
\begin{aligned}
&\left|\frac1{\ln\Omega}\int_{\frac1{\ln\Omega}<|s|\leq1}K_\Omega(s)\int_\R F(x_3)G(x_3-s) dx_3ds\right| \\
\leq&\frac1{\ln\Omega}\int_{\frac1{\ln\Omega}<|s|\leq1}\frac1{|s|}\|F\|_{L^2(\R)}
\|G\|_{L^2(\R)}ds \\
\leq&C \frac{\ln\ln\Omega}{\ln\Omega}\|F\|_{L^2(\R)}
\|G\|_{L^2(\R)}=o(1)\|F\|_{L^2(\R)}
\|G\|_{L^2(\R)}\ \ ~\text{as}\ \ \Omega\to+\infty.
\end{aligned}
\end{equation*}
Finally, we estimate the far part $|s|>1$. Since $K_\Omega(s)\leq |s|^{-1}\leq1$, one can calculate that
\begin{equation*}
\begin{aligned}
&\left|\frac1{\ln\Omega}\int_{|s|>1}K_\Omega(s)\int_\R F(x_3)G(x_3-s) dx_3ds\right|     \\
\leq&\frac1{\ln\Omega}\int_{|s|>1}\int_\R |F(x_3)| |G(x_3-s)| dx_3ds\leq\frac1{\ln\Omega}\int_{\R}\int_\R |F(x_3)| |G(x_3-s)| dx_3ds   \\
=&\frac1{\ln\Omega}\|F\|_{L^1(\R)}\|G\|_{L^1(\R)}
=o(1)\|F\|_{L^1(\R)}\|G\|_{L^1(\R)}  \ \ ~\text{as}\ \ \Omega\to+\infty.
\end{aligned}
\end{equation*}
Combining the near, middle and far estimates, we obtain that
\begin{equation*}
\mathcal K_\Omega(F,G) = \frac12\int_\R F G dx_3 +o(1)\big(\|F\|_{L^1(\R)}\|G\|_{L^1(\R)}+  \|F\|_{L^2(\R)}\|G\|_{H^1(\R)}\big)
\end{equation*}
as $\Omega\to+\infty$.
The proof of Lemma \ref{lem4.3} is therefore complete. $\qed$

\begin{Remark}
Following Lemmas \ref{lem4.2} and \ref{lem4.3}, we have
\begin{equation}\label{vv}
\frac{1}{\ln \Omega}\iint_{\R\times\R}|v_\Omega(x_3)|^2|v_\Omega(y_3)|^2 K_\Omega(x_3-y_3)dx_3dy_3=\f12\int_{\R}|v_\Omega|^4dx_3+o(1)
\end{equation}
as $\Omega\to+\infty$, where $v_\Omega$ defined by \eqref{4.15}.
This identity is the key step in identifying the limiting nonlinear term,
since it shows that the lowest-mode Hartree interaction is asymptotically
given by the local one-dimensional quartic term
$\frac12\int_{\R}|v_\Omega|^4 dx_3$.
%This convergence is crucial for analyzing the asymptotic behavior of $w_\Omega$ as $\Omega\to+\infty$. Its proof relies on the fact that the blown-up minimizer $w_\Omega$ is essentially projected onto the ground state $\psi_1$ of the two-dimensional Hermite operator. After this projection, $P_1w_\Omega$ has the ground-state profile $\psi_1$ in the first two variables. This property not only helps establish the uniform $H^1$-boundedness of the component of $P_1w_\Omega$ in the third variable in Lemma \ref{lem4.2}, but also enables us to accurately identify the limiting form of the nonlinear term.
\end{Remark}

\subsection{Proof of Theorem \ref{thm1.2}}
Based on the above facts, we are now ready to complete the proof of  the $H^1$-convergence in Theorem \ref{thm1.2}.
For convenience, we introduce the following lemma that we will need later.

\begin{Lemma}\label{lemA}
(cf. \cite[Lemma 3.1]{FG}) Let $a, b>0$, then
$$\inf\{\int_{\mathbb{R}}|f'(t)|^2dt-b\int_{\mathbb{R}}|f(t)|^4dt~:\int_{\mathbb{R}}|f(t)|^2dt=a\}
=-\frac{b^{2}}{12}a^3,$$
and the infimum is attained at
$f_{a,b}(t)=\frac{a\sqrt{b}}{2}\Big(\cosh \frac{abt}{2}\Big)^{-1}$.
\end{Lemma}

\noindent{\bf Proof of the $H^1$-convergence in Theorem \ref{thm1.2}.} It yields from Proposition \ref{lem4.1} and Lemma \ref{lem4.2} that
\begin{equation}\label{4.31}
\int_{\R^3}|\partial_{x_3} w_\Omega|^2dx=\int_{\R}|v'_\Omega(x_3)|^2dx_3+o(1)\ \ \text{as}\ \ \Omega\to+\infty,
\end{equation}
where $w_\Omega(x)$ and $v_\Omega(x_3)$ are as in \eqref{scaling} and (\ref{4.15}), respectively. Moreover, it follows from \eqref{vv} that
$$-\frac{1}{2\ln \Omega}\iint_{\R\times\R}|v_\Omega(x_3)|^2|v_\Omega(y_3)|^2 K_\Omega(x_3-y_3)dx_3dy_3=-\f14\int_{\R}|v_\Omega|^4dx_3+o(1)\ \ \text{as}\ \ \Omega\to+\infty,$$
where $K_\Omega(s)$ is as in \eqref{4.16}. Since $u_\Omega$ is a positive minimizer of $E_\Omega$ defined by \eqref{inf}, one can derive from \eqref{scaling} that
\begin{align}
\frac{E_\Omega-2\sqrt{\Omega}}{(\ln \Omega)^2}=&\frac{\sqrt{\Omega}}{(\ln \Omega)^2}\langle w_\Omega, (A_1-2)w_\Omega\rangle+\|\partial_{x_3} w_\Omega\|_{2}^2-\f{1}{2\ln \Omega} D_\Omega\big(|w_\Omega|^2, |w_\Omega|^2\big)\nonumber\\
\geq&\|\partial_{x_3} w_\Omega\|_{2}^2-\f{1}{2\ln \Omega} D_\Omega\big(|w_\Omega|^2, |w_\Omega|^2\big),\label{4.30}
\end{align}
where $A_1$ is as in \eqref{A}, and $D_\Omega\big(|w_\Omega|^2, |w_\Omega|^2\big):=\iint_{\R^3\times\R^3}\f{|w_\Omega(x)|^2|w_\Omega(y)|^2}{\sqrt{\Omega^{-\f12}(\ln \Omega)^2|x_{\perp}-y_{\perp}|^2+|x_{3}-y_{3}|^2}}dxdy$ for $x=(x_{\perp}, x_{3})\in\R^3$ and $y=(y_{\perp}, y_{3})\in\R^3$.

We then claim that
\begin{equation}\label{4.29}
\f{1}{\ln \Omega} D_\Omega\big(|w_\Omega|^2, |w_\Omega|^2\big)=\f{1}{\ln \Omega} D_\Omega\big(|P_1 w_\Omega|^2, |P_1 w_\Omega|^2\big)+o(1)\ \ \text{as}\ \ \Omega\to+\infty,
\end{equation}
where $P_1$ is defined by \eqref{P}. Indeed, it follows from \eqref{4.4} that for sufficiently large $\Omega>0$,
\begin{equation}\label{4.25}
\frac{\sqrt{\Omega}}{(\ln \Omega)^2}\Big(\|\n_\perp Q_1 w_\Omega\|_2^2+\int_{\R^3}(1+|z_\perp|^2)|Q_1 w_\Omega|^2dz\Big)+\|\partial_{x_3} Q_1 w_\Omega\|_2^2\leq\frac{C}{\ln \Omega}.
%\ \ \text{as}\ \ \Omega\to+\infty.
\end{equation}
Moreover, take $u(x_{\perp}, x_{3})=\Omega^{\frac{1}{4}}\sqrt{\ln \Omega}~w\big(\Omega^{\frac{1}{4}}x_{\perp}, (\ln \Omega)x_3\big)$ in \eqref{4.10}, we obtain that
$$\frac{1}{\ln \Omega}D_{\Omega}(|w|^2, |w|^2)\leq C\Big[\frac{\sqrt{\Omega}}{(\ln \Omega)^2}\Big(\|\n_\perp w\|_2^2+\int_{\R^3}(1+|z_\perp|^2)|w|^2dz\Big)+\|\partial_{x_3}  w\|_2^2\Big]+C.$$
Let $w=\frac{Q_1 w_\Omega}{\|Q_1 w_\Omega\|_2}$, we derive from \eqref{4.4} and \eqref{4.25} that
\begin{align}
&\f{1}{\ln \Omega} D_\Omega\big(|Q_1 w_\Omega|^2, |Q_1 w_\Omega|^2\big)\label{4.26}\\
\leq&C\|Q_1 w_\Omega\|_2^2\Big[\frac{\sqrt{\Omega}}{(\ln \Omega)^2}\Big(\|\n_\perp Q_1 w_\Omega\|_2^2+\int_{\R^3}(1+|x_\perp|^2)|Q_1 w_\Omega|^2dx\Big)+\|\partial_{x_3} Q_1 w_\Omega\|_2^2\Big]\nonumber\\
&+C\|Q_1 w_\Omega\|_2^4
\leq O\Big(\frac{1}{\sqrt{\Omega}}\Big)=o(1)\ \ \text{as}\ \ \Omega\to+\infty.\nonumber
\end{align}
Furthermore, it yields from Lemmas \ref{lem4.2} and \ref{lem4.3} that
\begin{equation}\label{4.27}
\f{1}{\ln \Omega} D_\Omega\big(|P_1 w_\Omega|^2, |P_1 w_\Omega|^2\big)=O(1)\ \ \text{as}\ \ \Omega\to+\infty.
\end{equation}

For the mixed term, one can deduce from \eqref{4.26} and \eqref{4.27} that
\begin{align}
&\f{1}{\ln \Omega}\iint_{\R^3\times\R^3}\frac{|P_1 w_\Omega(x)||Q_1 w_\Omega(x)||P_1 w_\Omega(y)||Q_1 w_\Omega(y)|}{\sqrt{\Omega^{-\f12}(\ln \Omega)^2|x_{\perp}-y_{\perp}|^2+|x_{3}-y_{3}|^2}}dxdy\label{4.28}\\
\leq&\f{1}{(\ln \Omega)^3}D_\Omega\big(|P_1 w_\Omega|^2, |P_1 w_\Omega|^2\big)+\f{\ln \Omega}{4}D_\Omega\big(|Q_1 w_\Omega|^2, |Q_1 w_\Omega|^2\big) \to0\ \ \text{as}\ \ \Omega\to+\infty.\nonumber
\end{align}
Since
$|D_\Omega(f, g)|\leq\sqrt{D_\Omega(f, f)}\cdot\sqrt{D_\Omega(g, g)}$, we also have
\begin{equation}\label{P2Q2}
\begin{aligned}
\f{1}{\ln \Omega} D_\Omega\big(|P_1 w_\Omega|^2, |Q_1 w_\Omega|^2\big)\leq \sqrt{\f{1}{\ln \Omega}D_\Omega\big(|P_1 w_\Omega|^2, |P_1 w_\Omega|^2\big)}\cdot\sqrt{\f{1}{\ln \Omega}D_\Omega\big(|Q_1 w_\Omega|^2, |Q_1 w_\Omega|^2\big)}\to0
\end{aligned}
\end{equation}
as $\Omega\to+\infty$.
Similarly, one can also check that as $\Omega\to+\infty$,
\begin{equation}\label{PPPQ}
\f{1}{\ln \Omega} D_\Omega\big(|P_1 w_\Omega|^2, P_1 w_\Omega Q_1 w_\Omega \big) \to0,\ \ \f{1}{\ln \Omega} D_\Omega\big(P_1 w_\Omega Q_1 w_\Omega , |Q_1 w_\Omega|^2\big) \to0.
\end{equation}
It then follows from  \eqref{4.26} and \eqref{4.28}--\eqref{PPPQ} that
\begin{align}
\Big|\f{1}{\ln \Omega} D_\Omega\big(|w_\Omega|^2, |w_\Omega|^2\big)-\f{1}{\ln \Omega} D_\Omega\big(|P_1 w_\Omega|^2, |P_1 w_\Omega|^2\big)\Big|
\to0\ \ \text{as}\ \ \Omega\to+\infty,\nonumber
\end{align}
and the claim of \eqref{4.29} is true.

According to \eqref{vv}, we deduce from \eqref{4.31}--\eqref{4.29} that
\begin{align}
\frac{E_\Omega-2\sqrt{\Omega}}{(\ln \Omega)^2}\geq&\int_{\R}|v'_\Omega(x_3)|^2dx_3-\f{1}{2\ln \Omega} D_\Omega\big(|P_1 w_\Omega|^2, |P_1 w_\Omega|^2\big)+o(1)\nonumber\\
=&\int_{\R}|v'_\Omega(x_3)|^2dx_3-\f14\int_{\R}|v_\Omega|^4dx_3+o(1)\ \ \text{as}\ \ \Omega\to+\infty,\nonumber
\end{align}
which implies from Theorem \ref{thm3.1} that
\begin{equation}\label{4.33}
\int_{\R}|v'_\Omega(x_3)|^2dx_3-\f14\int_{\R}|v_\Omega|^4dx_3\leq-\frac{1}{192}+o(1)\ \ \text{as}\ \ \Omega\to+\infty.
\end{equation}
Next, we prove that
\begin{equation*}
\int_{\R}|v'_\Omega(x_3)|^2dx_3-\f14\int_{\R}|v_\Omega|^4dx_3\geq-\frac{1}{192}+o(1)\ \ \text{as}\ \ \Omega\to+\infty.
\end{equation*}
Define
\begin{equation}\label{4.35}
I_{\tilde{\alpha}}(m):=\inf_{\|f\|_{L^2(\R)}^2=m}\mathcal{F}_{\tilde{\alpha}}(f)\ \ \text{with}\ \ \mathcal{F}_{\tilde{\alpha}}(f):=\int_{\R}|f'(t)|^2dt-\f{\tilde{\alpha}}{4}\int_{\R}|f|^4dt.
\end{equation}
It follows from Lemma \ref{lemA} that $I_{\tilde{\alpha}}(1)=-\frac{\tilde{\alpha}^2}{192}$. Let $g=\frac{f}{\sqrt{m}}$, then $\|g\|_{L^2(\R)}^2=1$ and
$$\mathcal{F}_{\tilde{\alpha}}(f)=m\Big(\int_{\R}|g'(t)|^2dt-\f{m \tilde{\alpha}}{4}\int_{\R}|g|^4dt\Big).$$
Using Lemma \ref{lemA}  again, we have
\begin{equation}\label{4.32}
I_{\tilde{\alpha}}(m)=mI_{\tilde{\alpha} m}(1)=-\frac{\tilde{\alpha}^2}{192}m^3.
\end{equation}
Denote $m_{\Omega}:=\|v_\Omega\|_{L^2(\R)}^2$, we obtain from \eqref{4.32} that
$\mathcal{F}_1(v_\Omega)\geq-\f{1}{192}m_{\Omega}^3$. Since $m_{\Omega}\to1$ as $\Omega\to+\infty$ in view of Proposition \ref{lem4.1}, we have
$\mathcal{F}_1(v_\Omega)\geq-\frac{1}{192}+o(1)$ as $\Omega\to+\infty$.
Together with \eqref{4.33}, we deduce that
\begin{equation}\label{4.34}
\mathcal{F}_1(v_\Omega)=\int_{\R}|v'_\Omega(x_3)|^2dx_3-\f14\int_{\R}|v_\Omega|^4dx_3
=-\frac{1}{192}+o(1)\ \ \text{as}\ \ \Omega\to+\infty.
\end{equation}

Set $\tilde{v}_{\Omega}:=\frac{v_\Omega}{\|v_\Omega\|_{L^2(\R )}}$, then we obtain from Lemma \ref{lem4.2} and \eqref{4.34} that $\mathcal{F}_1(\tilde{v}_\Omega)\to-\frac{1}{192}$ as $\Omega\to+\infty$. Thus, $\{\tilde{v}_\Omega\}$ is a minimizing sequence of $I_1(1)$ defined by \eqref{4.35}. Recall the  following Gagliardo-Nirenberg inequality% (cf. \cite{Weinstein})
$$\|f\|_{L^4(\R)}^4\leq 2 \|f'\|_{L^2(\R)}\|f\|^3_{L^2(\R)}\ \ \text{for any}\ \ f\in H^1(\R).$$
Since $\|\tilde{v}_\Omega\|_{L^2(\R)}=1$, we have $\mathcal{F}_1(\tilde{v}_\Omega)\geq\|\tilde{v}_\Omega'\|_{L^2(\R)}^2-\f12\|\tilde{v}_\Omega'\|_{L^2(\R)}$, which implies that $\{\tilde{v}_\Omega\}$ is bounded uniformly with respect to $\Omega$ in $H^1(\R)$.
Using the concentration--compactness principle, we deduce that there exist a subsequence $\{v_{\Omega_k}\}$ and a constant sequence $\{\tau_{\Omega_k}\}\subset\R$ such that
\begin{equation}\label{rho}
v_{\Omega_k}(x_3+\tau_{\Omega_k})\to v_0(x_3) \ \ \text{strongly in}\ \ H^1(\R) \ \ \text{as}\ \ k\to+\infty,
\end{equation}
where $v_0(x_3)>0$ is a minimizer of $I_1(1)$. It follows from \cite{Kwong} and Lemma \ref{lemA} that $v_0(x_3)=\f14(\cosh \frac{x_3}{8})^{-1}$ is the unique positive minimizer of $I_1(1)$, up to translations.  Additionally, since the limiting function is unique and the above convergence is independent of the subsequence that we choose, we conclude that the above convergence holds true for the whole sequence $\{v_\Omega\}$. This proves the $H^1$-convergence in Theorem \ref{thm1.2}  in view of Proposition \ref{lem4.1}.
\qed

\vskip 0.05truein
We now prove the
$L^\infty$-convergence. The starting point is the rescaled
Euler--Lagrange equation.
By the symmetry of positive minimizers and the translational invariance
of \eqref{inf} in the $x_3$-direction, we henceforth choose every positive
minimizer to be even and take $\tau_\Omega=0$ in Theorem \ref{thm1.2}.
%By the symmetry of positive minimizers and the translational invariance of \eqref{inf} in the $x_3$-direction, we choose each positive minimizer to be even. Under this normalization, the translation parameter in Theorem \ref{thm1.2} is $\tau_\Omega=0$.
Substituting the scaling \eqref{scaling} into \eqref{ele}, we obtain that
\begin{align}
-\p_{x_3}^2w_\Omega+a_\Omega(A_1-2)w_\Omega=(U_{\Omega}+\nu_\Omega)w_\Omega\text{ \ in \ }\R^3, \label{4}
\end{align}
%&\sqrt{\Omega} \big(-\Delta_\perp+|x_\perp|^2\big) w_\Omega-\big( \ln \Omega \big)^2\p_{x_3}^2w_\Omega\label{4}\\&- \ln \Omega
%\int_{\R^3}\f{|w_\Omega(y_\perp, y_3)|^2}{\sqrt{\Omega^{-\f12} {\big(\ln \Omega \big)^2}|x_\perp-y_\perp|^2+(x_3-y_3)^2}}dy w_\Omega(x)=\mu_\Omega w_\Omega\text{ \ in \ }\R^3.\nonumber
where $A_1$ is as in \eqref{A}, and
\begin{equation}\label{5.3}
a_\Omega=\frac{\sqrt{\Omega}}{(\ln\Omega)^2},\ \ \nu_\Omega=\frac{\mu_\Omega-2\sqrt{\Omega}}{(\ln\Omega)^2}
\end{equation}
\begin{equation}\label{5.1}
U_\Omega(x):=\frac1{\ln\Omega}\int_{\R^3}\frac{|w_\Omega(y)|^2}
{\sqrt{\Omega^{-1/2}(\ln\Omega)^2|x_\perp-y_\perp|^2+(x_3-y_3)^2}}dy.
\end{equation}

\begin{Proposition}\label{pro5.1}
For every $\Omega>1$, let $u_\Omega\in \H$ be the corresponding minimizer of \eqref{inf}. Then
\begin{align}
w_\Omega(x_\perp, x_3)
\to W_0(x_\perp,x_3)
\ \ \text{strongly in }\ \ L^\infty(\R^3)\ \ \text{as}\ \ \Omega\to+\infty,\nonumber
\end{align}
where $w_\Omega$ is as in \eqref{scaling}, and $W_0(x_\perp,x_3)=\frac{1}{4\sqrt{\pi}}e^{-\frac{|x_\perp|^2}{2}}\Big(\cosh \f{x_3}{8}\Big)^{-1}$ for $x=(x_\perp, x_3)\in \R^3$.
%there exists a constant $C>0$, independent of $\Omega$, such that the function $w_\Omega$ defined by \eqref{scaling} satisfies $\|w_\Omega\|_{L^\infty(\R^3)}\leq C$ as $\Omega\to+\infty$.
\end{Proposition}
\noindent{\bf Proof.}
For simplicity, denote
$$B:=-\Delta+1+|x_\perp|^2 =-\partial_{x_3}^2+A_1+1\ \ \text{ for}\ \  x=(x_\perp, x_3)\in\R^3,$$
where $A_1$ is as in \eqref{A}. We first claim that there exists a constant $C>0$, independent of $\Omega>0$, such that for  $F_\Omega:=(U_\Omega+\nu_\Omega)w_\Omega$,
\begin{equation}\label{5.4}
\|B w_\Omega\|_{2}\leq C\big(\|F_\Omega\|_{2}+\| w_\Omega\|_{2}\Big) \quad  \text{as}  \quad \Omega\to+\infty.
\end{equation}
Indeed, note from \eqref{P} and \eqref{QO} that $w_\Omega=P_1 w_\Omega+Q_1 w_\Omega$ and $F_\Omega=P_1 F_\Omega+Q_1 F_\Omega$. We first consider the $P_1-$part. Since $A_1 P_1=2 P_1$, applying $P_1$ to \eqref{4} yields that $-\p_{x_3}^2 P_1 w_\Omega=P_1 F_\Omega$. Moreover, we have $BP_1=(-\partial_{x_3}^2+3)P_1$. Therefore, there exists a constant $C>0$, independent of $\Omega>0$, such that
\begin{align}
 \|B P_1 w_\Omega\|_2\leq C(\|-\p_{x_3}^2 P_1 w_\Omega\|_2+\|P_1 w_\Omega\|_2)
\leq C(\|F_\Omega\|_2+\|w_\Omega\|_2).\label{5.5}
\end{align}
We next consider the $Q_1-$part. Set $L_\Omega=-\partial_{x_3}^2+a_\Omega(A_1-2)$, where $a_\Omega$ and $A_1$ are as in \eqref{5.3} and \eqref{A}, respectively. Since $a_\Omega\to +\infty$ as $\Omega\to +\infty$, and $A_1\geq 4$ on $\operatorname{Ran} Q_1$, we have
\begin{equation}\label{5.6}
    B\leq C(L_\Omega+1)\ \ \text{on}\ \ \operatorname{Ran} Q_1 \quad  \text{as}  \quad \Omega\to+\infty.
\end{equation}
Since $B$ and $L_\Omega+1$ are nonnegative commuting
self-adjoint operators on $\operatorname{Ran}Q_1$, the
Simultaneous Diagonalization Theorem
\cite[Theorem~4.43]{Lewin} and \eqref{5.6} yield
\begin{equation}\label{5.7}
\|BQ_1w_\Omega\|_2
\leq C\|(L_\Omega+1)Q_1w_\Omega\|_2
\leq C\big (\|F_\Omega\|_2+\|w_\Omega\|_2\big )\quad  \text{as}  \quad \Omega\to+\infty.
\end{equation}
%Applying $Q_1$ to \eqref{4}, we obtain that $L_\Omega Q_1 w_\Omega=Q_1 F_\Omega$. It then follows from \eqref{5.6} that
%\begin{equation}\label{5.7}
% \|B Q_1 w_\Omega\|_2\leq C(\|L_\Omega Q_1 w_\Omega\|_2+\|Q_1 w_\Omega\|_2)\leq C(\|F_\Omega\|_2+\|w_\Omega\|_2)\quad  \text{as}  \quad \Omega\to+\infty.
%\end{equation}
Thus, the claim \eqref{5.4} is proved in view of \eqref{5.5} and \eqref{5.7}.

Next, we prove that
there exists a constant $C>0$, independent of $\Omega>0$, such that
\begin{equation}\label{FO}
\|F_\Omega\|_2\leq C \quad \text{as} \quad \Omega\to+\infty.
\end{equation}
First, one can check from Theorem \ref{thm3.1}, Lemmas \ref{lem4.2} and \ref{lem4.3} that there exists a constant $C>0$, independent of $\Omega$, such that $|\nu_\Omega|\leq C$ as $\Omega\to+\infty$. We next prove that $\{U_\Omega\}$ is bounded uniformly in $L^\infty (\R^3)$ as $\Omega\to+\infty$. For simplicity, set
$$T_\Omega g(x):=\frac1{\ln\Omega}\int_{\R^3}\frac{g(y)}
{\sqrt{\Omega^{-1/2}(\ln\Omega)^2|x_\perp-y_\perp|^2+(x_3-y_3)^2}}dy,$$
and
$$\tilde{K}_\Omega (z):=\frac{1}
{\sqrt{\Omega^{-1/2}(\ln\Omega)^2|z_\perp|^2+z_3^2}},$$
then $U_\Omega (x)=T_\Omega(|w_\Omega|^2)$.
Decompose the integral into the local part
$\mathcal N_x = \{y:\ |x_\perp-y_\perp|\leq 2,\ |x_3-y_3|\leq2\}
$
and its complement. Using polar coordinates
$r=|x_\perp-y_\perp|$ and $s=x_3-y_3$  on $\mathcal N_x$, we have
\begin{align}
&\int_{\mathcal N_x}\tilde{K}_\Omega(x-y)dy\leq C\int_0^2 r dr\int_{-2}^{2} \frac{ds}{\sqrt{\Omega^{-1/2}(\ln\Omega)^2r^2+s^2}} \nonumber\\
\leq &C \int_0^2 r\left(1+\big|\ln\big(\Omega^{-1/4}\ln \Omega \big)\big|+|\ln r|\right)dr \leq C(1+ \ln \Omega )\label{5.8}
\end{align}
for sufficiently large $\Omega>0$. It then follows from \eqref{5.8} that for sufficiently large $\Omega>0$,
\begin{equation}
\label{5.9}
\frac1{\ln \Omega}\int_{\mathcal N_x}\tilde{K}_\Omega(x-y)|g(y)|dy
\leq C\|g\|_{L^\infty(\R^3)} .
\end{equation}
For the complementary region, we use the decay supplied by the transverse
weights. Since
$(1+|y_\perp|^2)^{-2}\in L^1(\R^2)$, we have
\begin{align}
&\frac1{\ln \Omega}\int_{\R^3\setminus\mathcal N_x}\tilde{K}_\Omega(x-y)|g(y)|dy\nonumber\\
=&\frac1{\ln \Omega}\int_{\big\{x,y\in\R^3:~|x_3-y_3|>2\big\}}\tilde{K}_\Omega(x-y)|g(y)|dy+\frac1{\ln \Omega}\int_{\big\{x,y\in\R^3:~|x_{\perp}-y_\perp|>2, ~|x_3-y_3|\leq 2\big\}}\tilde{K}_\Omega(x-y)|g(y)|dy\nonumber\\
\leq &\f{\|(1+|y_\perp|^2)^2 g\|_{L^\infty(\R^3)}}{\ln \Omega}\int_{\R^2}(1+|y_\perp| ^2)^{-2}dy_\perp\int_{-2}^2\f{d(x_3-y_3)}{\sqrt{4\Omega^{-1/2}(\ln\Omega)^2 +(x_3-y_3)^2}}+\frac{1}{2 \ln \Omega}\int_{\R^3}|g(y)|dy \label{5.10}\\
\leq&  C\|(1+|y_\perp|^2)^2 g\|_{L^1(\R^3)\cap L^\infty(\R^3)}.\nonumber
\end{align}
Combining \eqref{5.9} and \eqref{5.10}, we obtain that for sufficiently large $\Omega$,
\begin{equation}\label{5.11}
\|T_\Omega g\|_{L^\infty(\R^3)}\leq  C\|(1+|y_\perp|^2)^2 g\|_{L^1(\R^3)\cap L^\infty(\R^3)}.
\end{equation}
Since $P_1 w_\Omega(x)=\psi_1(x_\perp)v_\Omega(x_3)$ with $\psi_1(x_\perp)=\frac{1}{\sqrt{\pi}}e^{-\frac{|x_{\perp}|^2}{2}}$, it follows from Lemma \ref{lem4.2} and \eqref{5.11} that there exists a constant $C>0$, independent of $\Omega$, such that
\begin{equation}\label{5.15}
\|T_\Omega |P_1 w_\Omega|^2\|_{L^\infty(\R^3)}\leq  C \quad  \text{as}  \quad \Omega\to+\infty.
\end{equation}

It remains to estimate the $Q_1-$part. Recall the three-dimensional Gagliardo-Nirenberg inequality (cf. \cite{Weinstein})
\begin{equation}\label{5.12}
\|Q_1 w_{\Omega}\|_{L^{4}(\R^3)}\leq C\| Q_1 w_{\Omega}\|_{2}^{1 / 4}\|\nabla (Q_1 w_{\Omega})\|_{2}^{3 / 4}.
\end{equation}
Moreover, direct calculations yield that
$$\int_{\mathcal N_x}|\tilde{K}_\Omega(x-y)|^2dy\leq C\int_0^2rdr\int_{-2}^2\f{ds}{\Omega^{-1/2}(\ln\Omega)^2
    r^2+s^2}\leq C \Omega^{1/4}(\ln\Omega)^{-1}\ \ \text{as}\ \ \Omega\to+\infty.$$
We then obtain from Proposition \ref{lem4.1}, \eqref{4.25} and \eqref{5.12} that
\begin{align}
\frac1{\ln \Omega} \int_{\mathcal N_x}\tilde{K}_\Omega(x-y)|Q_1 w_{\Omega}(y)|^2dy\leq& \frac1{\ln \Omega} \|\tilde{K}_\Omega(x-y)\|_{L^2(\mathcal N_x)}\|Q_1 w_{\Omega}\|_{L^{4}(\R^3)}^2\label{5.13}\\
\leq&\f{C}{\ln \Omega}\Omega^{\f18}(\ln \Omega)^{-\f12}\|Q_1 w_{\Omega}\|_{2}^{\f12}\|Q_1 w_{\Omega}\|_{H^{1}(\R^3)}^{\f32}\leq C\nonumber
\end{align}
as $\Omega\to+\infty$.
Since $|\tilde{K}_\Omega(x-y)|\leq C \Omega^{\f14}(\ln \Omega)^{-1}$ on $\mathcal N_x^c$, we derive from \eqref{4.25} that
\begin{equation}\label{5.14}
\frac1{\ln \Omega}
\int_{\mathcal N_x^c}\tilde{K}_\Omega(x-y)|Q_1 w_{\Omega}(y)|^2dy\leq C \Omega^{\f14}(\ln \Omega)^{-2} \|Q_1 w_{\Omega}\|_{2}^{2}\leq C\ \ \text{as}\ \ \Omega\to+\infty.
\end{equation}
Combining \eqref{5.13} and \eqref{5.14}, we obtain that
\begin{equation}\label{5.16}
\|T_\Omega |Q_1 w_\Omega|^2\|_{L^\infty(\R^3)} \leq C\ \ \text{as}\ \ \Omega\to+\infty.
\end{equation}
We finally conclude from \eqref{5.15} and \eqref{5.16} that there exists a constant $C>0$, independent of $\Omega$, such that
\begin{equation}\label{U}
\|U_\Omega\|_{L^\infty(\R^3)} =\|T_\Omega | w_\Omega|^2\|_{L^\infty(\R^3)}\leq 2\big(\|T_\Omega |P_1 w_\Omega|^2\|_{L^\infty(\R^3)}+\|T_\Omega |Q_1 w_\Omega|^2\|_{L^\infty(\R^3)}\big)\leq C
\end{equation}
%\ \ \text{as}\ \ \Omega\to+\infty.$$
as $\Omega\to+\infty$.
Thus, $\|F_\Omega\|_2\leq\big(\|U_\Omega\|_{L^\infty(\R^3)}+ |\nu_\Omega |\big) \|w_\Omega\|_2\leq C$ as $\Omega\to+\infty$ for some constant $C>0$ independent of $\Omega$, and \eqref{FO} is proved.

By the definition of $B$, one can calculate that
\begin{equation}\label{5.17}
\|B w_\Omega\|_2^2=\|\Delta w_\Omega\|_2^2+\|(1+|x_\perp|^2)w_\Omega\|_2^2+2\int_{\R^3}(1+|x_\perp|^2)|\nabla w_\Omega|^2dx-4\|w_\Omega\|_2^2.
\end{equation}
Since
 $\|\Delta w_\Omega\|_2^2=\sum_{i,j=1}^3\|\p_i \p_j w_\Omega\|_2^2$,
%\|\widehat{\Delta w_\Omega}\|_2^2=\int_{\R^3}|\xi|^4|\widehat{ w_\Omega}|^2 d\xi=\sum_{i,j=1}^3\int_{\R^3}\xi_i^2\xi_j^2|\widehat{ w_\Omega}|^2 d\xi=
%One can also calculate from \eqref{5.17} that $\|\nabla w_\Omega\|_2^2\leq \|w_\Omega\|_2\|\Delta w_\Omega\|_2\leq C(\|B w_\Omega\|_2^2+\|w_\Omega\|_2^2)$.
we conclude from \eqref{5.17} that there exists a constant $C>0$, independent of $\Omega$, such that
\begin{equation}\label{H2}
\|w_\Omega\|_{H^2(\R^3)}^2\leq C(\|B w_\Omega\|_2^2+4\|w_\Omega\|_2^2) .
\end{equation}
Following \eqref{5.4}, \eqref{FO} and \eqref{H2}, we obtain that $\{w_\Omega\}$ is bounded uniformly in $H^2(\R^3)$ as $\Omega\to+\infty$. Note that $H^{2\alpha}(\R^3)\hookrightarrow L^\infty(\R^3)$ is continuous for $\alpha>\f34$. Applying the Sobolev interpolation inequality, one can further deduce from the $H^1$-convergence in Theorem \ref{thm1.2} that for $\f34<\alpha<1$,
\begin{align}
\|w_\Omega-W_0\|_{L^\infty(\R^3)}&\leq C\|w_\Omega-W_0\|_{H^{2\alpha}(\R^3)}\leq C\|w_\Omega-W_0\|_{2}^{1-\alpha}
\|w_\Omega-W_0\|_{H^2(\R^3)}^\alpha\to0\nonumber
%\ \ \text{as}\ \
\end{align}
as $\Omega\to+\infty$.
%Since $H^2(\R^3)\hookrightarrow L^\infty(\R^3)$ is continuous according to Fourier transformation, we conclude that $\{w_\Omega\}$ is also bounded uniformly in $L^\infty(\R^3)$ as $\Omega\to+\infty$.
The proof of Proposition \ref{pro5.1} is therefore complete.
\qed

\vskip 0.05truein
\noindent{\bf Proof of Theorem \ref{thm1.2}.}
Combining the $H^1$-convergence proved above with the $L^\infty$-convergence obtained in Proposition \ref{pro5.1}, we obtain that
\begin{equation*}
w_\Omega^\tau\to W_0
\quad\text{strongly in }H^1(\R^3)\cap L^\infty(\R^3) \ \ \text{as}\ \ \Omega\to+\infty.
\end{equation*}
This completes the proof of Theorem \ref{thm1.2}.
\qed

\section{Uniqueness of Ground States as $\Omega\to+\infty$}\label{S5}
In this section, we prove that positive minimizers are unique up to
translations in the $x_3$-direction in the strongly confined regime
$\Omega\to+\infty$. By Proposition \ref{lem2.3}, it suffices to prove uniqueness among
positive even minimizers. We first recall the one-dimensional limiting profile and the corresponding
linearized coercivity estimate, which will be used in the proof of the
large-$\Omega$ coercivity.
Let
\begin{equation}\label{0rd}
\rho_0(x_3):=\f14\left(\cosh\f{x_3}{8}\right)^{-1}.
\end{equation}
A direct computation gives
\begin{equation}\label{1rho}
-\rho_0''-\f12\rho_0^3=\lambda_{1d}\rho_0 \text{ \ in \ }\R,\qquad\lambda_{1d}:=-\f1{64}.
\end{equation}
Define the one-dimensional linearized quadratic form by
\begin{equation}\label{q1d}
q_{1d}(h):=\int_{\R}|h'|^2dx_3-\lambda_{1d}\int_{\R}h^2dx_3-\f32\int_{\R}\rho_0^2h^2dx_3.
\end{equation}
The associated one-dimensional linearized operator is
\begin{equation*}
L_{1d}=-\p_{x_3}^2-\lambda_{1d}-\f32\rho_0^2=-\p_{x_3}^2+\f1{64}-\f32\rho_0^2 .
\end{equation*}
It is well known from \cite[Propositions 2.8 and 2.9]{Weinstein1} that the one-dimensional ground state $\rho_0$ is
nondegenerate in the sense that
$\ker L_{1d}=\operatorname{span}\{\rho_0'\}$, and $L_{1d}$ is coercive on the restricted even subspace. More precisely,
there exists a constant $c>0$ such that
\begin{equation}\label{L2}
q_{1d}(g)\geq c\|g\|_{L^2(\R)}^2
\end{equation}
for all real-valued even functions $g\in H^1(\R)$ satisfying  $\langle g,\rho_0\rangle_{L^2(\R)}=0$.
Furthermore, \eqref{L2} can be upgraded to
\begin{equation}\label{0qce}
q_{1d}(g)\geq c_1\|g\|_{H^1(\R)}^2
\end{equation}
for some constant $0<c_1<1$ and every such $g$.

Next, we introduce the rescaled Coulomb interaction appearing after the
large-$\Omega$ anisotropic scaling. For $f,g\in L^1_{\rm loc}(\R^3)$, define
\begin{equation}\label{0Dd}
\mathcal D_\Omega(f,g):=\f1{\ln\Omega}\iint_{\R^3\times\R^3}\f{f(x)g(y)}
{\sqrt{\Omega^{-\f12}(\ln\Omega)^2|x_\perp-y_\perp|^2+(x_3-y_3)^2}}dxdy .
\end{equation}
The following lemma provides the estimates on the nonlocal terms which are
needed in the proof of the large-$\Omega$ coercivity.

\begin{Lemma}\label{0lnc}
Let $\phi\in\H$ be real-valued and write
\begin{equation}\label{c5}
\phi=P_1\phi+Q_1\phi=\psi_1(x_\perp)h(x_3)+\eta,
\qquad\eta:=Q_1\phi,
\end{equation}
where $P_1$ and $Q_1$ are as in \eqref{P} and \eqref{QO},
respectively.
%Assume that $\{h_\Omega\}$ is bounded uniformly in $H^1(\R)$ and that $\{\eta_\Omega\}$ is bounded uniformly in $\H$. Then
%\begin{equation}\label{0nl1}
%\int_{\R^3}U_\Omega |P_1\phi|^2dx
%=\frac12\int_{\R}\rho_0^2h_\Omega^2dx_3+o(1)\|h_\Omega\|_{H^1(\R)}^2
%\quad \text{as }\quad \Omega\to+\infty,
%\end{equation}
%and
%\begin{equation}\label{0nl2}
%\mathcal D_\Omega\big(w_\Omega P_1\phi,w_\Omega P_1\phi\big)
%=\frac12\int_{\R}\rho_0^2h_\Omega^2dx_3
%+o(1)\|h_\Omega\|_{H^1(\R)}^2
%\quad \text{as }\quad \Omega\to+\infty.
%\end{equation}
%Here $U_\Omega$, $\rho_0$, and $w_\Omega$ are defined by
%\eqref{5.1}, \eqref{0rd}, and \eqref{scaling}, respectively.
%and the $o(1)$-terms are uniform for $\{h_\Omega\}$ bounded in $H^1(\R)$.
Then, for any $\varepsilon>0$, there exists $C_\varepsilon>0$,
independent of $\Omega$, such that
\begin{align}
&\int_{\R^3}U_\Omega\eta ^2dx+2\mathcal D_\Omega(w_\Omega\eta ,w_\Omega\eta)    \label{0nh}    \\
\leq&\varepsilon\int_{\R^3}|\partial_{x_3}\eta|^2dx+\varepsilon a_\Omega\int_{\R^3}\Big(|\nabla_\perp\eta|^2+|x_\perp|^2\eta^2-2\eta^2\Big)dx
+C_\varepsilon\|\eta\|_2^2\quad \text{as }\quad \Omega\to+\infty,\nonumber
\end{align}
where $U_\Om$, $w_\Om$ and $a_\Omega$ as in \eqref{5.1}, \eqref{scaling} and \eqref{5.3}, respectively.
Moreover, for the above $\varepsilon>0$ and sufficiently large $\Omega>0$, we also have
\begin{align}
&\left|\int_{\R^3}U_\Omega \eta P_1\phi dx
\right|+\left|\mathcal D_\Omega\big(w_\Omega P_1\phi,w_\Omega\eta\big)
\right|  \label{0nm} \\
\leq&\sqrt{\varepsilon}\|h\|_{H^1(\R)}^2
+\sqrt{\varepsilon}\int_{\R^3}|\partial_{x_3}\eta|^2dx
+\sqrt{\varepsilon} a_\Omega\int_{\R^3}\Big(|\nabla_\perp\eta|^2
+|x_\perp|^2\eta^2-2\eta^2\Big)dx+C_{\varepsilon}\|\eta\|_2^2.\nonumber
\end{align}
\end{Lemma}

\noindent{\bf Proof.}
Let
\begin{equation*}
V_\Omega(x-y):=\frac1{\ln\Omega} \frac1{\sqrt{\Omega^{-1/2}(\ln\Omega)^2|x_\perp-y_\perp|^2 +(x_3-y_3)^2}} .
\end{equation*}
After the same rescaling as in the definition of $\mathcal D_\Omega$,
\eqref{4.8} implies that, for every $0<\delta<1$, there exists
$C_\delta>0$, independent of $\Omega$, such that
\begin{equation}\label{HH}
V_\Omega(x-y)\leq\delta(H_x+H_y)+C_\delta
\end{equation}
in the sense of quadratic forms on $L^2(\R^6)$, where $H_x=-\partial_{x_3}^2+a_\Omega(A_1-2)$, and $a_\Omega$ and $A_1$ are as in \eqref{5.3} and \eqref{A}, respectively.
Applying this estimate to
$\Psi(x,y):=w_\Omega(x)\eta(y)$,
and using $\|w_\Omega\|_2=1$, we obtain that
\begin{equation*}
\begin{aligned}
\mathcal D_\Omega(w_\Omega^2,\eta^2)
&=\iint_{\R^3\times\R^3}V_\Omega(x-y)w_\Omega^2(x)\eta^2(y) dxdy \leq \delta\langle H_xw_\Omega,w_\Omega\rangle\|\eta\|_2^2+\delta\langle H_y\eta,\eta\rangle+C_\delta\|\eta\|_2^2.
\end{aligned}
\end{equation*}
According to \eqref{4.4} and Lemma \ref{lem4.2}, one can check that there exists a constant $C>0$, independent of $\Omega$, such that
\begin{equation}\label{ww}
\langle H_xw_\Omega,w_\Omega\rangle\leq C \quad \text{as} \quad \Omega\to+\infty,
\end{equation}
and hence $\mathcal D_\Omega(w_\Omega^2,\eta^2) \leq\delta\langle H_y\eta,\eta\rangle+C_\delta\|\eta\|_2^2$ as $\Omega\to+\infty$. Since
$ \int_{\R^3}U_\Omega\eta^2dx = \mathcal D_\Omega(w_\Omega^2,\eta^2)$, we obtain that for any $\varepsilon>0$,
\begin{align}
&\int_{\R^3}U_\Omega\eta^2dx+2\mathcal D_\Omega(w_\Omega\eta,w_\Omega\eta)\nonumber\\
\leq &3\mathcal D_\Omega(w_\Omega^2,\eta^2) \leq \varepsilon\langle H_y\eta,\eta\rangle+C_\varepsilon\|\eta\|_2^2\label{0fbu}\\
\leq&\varepsilon\int_{\R^3}|\partial_{x_3}\eta|^2dx+\varepsilon a_\Omega\int_{\R^3}\Big(|\nabla_\perp\eta|^2+|x_\perp|^2\eta^2-2\eta^2\Big)dx+C_\varepsilon\|\eta\|_2^2   \quad \text{as} \quad \Omega\to+\infty.\nonumber
\end{align}
Thus, \eqref{0nh} is true.

In order to prove \eqref{0nm}, we first estimate $\int_{\R^3}U_\Omega |P_1\phi|^2 dx$.
Recall from \eqref{c1} and \eqref{c5} that
\begin{equation*}
P_1w_\Omega:=\psi_1(x_\perp)v_\Omega(x_3),\qquad P_1\phi:=\psi_1(x_\perp)h(x_3),
\end{equation*}
where $\psi_1(x_\perp)$ is as in \eqref{ot}.
By choosing $F=v_\Omega^2$ and $G=h^2$,
we obtain from Lemmas \ref{lem4.2} and \ref{lem4.3} that
\begin{equation*}
\mathcal D_\Omega(|P_1w_\Omega|^2,|P_1\phi|^2)= \mathcal K_\Omega(v_\Omega^2,h^2)=\f12\int_{\R}v_\Omega^2h^2dx_3+o(1)\|h\|_{H^1(\R)}^2 \quad \text{as }\quad \Omega\to+\infty .
\end{equation*}
Since
$v_\Omega\to\rho_0$ in $L^\infty(\R)$  as $\Omega\to+\infty$ in view of \eqref{rho} and \eqref{0rd}, it then follows that
\begin{equation}\label{0ldm}
\mathcal D_\Omega\big(|P_1w_\Omega|^2,|P_1\phi|^2\big) = \f12\int_{\R}\rho_0^2h^2dx_3
+o(1)\|h\|_{H^1(\R)}^2 \quad \text{as}\quad  \Omega\to+\infty.
\end{equation}
We next pass from $|P_1w_\Omega|^2$ to $w_\Omega^2$. Since
$ w_\Omega=P_1w_\Omega+Q_1w_\Omega$,
one can check that
\begin{equation*}
w_\Omega^2-|P_1w_\Omega|^2 =2 P_1w_\Omega  Q_1w_\Omega +|Q_1w_\Omega|^2 .
\end{equation*}
%Note that the terms containing at least one factor $Q_1w_\Omega$ are negligible.
Applying Lemma \ref{lem4.3}, we further calculate from \eqref{4.26}, \eqref{4.28} and \eqref{PPPQ} that
\begin{align}
\left|
\mathcal D_\Omega\big(w_\Omega^2-|P_1w_\Omega|^2,|P_1\phi|^2\big)
\right|\leq &\sqrt{\mathcal D_\Omega\big(w_\Omega^2-|P_1w_\Omega|^2, w_\Omega^2-|P_1w_\Omega|^2\big)}\cdot
\sqrt{\mathcal D_\Omega\big(|P_1\phi|^2, |P_1\phi|^2\big)}\nonumber\\
=&o(1)\|h\|_{H^1(\R)}^2 \quad \text{as}\quad  \Omega\to+\infty.\nonumber
\end{align}
Therefore, it then follows from \eqref{0ldm} that for sufficiently large $\Omega>0$,
\begin{equation}\label{0nl1}
\begin{aligned}
\int_{\R^3}U_\Omega |P_1\phi|^2dx
=\mathcal D_\Omega\big(w_\Omega^2,|P_1\phi|^2\big) =\f12\int_{\R}\rho_0^2h^2dx_3
+o(1)\|h\|_{H^1(\R)}^2.
%\quad \text{as}\quad  \Omega\to+\infty.
\end{aligned}
\end{equation}

We next estimate $\mathcal D_\Omega\big(w_\Omega P_1\phi,w_\Omega P_1\phi\big)$. Decompose
$w_\Omega P_1\phi= P_1w_\Omega P_1\phi+ Q_1w_\Omega P_1\phi$.
Since
 $ \|v_\Omega h\|_{L^1(\R)\cap H^1(\R)}\leq C\|h\|_{H^1(\R)}$ as $\Omega\to+\infty$,
it yields from Lemmas \ref{lem4.2} and \ref{lem4.3} that
\begin{equation*}
\begin{aligned}
 &\mathcal D_\Omega\big(P_1w_\Omega P_1\phi,P_1w_\Omega P_1\phi\big)=
\mathcal K_\Omega(v_\Omega h,v_\Omega h)\\
=&\f12\int_{\R}v_\Omega^2h^2dx_3+o(1)\|h\|_{H^1(\R)}^2=\f12\int_{\R}\rho_0^2h^2dx_3
+o(1)\|h\|_{H^1(\R)}^2 \quad \text{as}\quad  \Omega\to+\infty.
\end{aligned}
\end{equation*}
%Moreover, one can deduce from Theorem \ref{thm1.2} and \eqref{rho} that
%\begin{equation}\label{0qls}
%\|Q_1w_\Omega\|_{L^\infty(\R^3)}\leq \|w_\Omega-W_0\|_{L^\infty(\R^3)}+\|P_1w_\Omega-W_0\|_{L^\infty(\R^3)}\to0 \quad \text{as}\quad  \Omega\to+\infty .
%\end{equation}
Applying the Hardy--Littlewood--Sobolev
inequality after the change of variables
$
X_\perp=\Omega^{-1/4}(\ln\Omega)x_\perp$,
$X_3=x_3$,
we have
$$
|\mathcal D_\Omega(f,g)|\leq C\frac{\Omega^{1/6}}{(\ln\Omega)^{5/3}}\|f\|_{L^{6/5}(\R^3)}\|g\|_{L^{6/5}(\R^3)}.
$$
Since $P_1\phi=\psi_1h$, we deduce from the embedding
$H^1(\R)\hookrightarrow L^3(\R)$ that
$$
\|Q_1w_\Omega P_1\phi\|_{L^{6/5}(\R^3)}\leq\|Q_1w_\Omega\|_2\|\psi_1h\|_{L^3(\R^3)}
\leq C\|Q_1w_\Omega\|_2\|h\|_{H^1(\R)}.
$$
Consequently, we obtain from \eqref{4.4} that
\begin{align}
\left|
\mathcal D_\Omega\bigl(Q_1w_\Omega P_1\phi,
Q_1w_\Omega P_1\phi\bigr)
\right|
\leq&C\frac{\Omega^{1/6}}{(\ln\Omega)^{5/3}}\|Q_1w_\Omega\|_2^2\|h\|_{H^1(\R)}^2\\
\leq &\frac{C}{\Omega^{1/3}(\ln\Omega)^{2/3}}\|h\|_{H^1(\R)}^2
=o(1)\|h\|_{H^1(\R)}^2\quad\text{as}\quad  \Omega\to+\infty.\nonumber
\end{align}
For the mixed term, we have
\begin{equation*}
\begin{aligned}
\Big|\mathcal D_\Omega\big(P_1w_\Omega P_1\phi, Q_1w_\Omega P_1\phi\big)
\Big|\leq\sqrt{\mathcal D_\Omega\big( P_1w_\Omega P_1\phi, P_1w_\Omega P_1\phi\big)} \cdot \sqrt{\mathcal D_\Omega\big( Q_1w_\Omega P_1\phi,  Q_1w_\Omega P_1\phi\big)} =
o(1)\|h\|_{H^1(\R)}^2
\end{aligned}
\end{equation*}
as $\Omega\to+\infty$. Hence
\begin{equation}\label{0nl2}
\mathcal D_\Omega\big(w_\Omega P_1\phi,w_\Omega P_1\phi\big)=\f12\int_{\R}\rho_0^2h^2dx_3+o(1)\|h\|_{H^1(\R)}^2 \quad\text{as}\quad  \Omega\to+\infty.
\end{equation}

For the mixed terms, using Cauchy's inequality, we have
\begin{equation*}
\begin{aligned}
\left|\int_{\R^3}U_\Omega (P_1\phi)\eta dx\right|
&\leq\left(\int_{\R^3}U_\Omega |P_1\phi|^2dx\right)^{\f12}\left(\int_{\R^3}U_\Omega \eta ^2dx\right)^{\f12},
\end{aligned}
\end{equation*}
and
\begin{equation*}
\begin{aligned}
\Big|\mathcal D_\Omega\big(w_\Omega P_1\phi,w_\Omega\eta \big)\Big|
&\leq\sqrt{\mathcal D_\Omega\big(w_\Omega P_1\phi,w_\Omega P_1\phi\big)}\cdot
\sqrt{\mathcal D_\Omega\big(w_\Omega\eta ,w_\Omega\eta \big)}.
\end{aligned}
\end{equation*}
The first factors in the two estimates above are bounded by
$C\|h\|_{H^1(\R)}$ in view of  \eqref{0nl1} and
\eqref{0nl2}, while the second factors are controlled by
\eqref{0fbu}. Applying Young's inequality, we obtain
\eqref{0nm}. The proof of Lemma \ref{0lnc} is therefore complete. \qed

 \vskip 0.05truein

We now pass to the rescaled variational problem. For $v\in\H$ with $\|v\|_2=1$, define
\begin{equation}\label{0Id}
\begin{aligned}
\mathcal I_\Omega(v):=\frac12\int_{\R^3}|\partial_{x_3}v|^2dx+\frac{a_\Omega}{2}
\int_{\R^3}\Big(|\nabla_\perp v|^2+|x_\perp|^2v^2-2v^2\Big)dx -\frac14\mathcal D_\Omega(v^2,v^2),
\end{aligned}
\end{equation}
where $a_\Omega$ and $\mathcal D_\Omega$ are defined by \eqref{5.3} and \eqref{0Dd}, respectively. Define
$$u_v(x_\perp,x_3)
=\Omega^{1/4}(\ln\Omega)^{1/2}v\big(\Omega^{1/4}x_\perp,(\ln\Omega)x_3\big),$$
then $\|u_v\|_2=\|v\|_2$, and a direct change of variables yields that
$\mathcal E_\Omega(u_v)=2\sqrt{\Omega}+2(\ln\Omega)^2\mathcal I_\Omega(v)$.
Hence the original minimization problem \eqref{inf} is equivalent, after this scaling, to the minimization of $\mathcal I_\Omega$ under the constraint $\|v\|_2=1$. In particular, the function
$w_\Omega$ defined by \eqref{scaling} is a minimizer of \eqref{0Id} under the constraint $\|v\|_2=1$.

We next introduce the linearized quadratic form at a fixed minimizer
$w_\Omega$. For a fixed minimizer $w_\Omega$, define the linearized operator
$\mathcal L_{\Omega,w_\Omega}$ by the quadratic form
\begin{equation} \label{Q}
\begin{aligned}
\left\langle \mathcal L_{\Omega,w_\Omega}\phi,\varphi\right\rangle
:=&\int_{\R^3} \partial_{x_3}\phi \partial_{x_3} \varphi dx
+a_\Omega\int_{\R^3}\Big( \nabla_\perp\phi  \cdot \nabla_\perp \varphi+|x_\perp|^2\phi \varphi-2\phi \varphi\Big)dx \\
&-\int_{\R^3}U_\Omega\phi \varphi dx-2\mathcal D_\Omega(w_\Omega\phi,w_\Omega\varphi)
-\nu_\Omega\int_{\R^3}\phi \varphi dx, ~ \phi, \varphi\in\H,
\end{aligned}
\end{equation}
where $U_\Omega$, $\nu_\Omega$ and $\mathcal D_\Omega$ are given by \eqref{5.1},
\eqref{5.3} and \eqref{0Dd}, respectively. For simplicity, denote $\mathcal Q_\Omega(\phi)
:=\left\langle \mathcal L_{\Omega,w_\Omega}\phi, \phi\right\rangle$.
%Equivalently, this quadratic form is the second variation of the constrained
%Lagrangian at $w_\Omega$:
%\begin{equation}\label{Q}
%\mathcal Q_\Omega(\phi):=
%\left\langle \mathcal L_{\Omega,w_\Omega}\phi,\phi\right\rangle=
%\left.\frac{d^2}{dt^2}\right|_{t=0}\left[
     %   \mathcal I_\Omega(w_\Omega+t\phi)
      %  -\frac{\nu_\Omega}{2}\|w_\Omega+t\phi\|_2^2
%\right].\end{equation}
The following proposition is the main coercivity estimate for
$\mathcal Q_\Omega$ in the large-$\Omega$ regime.
%We shall use the following coercivity estimate, which is obtained from the coercivity of the limiting one-dimensional linearized operator and the decomposition $v=P_1v+Q_1v$.

\begin{Proposition}\label{0prlo}
There exist constants $\Omega_1>0$ and $c_0>0$ such that, for every
$\Omega>\Omega_1$,
\begin{equation}\label{0lc}
 \mathcal Q_\Omega(\phi)\geq c_0\|\phi\|_2^2
\end{equation}
for all real-valued $\phi\in\H$ satisfying
\begin{equation}\label{c2}
\phi(x_\perp,x_3)=\phi(x_\perp,-x_3),\qquad\langle \phi,w_\Omega\rangle=0 .
\end{equation}
%More precisely, for sufficiently large $\Omega>0$,
%\begin{equation}\label{0lcs}
 %       \mathcal Q_\Omega(\phi)      \geq
      %  c_0\|P_1\phi\|_{H^1(\R^3)}^2
% +  c_0a_\Omega\|Q_1\phi\|_2^2,
%\end{equation}
%where $P_1$, $Q_1$ and $a_\Omega$ are as in \eqref{P}, \eqref{QO} and \eqref{5.3}, respectively.
\end{Proposition}

\noindent{\bf Proof.}
Decompose $\phi$ as in \eqref{c5}.
Since $\phi$ is even with respect to $x_3$, the function $h$ is even. We divide the proof into three steps.

 \vspace{.1cm}
 {\it  Step 1. Reduction of the lowest mode.}
By \eqref{0nl1}, \eqref{0nl2} and the definition of
$\mathcal Q_\Omega$, we have
\begin{equation*}
\begin{aligned}
\mathcal Q_\Omega(P_1\phi)
=&\int_{\R}|h'|^2dx_3-\nu_\Omega\int_{\R}h^2dx_3-\frac32\int_{\R}\rho_0^2h^2dx_3
+o(1)\|h\|_{H^1(\R)}^2 \\
=&q_{1d}(h)+(\lambda_{1d}-\nu_\Omega)\int_{\R}h^2dx_3
+o(1)\|h\|_{H^1(\R)}^2,
\end{aligned}
\end{equation*}
where $q_{1d}(h)$ is defined by \eqref{q1d}.
We first prove that
\begin{equation}\label{0ntl}
\nu_\Omega\to\lambda_{1d}\quad\text{as }\ \ \Omega\to+\infty .
\end{equation}
Indeed, testing \eqref{4} by
$w_\Omega$ and using $\|w_\Omega\|_2=1$, we get that
\begin{equation*}
\begin{aligned}
\nu_\Omega
=&\int_{\R^3}|\partial_{x_3}w_\Omega|^2dx+a_\Omega\int_{\R^3}
\Big(|\nabla_\perp w_\Omega|^2+|x_\perp|^2w_\Omega^2-2w_\Omega^2\Big)dx
-\int_{\R^3}U_\Omega w_\Omega^2dx .
\end{aligned}
\end{equation*}
According to Theorem \ref{thm1.2}, we obtain from \eqref{4.4}, \eqref{vv}, \eqref{4.29} and \eqref{5.3} that
\begin{equation*}
\nu_\Omega\to\int_{\R}|\rho_0'|^2dx_3-\frac12\int_{\R}\rho_0^4dx_3 \quad\text{as }\ \ \Omega\to+\infty.
\end{equation*}
On the other hand, multiplying \eqref{1rho}
by $\rho_0$ and using $\|\rho_0\|_2=1$, we obtain that
\begin{equation*}
\lambda_{1d} =\int_{\R}|\rho_0'|^2dx_3 -\frac12\int_{\R}\rho_0^4dx_3 .
\end{equation*}
Therefore, \eqref{0ntl} holds. Thus,
\begin{equation}\label{0lmec}
\mathcal Q_\Omega(P_1\phi) =q_{1d}(h)+o(1)\|h\|_{H^1(\R)}^2  \quad\text{as }\ \ \Omega\to+\infty.
\end{equation}

Next we prove that $h$ is almost orthogonal to $\rho_0$.
%Since
%\begin{equation*}
%\langle \phi,w_\Omega\rangle_{L^2(\R^3)}=0,\qquad \phi=\psi_1h+\eta, \qquad P_1w_\Omega=\psi_1v_\Omega ,
%\end{equation*}
%and since $P_1$ and $Q_1$ are orthogonal, we have
Note from \eqref{c1}, \eqref{c5} and \eqref{c2} that
\begin{equation*}
0 = \langle h,v_\Omega\rangle + \langle \eta,Q_1w_\Omega\rangle  .
\end{equation*}
Using $v_\Omega\to\rho_0$ in $L^2(\R)$ and
$\|Q_1w_\Omega\|_2\to0$ as $ \Omega\to+\infty$, we then obtain that for sufficiently large $\Omega>0$,
\begin{equation}\label{0aoc}
\begin{aligned}
|\langle h,\rho_0\rangle |\leq|\langle h,v_\Omega-\rho_0\rangle |
+|\langle \eta,Q_1w_\Omega\rangle |\leq
o(1)\|h\|_2+o(1)\|\eta\|_2. % \quad\text{as }\ \ \Omega\to+\infty.
\end{aligned}
\end{equation}
Set $\beta:= \langle h,\rho_0\rangle ,\qquad h_\perp:=h-\beta\rho_0$,so that $\langle h_\perp,\rho_0\rangle_{L^2(\R)}=0$ and $h_\perp$ is even. Thus, it follows from \eqref{0qce} that $q_{1d}(h_\perp)\geq c_1\|h_\perp\|_{H^1(\R)}^2$. Let $B_{1d}$ be the bilinear form associated with $q_{1d}$. Since
$\rho_0\in H^1(\R)\cap L^\infty(\R)$, we have
\begin{equation}\label{2}
|B_{1d}(h_\perp,\rho_0)|\leq C\|h_\perp\|_{H^1(\R)},\qquad|q_{1d}(\rho_0)|\leq C .
\end{equation}
Therefore, it follows from \eqref{0aoc}--\eqref{2} that
\begin{equation}\label{0qhlc}
\begin{aligned}
q_{1d}(h)&=q_{1d}(h_\perp+\beta\rho_0)  =q_{1d}(h_\perp)+2\beta B_{1d}(h_\perp,\rho_0)
+\beta^2q_{1d}(\rho_0)  \\
&\geq\frac{3c_1}{4}\|h_\perp\|_{H^1(\R)}^2
-C\beta^2 \geq \frac{c_1}{2}\|h\|_{H^1(\R)}^2-o(1)\|\eta\|_2^2 \quad\text{as }\ \ \Omega\to+\infty.
\end{aligned}
\end{equation}
Combining \eqref{0lmec} and
\eqref{0qhlc}, we obtain that
\begin{equation}\label{0lfc}
\mathcal Q_\Omega(P_1\phi) \geq\frac{c_1}{3}\|h\|_{H^1(\R)}^2-o(1)\|\eta\|_2^2 \quad\text{as }\ \ \Omega\to+\infty.
\end{equation}

\vspace{.1cm}
{\it Step 2. Estimate of the high-mode component.}
By the spectral gap of $A_1$ on $\operatorname{Ran} Q_1$, where $A_1$ is defined in \eqref{A}, we have
\begin{equation}\label{0sgc}
\int_{\R^3}\Big(|\nabla_\perp\eta|^2+|x_\perp|^2\eta^2-2\eta^2\Big)dx
\geq2\|\eta\|_2^2 .
\end{equation}
It then follows from \eqref{0nh}, \eqref{Q} and \eqref{0sgc} that
\begin{equation}\label{0hfc}
\begin{aligned}
\mathcal Q_\Omega(\eta)
&\geq\frac12\int_{\R^3}|\partial_{x_3}\eta|^2dx+\frac{a_\Omega}{2}
\int_{\R^3}\Big(|\nabla_\perp\eta|^2+|x_\perp|^2\eta^2-2\eta^2\Big)dx-C\|\eta\|_2^2  \\
&\geq(a_\Omega-C)\|\eta\|_2^2 .
\end{aligned}
\end{equation}
Since $a_\Omega=\frac{\sqrt{\Omega}}{(\ln\Omega)^2}\to+\infty$ as $\Omega\to+\infty$, the
high-mode component is positive for all sufficiently large $\Omega>0$.

\vspace{.1cm}
{\it  Step 3. Estimate of the mixed terms.}
We first compute the mixed term between $P_1\phi$ and
$\eta=Q_1\phi$. Since $P_1\partial_{x_3}\eta=\partial_{x_3}P_1\eta=0$, we have $\int_{\R^3}\partial_{x_3}(P_1\phi) \partial_{x_3}\eta dx=0$.
Moreover, one can check that
\begin{equation*}
\int_{\R^3}
\Big[
\nabla_\perp(P_1\phi)\cdot\nabla_\perp\eta
+
|x_\perp|^2(P_1\phi)\eta
-
2(P_1\phi)\eta
\Big]dx=0.
\end{equation*}
It then follows from \eqref{Q} that
\begin{equation*}
\mathcal Q_\Omega(P_1\phi,\eta)=-\int_{\R^3}U_\Omega \psi_1h \eta dx-2\mathcal D_\Omega(w_\Omega\psi_1h,w_\Omega\eta),
\end{equation*}
where $\mathcal D_\Omega$ is as in \eqref{0Dd} and $\int_{\R^3}(P_1\phi)\eta dx=0$ is used.
We further obtain from \eqref{0nm} that for every
$\varepsilon>0$,
\begin{align}
\left|\mathcal Q_\Omega(P_1\phi,\eta)\right|\leq&\varepsilon\|h\|_{H^1(\R)}^2+
\varepsilon\int_{\R^3}|\partial_{x_3}\eta|^2dx  \label{0mfc}\\
&+\varepsilon a_\Omega\int_{\R^3}\Big(|\nabla_\perp\eta|^2+|x_\perp|^2\eta^2
-2\eta^2\Big)dx+C_{\varepsilon }\|\eta\|_2^2 \quad\text{as }\ \ \Omega\to+\infty.\nonumber
\end{align}
Taking $\varepsilon=c_1/12$ in \eqref{0mfc}, and using
$ \mathcal Q_\Omega(\phi)=\mathcal Q_\Omega(P_1\phi)+2\mathcal Q_\Omega(P_1\phi,\eta)
+\mathcal Q_\Omega(\eta)$,
we obtain from \eqref{0lfc}, \eqref{0hfc} and
\eqref{0mfc} that for sufficiently large $\Omega>0$,
\begin{equation*}
\mathcal Q_\Omega(\phi)\geq\frac{c_1}{6}\|h\|_{H^1(\R)}^2+\Big(\f23 a_\Omega-C \Big)\|\eta\|_2^2 .
\end{equation*}
Note that $\f23 a_\Omega-C\geq \frac12a_\Omega$ for sufficiently large $\Omega>0$.
Therefore, for sufficiently large $\Omega>0$ and some constant $c_0>0$,
\begin{equation*}
\mathcal Q_\Omega(\phi)\geq c\|h\|_{H^1(\R)}^2+ca_\Omega\|\eta\|_2^2 \geq c_0\|P_1\phi\|_{H^1(\R^3)}^2+c_0a_\Omega\|Q_1\phi\|_2^2\geq c_0\|\phi\|_2^2.
\end{equation*}
%Since $P_1\phi=\psi_1h$ and $\|\psi_1\|_{L^2(\R^2)}=1$, the norm $\|P_1\phi\|_{H^1(\R^3)}$ is equivalent to $\|h\|_{H^1(\R)}$.
This proves \eqref{0lc}, and the proof of Proposition \ref{0prlo} is therefore complete.
\qed

%\begin{Remark}
%The evenness assumption is used to remove the translation direction.
%Indeed, since $\rho_0$ is even, $\rho_0'$ is odd. Hence, if $\phi$ is even in $x_3$, then the corresponding function $h$ is even and therefore
%\begin{equation*}
 %       \langle h,\rho_0'\rangle_{L^2(\R)}=0 .
%\end{equation*}
%Thus, the zero mode $\rho_0'$ of the one-dimensional linearized operator is automatically excluded in the even subspace. Without imposing evenness, one has to impose the
%additional orthogonality condition:
%\begin{equation*}
%        \langle \phi,\partial_{x_3}w_\Omega\rangle_{L^2(\R^3)}=0 ,
%\end{equation*}
%which, in the limit $\Omega\to+\infty$, corresponds to
%$\langle h, \rho_0'\rangle_{L^2(\R)}=0$. Consequently, the
%coercivity statement in the full space should be formulated under the two orthogonality conditions:
%\begin{equation*}
%        \langle \phi,w_\Omega\rangle_{L^2(\R^3)}=0
 %       \quad  \text{and} \quad
  %      \langle \phi, \partial_{x_3}w_\Omega\rangle_{L^2(\R^3)}=0 .
%\end{equation*}
%In the even subspace, the second condition is automatic, provided that $w_\Omega$ is chosen to be even in $x_3$.
%\end{Remark}

\vskip 0.05truein

We are now ready to prove Theorem \ref{large-uniqueness}. The idea is to use the
coercivity of the linearized quadratic form to rule out two distinct
minimizers after fixing the translation in the $x_3$-direction.

\vskip 0.05truein

\noindent{\bf Proof of Theorem \ref{large-uniqueness}.}
In view of Proposition \ref{lem2.3}, it suffices to
prove uniqueness among positive even minimizers.
We argue by contradiction. Suppose that there exist a sequence
$\Omega_n\to+\infty$ and two different positive even minimizers
$u_{\Omega_n}$ and $\widetilde u_{\Omega_n}$ of \eqref{inf}. %such that both are even with respect to the third component.
For simplicity, we omit the
subscript $n$ and write $\Omega$ instead of $\Omega_n$.

Let $w_\Omega$ and $\widetilde w_\Omega$ be the corresponding rescaled
functions defined by \eqref{scaling}. Then both $w_\Omega$ and
$\widetilde w_\Omega$ are minimizers of $\mathcal I_\Omega(v)$ under the
constraint $\|v\|_2=1$. Moreover, by Theorem \ref{thm1.2}, we have
\begin{equation*}
w_\Omega\to W_0,\quad
\widetilde w_\Omega\to W_0
\quad\text{strongly in}\quad H^1(\R^3)\cap L^\infty(\R^3)\quad\text{as}\quad \Omega\to+\infty.
\end{equation*}
In particular,
\begin{equation*}
\|w_\Omega-\widetilde w_\Omega\|_{H^1(\R^3)}+\|w_\Omega-\widetilde w_\Omega\|_{L^\infty(\R^3)}\to0 \quad \text{as}\quad \Omega\to+\infty.
\end{equation*}
Set
$c_\Omega:=\langle \widetilde w_\Omega,w_\Omega\rangle$.
Since $\|w_\Omega\|_2=\|\widetilde w_\Omega\|_2=1$ and
$\widetilde w_\Omega-w_\Omega\to 0$ in $L^2(\R^3)$ as $\Omega\to+\infty$, we have
$c_\Omega\to1$ as $\Omega\to+\infty$. Hence $c_\Omega>0$ for all sufficiently large $\Omega>0$.
Define an even function $R_\Omega:=\widetilde w_\Omega-c_\Omega w_\Omega$. Then
\begin{equation}\label{0Ro}
\langle R_\Omega,w_\Omega\rangle =0 .
\end{equation}
Moreover, since $\|w_\Omega\|_2=\|\widetilde w_\Omega\|_2=1$, we have
$\|R_\Omega\|_2^2=1-c_\Omega^2=:\delta_\Omega^2 >0$.
%Since $w_\Omega\neq \widetilde w_\Omega$, we have $\delta_\Omega>0$.
%Indeed, if $\delta_\Omega=0$, then $R_\Omega=0$, and hence
%$\widetilde w_\Omega=c_\Omega w_\Omega$. Since $c_\Omega>0$ and both
%functions have $L^2$-norm equal to $1$, this gives $c_\Omega=1$, and therefore $\widetilde w_\Omega=w_\Omega$, a contradiction.

Since $c_\Omega>0$ and $\delta_\Omega\to 0$ as $\Omega\to+\infty$, we have $c_\Omega=(1-\delta_\Omega^2)^{1/2}$ and $c_\Omega-1=O(\delta_\Omega^2)$. Consequently,
\begin{equation}\label{0wd}
\widetilde w_\Omega=(1-\delta_\Omega^2)^{1/2}w_\Omega+R_\Omega.
\end{equation}
Since $w_\Omega$ and $\widetilde w_\Omega$ are both minimizers of
$\mathcal I_\Omega$ under the constraint $\|v\|_2=1$, we have
%\begin{equation}\label{0ee}
%\mathcal I_\Omega(\widetilde w_\Omega)=
%\mathcal I_\Omega(w_\Omega).
%\end{equation}
%Since the two functions also have the same $L^2$-norm, \eqref{0ee}
%implies that
\begin{equation}\label{0mee}
\mathcal J_\Omega(\widetilde w_\Omega)=\mathcal J_\Omega(w_\Omega),
\end{equation}
where
\begin{equation*}
\mathcal J_\Omega(v):=\mathcal I_\Omega(v)-\frac{\nu_\Omega}{2}\|v\|_2^2 .
\end{equation*}
Set
\begin{equation}\label{0xd}
\xi_\Omega:=\widetilde w_\Omega-w_\Omega .
\end{equation}
We now expand $\mathcal J_\Omega$ at $w_\Omega$. By the explicit expansion
of the quartic Hartree term, we have
\begin{equation*}
\begin{aligned}
\mathcal J_\Omega(\widetilde w_\Omega)
=\mathcal J_\Omega(w_\Omega+\xi_\Omega)=\mathcal J_\Omega(w_\Omega)+\mathcal J_\Omega'(w_\Omega)[\xi_\Omega]+\frac12\left\langle\mathcal L_{\Omega,w_\Omega}\xi_\Omega,\xi_\Omega\right\rangle -\mathcal D_\Omega(w_\Omega\xi_\Omega,\xi_\Omega^2)-\frac14\mathcal D_\Omega(\xi_\Omega^2,\xi_\Omega^2),
\end{aligned}
\end{equation*}
where $\mathcal D_\Omega$ is as in \eqref{0Dd}.

According to \eqref{4},
we have
$\mathcal J_\Omega'(w_\Omega)=0$.
It remains to estimate the
cubic and quartic terms. We first prove that
\begin{equation}\label{0qs}
\mathcal D_\Omega(\xi_\Omega^2,\xi_\Omega^2)=o(1)\|\xi_\Omega\|_2^2\quad\text{as}\quad\Omega\to+\infty .
\end{equation}
Applying \eqref{HH}
to $\Psi(x,y):=\xi_\Omega(x)\xi_\Omega(y)$,
we obtain that
\begin{equation*}
\begin{aligned}
\mathcal D_\Omega(\xi_\Omega^2,\xi_\Omega^2)
=\iint_{\R^3\times\R^3}V_\Omega(x-y)\xi_\Omega^2(x)\xi_\Omega^2(y) dxdy \leq2\varepsilon\left\langle H_x\xi_\Omega,\xi_\Omega\right\rangle
\|\xi_\Omega\|_2^2+C_\varepsilon\|\xi_\Omega\|_2^4 .
\end{aligned}
\end{equation*}
Moreover, it follows from \eqref{ww} and \eqref{0xd} that
\begin{equation*}
\big|\left\langle H_x\xi_\Omega,\xi_\Omega\right\rangle\big| \leq 2\left\langle H_x w_\Omega,w_\Omega\right\rangle+2\left\langle H_x\widetilde w_\Omega,\widetilde w_\Omega\right\rangle\leq C .
\end{equation*}
Therefore
\begin{equation}\label{d}
\mathcal D_\Omega(\xi_\Omega^2,\xi_\Omega^2)\leq C\varepsilon\|\xi_\Omega\|_2^2 +
C_\varepsilon\|\xi_\Omega\|_2^4 .
\end{equation}
Note that
$\xi_\Omega\to0$ in
$H^1(\R^3)\cap L^\infty(\R^3)$ as $\Omega\to+\infty$.
Dividing \eqref{d} by $\|\xi_\Omega\|_2^2$, taking the upper limit as $\Omega\to+\infty$, and then
letting $\varepsilon\to0$, we obtain that \eqref{0qs} is true.

We next estimate the cubic term. Since $|w_\Omega\xi_\Omega| \leq\alpha w_\Omega^2+\frac{1}{4\alpha}\xi_\Omega^2$  for any $\alpha>0$, we obtain from \eqref{U} and \eqref{0qs} that
\begin{align}
\left|\mathcal D_\Omega(w_\Omega\xi_\Omega,\xi_\Omega^2)\right|\leq&
\mathcal D_\Omega(|w_\Omega\xi_\Omega|,\xi_\Omega^2) \leq
\alpha\mathcal D_\Omega(w_\Omega^2,\xi_\Omega^2)+\frac{1}{4\alpha}\mathcal D_\Omega(\xi_\Omega^2,\xi_\Omega^2)\nonumber\\
=&\alpha\int_{\R^3}U_\Omega\xi_\Omega^2 dx+\frac{1}{4\alpha}\mathcal D_\Omega(\xi_\Omega^2,\xi_\Omega^2)\leq C\alpha\|\xi_\Omega\|_2^2
+\frac{1}{4\alpha}o(1)\|\xi_\Omega\|_2^2  \quad\text{as}\quad\Omega\to+\infty.\nonumber
\end{align}
For fixed $\alpha>0$, taking $\Omega\to+\infty$, and then letting
$\alpha\searrow0$, we obtain that
\begin{equation*}
\mathcal D_\Omega(w_\Omega\xi_\Omega,\xi_\Omega^2) =o(1)\|\xi_\Omega\|_2^2\quad\text{as}\quad\Omega\to+\infty.
\end{equation*}
Combining the above estimates,
%we conclude that
%\begin{equation*}
 %       \mathcal D_\Omega(w_\Omega\xi_\Omega,\xi_\Omega^2)
  %      =    o(1)\|\xi_\Omega\|_2^2
        %\quad \text{and} \quad\mathcal D_\Omega(\xi_\Omega^2,\xi_\Omega^2)
        %= o(1)\|\xi_\Omega\|_2^2 \quad\text{as}\quad\Omega\to+\infty.
%\end{equation*} Consequently,
we have
\begin{equation}\label{Taylor}
\mathcal J_\Omega(\widetilde w_\Omega)
=\mathcal J_\Omega(w_\Omega)+\frac12\left\langle\mathcal L_{\Omega,w_\Omega}\xi_\Omega,\xi_\Omega\right\rangle+o(1)\|\xi_\Omega\|_2^2  \quad \text{as}\quad \Omega\to+\infty.
\end{equation}

It follows from \eqref{0wd} that
\begin{equation}\label{c4}
\xi_\Omega=R_\Omega+ \left((1-\delta_\Omega^2)^{1/2}-1\right)w_\Omega= R_\Omega+O(\delta_\Omega^2)w_\Omega  \quad \text{as}\quad \Omega\to+\infty.
\end{equation}
Then, for sufficiently large $\Omega>0$,
\begin{equation*}
\begin{aligned}
\left\langle
\mathcal L_{\Omega,w_\Omega}\xi_\Omega,\xi_\Omega\right\rangle= \left\langle
\mathcal L_{\Omega,w_\Omega}R_\Omega,R_\Omega\right\rangle +\left\langle \mathcal L_{\Omega,w_\Omega} R_\Omega,w_\Omega\right\rangle O(\delta_\Omega^2)    +\left\langle
\mathcal L_{\Omega,w_\Omega}w_\Omega,w_\Omega
\right\rangle O(\delta_\Omega^4)  .
\end{aligned}
\end{equation*}
Using the Euler--Lagrange equation \eqref{4}, we deduce from \eqref{Q} that
$$ \big |\left\langle\mathcal L_{\Omega,w_\Omega} \varphi,w_\Omega\right\rangle\big|=2 \big |\mathcal D_\Omega (\varphi w_\Omega, w_\Omega^2)\big |\leq 2\|U_\Omega\|_{L^{\infty}(\R^3)}\|\varphi\|_{2}\|w_\Omega\|_{2} \quad \text{for any} \quad \varphi \in \H.$$
One can further deduce from
\eqref{U} that for sufficiently large $\Omega>0$,
\begin{equation*}
\big |\left\langle\mathcal L_{\Omega,w_\Omega} R_\Omega,w_\Omega\right\rangle\big| \leq C\|R_\Omega\|_2,\qquad
\big | \left\langle\mathcal L_{\Omega,w_\Omega}w_\Omega,w_\Omega\right\rangle\big| \leq C ,
\end{equation*}
where $C>0$ is independent of $\Omega$. Since
$ \|R_\Omega\|_2=\delta_\Omega\to 0$ as $\Omega\to+\infty$,
we obtain that
\begin{equation}\label{0qr}
\left\langle
\mathcal L_{\Omega,w_\Omega}\xi_\Omega,\xi_\Omega\right\rangle=
\left\langle\mathcal L_{\Omega,w_\Omega}R_\Omega,R_\Omega\right\rangle
+o(1)\|R_\Omega\|_2^2  \quad \text{as}\quad \Omega\to+\infty.
\end{equation}
Moreover,  we also
derive from \eqref{c4} that
\begin{equation}\label{f1}
\|\xi_\Omega\|_2^2 = \|R_\Omega\|_2^2+O(\delta_\Omega^3)= \|R_\Omega\|_2^2+o(1)\|R_\Omega\|_2^2 \quad \text{as}\quad \Omega\to+\infty .
\end{equation}
Combining \eqref{0mee}, \eqref{Taylor},
\eqref{0qr} and \eqref{f1}, we derive that
\begin{equation}\label{0ke}
0=\frac12\left\langle\mathcal L_{\Omega,w_\Omega}R_\Omega,R_\Omega\right\rangle
+o(1)\|R_\Omega\|_2^2  \quad \text{as}\quad \Omega\to+\infty.
\end{equation}

%On the other hand, $R_\Omega$ is real-valued and even with respect to $x_3$. Moreover, by \eqref{0Ro}, $R_\Omega$ is orthogonal to $w_\Omega$ in $L^2(\R^3)$.
Applying Proposition
\ref{0prlo}, we obtain from \eqref{0Ro} that
\begin{equation}\label{0cR}
\left\langle
\mathcal L_{\Omega,w_\Omega}R_\Omega,R_\Omega
\right\rangle
\geq
c_0\|R_\Omega\|_2^2  \quad \text{as}\quad \Omega\to+\infty.
\end{equation}
It then follows from \eqref{0ke} and \eqref{0cR} that $0\geq\left(\frac{c_0}{2}+o(1)\right)\|R_\Omega\|_2^2$ as $\Omega\to+\infty$. This is impossible because $c_0>0$ and
$\|R_\Omega\|_2=\delta_\Omega>0$ for sufficiently large $\Omega$. This contradiction proves that the positive even minimizer is unique for all sufficiently large $\Omega>0$.
%Now let $u_\Omega$ and $\widetilde u_\Omega$ be two arbitrary positive minimizers. By Proposition \ref{lem2.3}, after suitable translations in the $x_3$-direction, both can be made even with respect to $x_3$.By the first part of the proof, these two translated minimizers coincide. Hence there exists $z\in\R$ such that
%\begin{equation*}
 %       \widetilde u_\Omega(x_\perp,x_3)= u_\Omega(x_\perp,x_3-z).
%\end{equation*}
This completes the proof of Theorem \ref{large-uniqueness}.\qed
%It remains to consider complex-valued minimizers. Let $u\in\mathcal H$ be a complex-valued minimizer with $\|u\|_2=1$. Since all the terms in $\mathcal E_\Omega$, except the kinetic term, depend only on $|u|$, and since $ |\nabla |u||\leq |\nabla u| ~a.e.$ in $\mathbb R^3$,
%we have $\mathcal E_\Omega(|u|)\leq \mathcal E_\Omega(u)$. Hence, $|u|$ is also a non-negative minimizer. By the uniqueness of non-negative minimizers modulo translations, there exists $z\in\mathbb R$ such that
%\begin{equation*}       |u(x_\perp,x_3)|=u_\Omega(x_\perp,x_3-z).\end{equation*}
%Moreover, equality must hold in the above gradient inequality. Since $|u|>0$, writing $u=|u|\eta$ with $|\eta|=1$, the standard identity
%\begin{equation*}       |\nabla u|^2=|\nabla |u||^2+|u|^2|\nabla\eta|^2\end{equation*}
%implies that $\nabla\eta=0$ a.e. in $\mathbb R^3$. Thus
%$\eta=e^{i\theta}$ for some $\theta\in[0,2\pi)$, and therefore
%\begin{equation*}   u(x_\perp,x_3)=e^{i\theta}u_\Omega(x_\perp,x_3-z).\end{equation*}
%Consequently, the set of complex-valued minimizers is exactly
%\begin{equation*}   \left\{ e^{i\theta}u_\Omega(x_\perp,x_3-z): \theta\in[0,2\pi),\ z\in\mathbb R \right\},\end{equation*}

\appendix
\renewcommand{\theequation}{A.\arabic{equation}}
\setcounter{equation}{0}
\phantomsection
\section{Appendix}\label{Appendix}
%. Uniqueness of Ground States as $\Omega\searrow 0$}
\addcontentsline{toc}{section}{Appendix. Uniqueness of Ground States as $\Omega\searrow 0$}
This section is devoted to the proof of Theorem \ref{small-uniqueness} on the uniqueness of ground state $u_\Omega$ for the Pekar ground state energy \eqref{inf} as $\Omega\searrow 0$. Recall from \cite{Lieb76} that the ground state $Q$ of $E_0$ is unique up to translations in $\R^3$ and a constant phase. We fix $Q$ to be positive and radial.
Throughout this appendix,
a function  $h:\R^3\to\R$ is said to be even if
$$
h(-x)=h(x)\quad\text{for all} \quad x\in\R^3.
$$
Let $u_\Omega
>0$ be the positive even minimizer of \eqref{inf} obtained in Theorem \ref{existence}. By a standard concentration-compactness principle, one can then verify that
\begin{equation}\label{e10}
 u_\Omega(x)\to Q(x) \text{ \ strongly in \ }\H\text{ \ as \ }\Omega\searrow 0.
\end{equation}
The nondegeneracy of $Q$ states that the linearized operator $L_0$ given by
\begin{equation}\label{L0}
L_0\phi=-\Delta\phi-\mu_0\phi-\left(|x|^{-1}*Q^2\right)\phi-
2\left(|x|^{-1}*(Q\phi)\right)Q,
\end{equation}
acting on $L^2(\R^3)$ with domain $H^2(\R^3)$, has the kernel (cf. \cite{Lenzmann}) $\operatorname{Ker} L_0=\operatorname{Span}\{\partial_1Q,\partial_2Q,\partial_3Q\}$, where $\mu_0<0$ is the Lagrange multiplier associated with $Q$. For later use, we introduce the quadratic form associated with $L_0$:
\begin{equation}\label{e0}
\begin{aligned}
\mathfrak q_0(h,k):=\int_{\R^3}\nabla h\cdot\nabla k dx-\mu_0\int_{\R^3}hk dx  -\int_{\R^3}\bigl(|x|^{-1}*Q^2\bigr)hk dx-2\int_{\R^3}\bigl(|x|^{-1}*(Qh)\bigr)Qk dx ,
\end{aligned}
\end{equation}
and write $\mathfrak q_0(h):=\mathfrak q_0(h,h)$.
It follows from \cite[Proposition 5]{FLR13} that there exists a constant $c>0$ such that
\begin{equation}\label{e1}
\mathfrak q_0(h)\geq c\|h\|_{2}^2
\end{equation}
for every even function $h\in H^1(\R^3)$ satisfying $\langle Q,h\rangle_{L^2(\R^3)}=0$. According to \eqref{e0}, one can further upgrade \eqref{e1} to
\begin{equation}\label{e2}
\mathfrak q_0(h)\geq c_0\|h\|_{H^1(\R^3)}^2
\end{equation}
for some constant $c_0>0$ and every such $h$.

\vskip 0.05truein
The following lemma extends the coercivity estimate \eqref{e2} to
functions with a small component in the $Q$ direction.

\begin{Lemma}
\label{leme3}
There exist constants $\delta_0>0$ and $c_2>0$ such that, if
$h\in H^1(\R^3)$ is even and satisfies
\begin{equation}\label{e4}
|\langle Q,h\rangle|\leq \delta_0\|h\|_2,
\end{equation}
then
\begin{equation}\label{e5}
\mathfrak q_0(h)\geq c_2\|h\|_{H^1(\R^3)}^2.
\end{equation}
\end{Lemma}

\noindent{\bf Proof }
Since $\|Q\|_2=1$, write
\begin{equation*}
h=h^\perp+aQ, \qquad a:=\langle Q,h\rangle, \qquad \langle Q,h^\perp\rangle=0.
\end{equation*}
We obtain from \eqref{e2} that $\mathfrak q_0(h^\perp) \geq c_0\|h^\perp\|_{H^1(\R^3)}^2$. Moreover, one can check that $|\mathfrak q_0(h, k)|\leq C\|h\|_{H^1(\R^3)}\|k\|_{H^1(\R^3)}$ for some constant $C>0$. Thus, there exists a constant $C>0$ such that
\begin{equation}\label{D1}
\mathfrak q_0(h)=\mathfrak q_0(h^\perp)+2a\mathfrak q_0(h^\perp,Q)
+a^2\mathfrak q_0(Q)\geq\frac{c_0}{2}\|h^\perp\|_{H^1(\R^3)}^2-Ca^2.
\end{equation}
Additionally, it yields from \eqref{e4} that
\begin{equation}\label{D3}
|a|\leq\delta_0\|h\|_2 \leq\delta_0\|h\|_{H^1(\R^3)}.
\end{equation}
Note that
\begin{equation}\label{D2}
\|h^\perp\|_{H^1(\R^3)} \geq \|h\|_{H^1 (\R^3)}-|a|\|Q\|_{H^1 (\R^3)}.
\end{equation}
Thus, by choosing $\delta_0>0$ sufficiently small, we conclude from \eqref{D1}--\eqref{D2} that \eqref{e5} is true, and the proof of Lemma \ref{leme3} is complete.
\qed
\vskip 0.05truein

We next establish an estimate for the nonlinear remainder of the
Hartree term. Define $\mathcal N(u):=\bigl(|x|^{-1}*u^2\bigr)u$, so that
\begin{equation}\label{e6}
\mathcal N'(u)h =\big (|x|^{-1}*u^2\big )h +2\big (|x|^{-1}*(uh)\big )u.
\end{equation}

\vskip 0.05truein

\begin{Lemma}
\label{leme8}
There exists a constant $C>0$ such that, for all
$u,v\in H^1(\R^3)$,
\begin{align}
& | \langle\mathcal N(u)-\mathcal N(v)-\mathcal N'(Q)(u-v),u-v \rangle | \label{e7}\\
\leq&C \big(\|u\|_{H^1(\R^3)}+\|v\|_{H^1(\R^3)}+\|Q\|_{H^1(\R^3)}\big) \bigl(\|u-Q\|_{H^1(\R^3)}+\|v-Q\|_{H^1(\R^3)}\bigr)\|u-v\|_{H^1(\R^3)}^2,\nonumber
\end{align}
where $Q>0$ is the unique radial minimizer of $E_0$.
\end{Lemma}

\noindent{\bf Proof }
Using Hardy--Littlewood--Sobolev inequality, we obtain that
for any $a,b,c,d\in H^1(\R^3)$,
\begin{equation}\label{D4}
\begin{aligned}
\left|\iint_{\R^3\times\R^3}\frac{a(x)b(x)c(y)d(y)}{|x-y|} dx dy\right|
&\leq C\|ab\|_{L^{\f65}(\R^3)}\|cd\|_{L^{\f65}(\R^3)}\\
&\leq  C\|a\|_{H^1(\R^3)}\|b\|_{H^1(\R^3)} \|c\|_{H^1(\R^3)}\|d\|_{H^1(\R^3)}
\end{aligned}
\end{equation}
for some constant $C>0$.
%It then follows from \eqref{e6} and \eqref{D4} that there exists a constant $C>0$ such that
It then follows from \eqref{e6} and
\eqref{D4} that for any $\varphi\in H^1(\R^3)$,
\begin{align}
\left|\left\langle\bigl(\mathcal N'(u)-\mathcal N'(v)\bigr)h,\varphi\right\rangle
\right|\leq &C\bigl(\|u\|_{H^1(\R^3)}+\|v\|_{H^1(\R^3)}\bigr)\|u-v\|_{H^1(\R^3)}\|h\|_{H^1(\R^3)}
\|\varphi\|_{H^1(\R^3)}.\nonumber
\end{align}
%\begin{equation}\label{e9}
%\begin{aligned}
%&\big \|\bigl(\mathcal N'(u)-\mathcal N'(v)\bigr)h\big \|_{H^{-1}(\R^3)}\\
%\leq& C\Big\{ \big \| \big[|x|^{-1}*\big(u^2-v^2\big)\big]h
%\big \|_{H^{-1}(\R^3)}+\big \| \big[|x|^{-1}*\big( (u-v)h\big)\big]u
%\big \|_{H^{-1}(\R^3)} \\
%&+\big \| \big[|x|^{-1}*\big( vh\big)\big](u-v)\big \|_{H^{-1}(\R^3)}\Big\}\\
%\leq&C \big(\|u\|_{H^1(\R^3)}+\|v\|_{H^1(\R^3)}\big) \|u-v\|_{H^1(\R^3)}\|h\|_{H^1(\R^3)} .
%\end{aligned}
%\end{equation}
%with $\langle f,g\rangle_{H^{-1}(\R^3), H^1(\R^3)}=\int_{\R^3}fgdx$.
Finally, we calculate from above that
\begin{equation*}
\begin{aligned}
&\big|\big\langle \mathcal N(u)-\mathcal N(v)-\mathcal N'(Q)(u-v),
u-v \big\rangle \big|\leq\int_0^1\big|\big\langle \big (\mathcal N'(v+t(u-v))-\mathcal N'(Q)\big )(u-v) ,u-v \big\rangle \big|dt\\
\leq&C \big(\|u\|_{H^1(\R^3)}+\|v\|_{H^1(\R^3)}+\|Q\|_{H^1(\R^3)}\big)   \bigl(\|u-Q\|_{H^1(\R^3)}+\|v-Q\|_{H^1(\R^3)}\bigr)\|u-v\|_{H^1(\R^3)}^2.
\end{aligned}
\end{equation*}
The proof of Lemma \ref{leme8} is therefore complete.\qed

\vskip 0.05truein

Combining the coercivity estimate in Lemma \ref{leme3} with the
nonlinear remainder estimate in Lemma \ref{leme8}, we are now ready
to prove Theorem \ref{small-uniqueness}.

\vskip 0.05truein
\noindent{\bf Proof of Theorem \ref{small-uniqueness}.}
Let $u_\Omega$ and $\widetilde u_\Omega$ be two positive even
minimizers of \eqref{inf}, and let $\mu_\Omega$ and
$\widetilde\mu_\Omega$ be their respective Lagrange multipliers.
We derive from \eqref{e10} that $\mu_\Omega\to \mu_0$ and
$\widetilde\mu_\Omega\to \mu_0$ as $\Omega\searrow0$. Then
\begin{equation}\label{e11}
\begin{aligned}
\|u_\Omega-Q\|_{\H}+\|\widetilde u_\Omega-Q\|_{\H}+|\mu_\Omega-\mu_0|+|\widetilde\mu_\Omega-\mu_0|\to0 \quad \text{as} \quad \Omega\searrow0.
\end{aligned}
\end{equation}
Set
$\bar{w}_\Omega:=u_\Omega-\widetilde u_\Omega$.
Since both minimizers are normalized in $L^2(\R^3)$, we have
\begin{equation}\label{e12}
\begin{aligned}
2\langle Q,\bar{w}_\Omega\rangle=-\|u_\Omega-Q\|_2^2+\|\widetilde u_\Omega-Q\|_2^2=-\left\langle u_\Omega+\widetilde u_\Omega-2Q, \bar{w}_\Omega\right\rangle.
\end{aligned}
\end{equation}
Hence
\begin{equation}\label{e13}
|\langle Q,\bar{w}_\Omega\rangle |\leq\frac12\big(\|u_\Omega-Q\|_2+\|\widetilde u_\Omega-Q\|_2\big)\|\bar{w}_\Omega\|_2.
\end{equation}
Applying Lemma
\ref{leme3}, it follows from   \eqref{e11}
and \eqref{e13} that for some constant $c_2>0$,
\begin{equation}\label{e14}
\mathfrak q_0(\bar{w}_\Omega)\geq c_2\|\bar{w}_\Omega\|_{H^1(\R^3)}^2 \quad \text{as} \quad \Omega\searrow0.
\end{equation}

Note that the two minimizers satisfy
$$
-\Delta u_\Omega+\Omega  |x_\perp|^2u_\Omega-\mathcal N(u_\Omega)=\mu_\Omega u_\Omega,\ \ -\Delta\widetilde u_\Omega+\Omega  |x_\perp|^2\widetilde u_\Omega-\mathcal N(\widetilde u_\Omega)=\widetilde\mu_\Omega\widetilde u_\Omega.$$
It then follows from \eqref{L0} and \eqref{e6} that
%\begin{equation}\label{e15}
%\begin{aligned}
%L_0w_\Omega+\Omega |x_\perp|^2 \bar{w}_\Omega
%=
%\mathcal N(u_\Omega)-\mathcal N(\widetilde u_\Omega)
%-\mathcal N'(Q)\bar{w}_\Omega
%+(\mu_\Omega-\mu_0)u_\Omega-
%(\widetilde\mu_\Omega-\mu_0)\widetilde u_\Omega,
%\end{aligned}\end{equation}
%which implies that
\begin{align}
&\mathfrak q_0(\bar{w}_\Omega)
+\Omega\int_{\R^3} |x_\perp|^2|\bar{w}_\Omega|^2 dx\label{e16}\\
=&\left\langle\mathcal N(u_\Omega)-\mathcal N(\widetilde u_\Omega)-\mathcal N'(Q)\bar{w}_\Omega,\bar{w}_\Omega\right\rangle +(\mu_\Omega-\mu_0)
\langle u_\Omega, \bar{w}_\Omega\rangle -(\widetilde\mu_\Omega-\mu_0)
\langle\widetilde u_\Omega, \bar{w}_\Omega\rangle.\nonumber
\end{align}
Since $\|u_\Omega\|_2=\|\widetilde u_\Omega\|_2=1$, one can calculate that $\langle u_\Omega,\bar{w}_\Omega\rangle  =\frac12\|\bar{w}_\Omega\|_2^2$ and $\langle\widetilde u_\Omega,\bar{w}_\Omega\rangle =-\frac12\|\bar{w}_\Omega\|_2^2$. Therefore, \eqref{e16} becomes
\begin{equation*}
\begin{aligned}
\mathfrak q_0(\bar{w}_\Omega)
+\Omega\int_{\R^3}|x_\perp|^2|\bar{w}_\Omega|^2 dx
=\left\langle\mathcal N(u_\Omega)-\mathcal N(\widetilde u_\Omega)
-\mathcal N'(Q)\bar{w}_\Omega,\bar{w}_\Omega\right\rangle +\frac{(\mu_\Omega-\mu_0)
+(\widetilde\mu_\Omega-\mu_0)}{2}\|\bar{w}_\Omega\|_2^2.
\end{aligned}
\end{equation*}

By Lemma \ref{leme8}, we have
\begin{align}
&\left|\left\langle\mathcal N(u_\Omega)-\mathcal N(\widetilde u_\Omega)
-\mathcal N'(Q)\bar{w}_\Omega,\bar{w}_\Omega\right\rangle \right|\label{e18}\\
\leq&C\big(\|u_\Omega\|_{H^1(\R^3)}+\|\widetilde u_\Omega\|_{H^1(\R^3)}+\|Q\|_{H^1(\R^3)}\big) \big(\|u_\Omega-Q\|_{H^1(\R^3)}
+\|\widetilde u_\Omega-Q\|_{H^1(\R^3)}\big)\|\bar{w}_\Omega\|_{H^1(\R^3)}^2.\nonumber
\end{align}
It then follows from
\eqref{e14}--\eqref{e18} that for sufficiently small $\Om>0$,
\begin{equation*}
\begin{aligned}
c_2\|\bar{w}_\Omega\|_{H^1}^2
\leq &
C\Big[\big(\|u_\Omega\|_{H^1(\R^3)}+\|\widetilde u_\Omega\|_{H^1(\R^3)}+\|Q\|_{H^1(\R^3)}\big)\big(\|u_\Omega-Q\|_{H^1(\R^3)}
+\|\widetilde u_\Omega-Q\|_{H^1(\R^3)}  \big)\\
&+|\mu_\Omega-\mu_0|+|\widetilde\mu_\Omega-\mu_0|\Big]\|\bar{w}_\Omega\|_{H^1(\R^3)}^2.
\end{aligned}
\end{equation*}
We finally conclude from \eqref{e11}
that $\bar{w}_\Omega=0$, and the proof of Theorem \ref{small-uniqueness} is therefore complete.
\qed

\end{document}